\documentclass[10pt,reqno,final]{article}
\usepackage{cite}
\usepackage{amsmath}
\usepackage{graphicx}
\usepackage{float}
\usepackage{caption}
\usepackage{amsxtra}
\usepackage{cases}
\usepackage{float}
\usepackage{amsmath}
\usepackage{fancyhdr}
\usepackage{caption,subcaption}
\usepackage{overpic}
\usepackage{epstopdf}
\usepackage{color}
\DeclareCaptionLabelFormat{cont}{#1~#2\alph{ContinuedFloat}}
\usepackage{amsthm} 
\numberwithin{equation}{section}
\numberwithin{figure}{section}
\numberwithin{table}{section}

\theoremstyle{plain}
\newtheorem{theorem}{Theorem}[section]
\newtheorem{lemma}{Lemma}[section]

\theoremstyle{definition}

\theoremstyle{remark}
\newtheorem{remark}{Remark}[section]

\usepackage[justification=centering]{caption}
\usepackage{indentfirst}
\usepackage{enumerate}

\title{\Large
	Bifurcation analysis for a SIRS model with a nonlinear incidence rate}

\author{{ Xiaoling Wang} and { Kuilin Wu} \thanks{Corresponding author. E-mail: mathwangxiaoling@163.com (Wang), wukuilin@126.com (Wu).}
\\
{\small School of Mathematics and Statistics, Guizhou University, Guiyang 550025, PR China}
}
\date{}

\begin{document}
\captionsetup{justification=raggedright, singlelinecheck=false}
\captionsetup[figure]{labelfont={bf},name={Fig.},labelsep=period}

\maketitle
\maketitle
\begin{abstract}
In this paper, the main purpose is to explore an SIRS epidemic model with a general nonlinear incidence rate $f(I)S=\beta I(1+\upsilon I^{k-1})S$ ($k>0$). We analyzed the existence and stability of equilibria of the epidemic model. Local bifurcation theory is applied to explore the rich variety of dynamical behavior of the model. Normal forms of the epidemic model are derived for different types of bifurcation, including Bogdanov-Takens bifurcation, Nilpotent focus bifurcation and Hopf bifurcation. The first four focal values are computed to determine the codimension of the Hopf bifurcation, which can be undergo some limit cycles. Some numerical results and simulations are presented to illustrate these theoretical results.
		\smallskip
		
		\noindent
		{\bf Key words:}
		Bogdanov-Takens bifurcation; Hopf bifurcation; Nilpotent focus bifurcation
\end{abstract}
\section{Introduction}

 In mathematical epidemiology, bifurcation phenomena refer to abrupt changes in transmission dynamics or epidemic behavior influenced by environmental, host, or pathogen factors. Bifurcation phenomena in disease dynamics determine not only the rate and extent of transmission but can also alter viral pathogenicity, thereby impacting the efficacy of public health interventions. The analysis of bifurcation phenomena is essential for predicting epidemic trends, developing effective control strategies, and allocating public health resources efficiently. In recent years, accelerated globalization and ecological transformations have led to increasingly complex transmission patterns of infectious diseases, which has made bifurcation phenomena a subject of significant research in disease prevention and control, see \cite{WS,PJ,PR,ME,WP,YD}.

The incidence rate in the infectious disease model represents the probability that the susceptible individual turns into an infective individual in a unit time, which has an important impact on the dynamic behavior of the infectious disease model. Different forms of incidence may lead to rich dynamic behavior. For instance, bilinear incidence can only describe the ideal and simple propagation process. However, by introducing the characteristics of saturation and nonlinearity, the epidemic model can capture the rich dynamic behaviors in the real world caused by complex factors such as psychological behavior, resource constraints, and super transmission, such as sudden outbreaks, persistent epidemics, and periodic fluctuations. The study of epidemic models with different forms of incidence is one of the crucial research areas in biomathematics.

Let $S(t)$, $I(t)$ and $R(t)$ denote the numbers of susceptible, infective and recovered individuals at time $t$, respectively. In most classical epidemiological models, the incidence rate is defined by mass-action incidence with bilinear interactions i.e. $\beta IS$, where $\beta$ represents the probability of transmission per contact. Classical infectious disease models with bilinear incidence rates typically admit at most one endemic equilibrium, which means they cannot represent complex phenomena like bistability or periodicity, see \cite{ST}. These simple models can provide general conclusions for long-term disease dynamics. However, the inability to account for complex population behaviors constrains the understanding of disease transmission dynamics and hinders the development of effective control strategies. Some studies have focused on complex dynamical behaviors in epidemic models with nonlinear incidence rates, see \cite{MM,HP,WA,YS,WYY,LC}.

The classical susceptible-infective-recovered-susceptible (SIRS) model is given as follows, see \cite{GO},
\begin{equation}
\begin{array}{ll}\label{1.1}
\left\{
\begin{aligned}
\dot{S}&=b-dS-f(I)S+\delta R,\\[2ex]
\dot{I}&=f(I)S-(d+\mu)I,\\[2ex]
\dot{R}&=\mu I-(d+\delta)R,
\end{aligned}
\right.
\end{array}
\end{equation}
where $b>0$ is the population recruitment rate, $d>0$ represents the natural death rate, $\mu>0$ represents the natural recovery rate of infective individuals, $\delta\geq0$ represents the rate at which recovered individuals lose immunity and return to the susceptible class.
$f(I)S$ denotes the incidence rate. It may depend on many factors, such as population density, social habits, and public health measures. Hence, various types of incidence rate have been used.

(I) Saturated incidence rate:
\begin{equation}
\begin{array}{ll}\label{0.1}
\begin{aligned}
f(I)S=\frac{\beta IS}{1+\alpha I},
\end{aligned}
\end{array}
\end{equation}
where $\beta I$ measures the infection force of the disease and $\frac{1}{1+\alpha I}$ describes the ''psychological'' effects, i.e., inhibition effect. Capasso and Serio \cite{GV} introduced the saturated incidence rate \eqref{0.1} and they extended the threshold theorem by the stability analysis of the system \eqref{1.1} equilibrium.

(II) Non-monotonic incidence rate:

Xiao and Ruan \cite{SD} investigated system \eqref{1.1} with a nonmonotone incidence rate
\begin{equation}
\begin{array}{ll}\label{0.2}
\begin{aligned}
f(I)S=\frac{\beta IS}{1+\alpha I^2}.
\end{aligned}
\end{array}
\end{equation}
They found that either the number of infective individuals tends to zero as time evolves or the disease persists by carrying out a global analysis of the model and studying the stability of the disease-free equilibrium and the endemic equilibrium.

Zhou et al. \cite{YD} proposed system \eqref{1.1} with non-monotonic incidence rate
\begin{equation}
\begin{array}{ll}\label{0.4}
\begin{aligned}
f(I)S=\frac{kIS}{1+\beta I+\alpha I^2}.
\end{aligned}
\end{array}
\end{equation}
They shown that system \eqref{1.1} with \eqref{0.4} undergoes cusp type Bogdanov-Takens bifurcation of codimension 2 and supercritical Hopf bifurcation. Xiao and Zhou \cite{DY} shown that a bistable occurs and a periodic oscillation appears.

(III) Nonlinear incidence rate:
\begin{equation}
\begin{array}{ll}\label{0.3}
\begin{aligned}
f(I)S=\frac{\beta I^2S}{1+\alpha I^2}.
\end{aligned}
\end{array}
\end{equation}
Ruan and Wang \cite{WS} shown that system \eqref{1.1} with \eqref{0.3} admits a saddle-node bifurcation, a cusp type Bogdanov-Takens bifurcation of codimension 2 and two limit cycles emerged from Hopf bifurcation. Tang et al. \cite{YDS} detected that system \eqref{1.1} with incidence rate \eqref{0.3} can admit at most two limit cycles arised from the Hopf bifurcation and undergoes homoclinic bifurcation.

(IV) General saturated incidence rate:
\begin{equation}
\begin{array}{ll}\label{0.5}
\begin{aligned}
f(I)S=\frac{\beta I^pS}{1+\alpha I^q}.
\end{aligned}
\end{array}
\end{equation}
Zhang et al. \cite{WF} investigated system \eqref{1.1} with \eqref{0.5} can undergo saddle-node bifurcation, cusp type Bogdanov-Takens bifurcation of codimension two and two limit cycles emerged by Hopf bifurcation.Cui and Zhao \cite{WYY} detected that system \eqref{1.1} with \eqref{0.5} can undergo  saddle-node bifurcation of codimension two, Bogdanov-Takens bifurcation of codimension two for general parameters $p$ and $q$. Hu et. al \cite{PZ} investigated the dynamics of system \eqref{1.1} with \eqref{0.5} and derived multiple types of bifurcations, such as supercritical Hopf bifurcation, subcritical Hopf bifurcation and cusp type Bogdanov-Takens bifurcation of codimension 2.

(V) Generalized nonlinear incidence:
\begin{equation}
\begin{array}{ll}\label{0.6}
\begin{aligned}
f(I)S=\beta\big(1+f(I,\nu)\big)\frac{IS^p}{N}.
\end{aligned}
\end{array}
\end{equation}
Alexander and Moghadas \cite{MM} found that system \eqref{1.1} with \eqref{0.6} undergoes a subcritical Hopf bifurcation, saddle-node bifurcation, Homoclinic bifurcations and two concentric limit cycles can coexist.

Some people consider an incidence rate of the form
\begin{equation}
\begin{array}{ll}\label{0.7}
\begin{aligned}
f(I)S=\beta I(1+\upsilon I^{k-1})S,
\end{aligned}
\end{array}
\end{equation}
where $\beta>0$, $\upsilon>0$ and $k>0$. $\beta IS$ represents the new infections caused by single contacts and $\beta\upsilon I^kS$ is the new infective individuals arising from $k$ exposures. $\beta$ denotes the average number of new infections per unit time in a fully susceptible population.

Lu et al. \cite{MD} shown that system \eqref{1.1} with \eqref{0.7} admits a cusp type Bogdanov-Takens bifurcation of codimension of 3 and two limit cycles bifurcated by Hopf bifurcation when $k=1$. Jin et al. \cite{YW} investigated system \eqref{1.1} with \eqref{0.7} when $k=2$ and demonstrated the existence of backward, supercritical Hopf bifurcation, subcritical Hopf bifurcation, and cusp type Bogdanov-Takens bifurcation of codimension of 2. Furthermore, they revealed bistable steady states and established explicit conditions for the asymptotic stability of the equilibria.

In this paper, we focus on the bifurcation phenomena of system \eqref{1.1} with the incidence rate \eqref{0.7},
\begin{equation}
\begin{array}{ll}\label{1.2}
\left\{
\begin{aligned}
\dot{S}&=b-dS-\beta I(1+\upsilon I^{k-1})S+\delta R,\\[2ex]
\dot{I}&=\beta I(1+\upsilon I^{k-1})S-(d+\mu)I,\\[2ex]
\dot{R}&=\mu I-(d+\delta)R,
\end{aligned}
\right.
\end{array}
\end{equation}
where $\beta>0$, $\upsilon>0$ and $k>0$. To simplify the model, we add all equations in system \eqref{1.2} and denote the number of the total population by $N(t)$ ($N(t)=S(t)+I(t)+R(t)$). This yields the following equation:
\begin{equation*}
\begin{array}{ll}
\begin{aligned}
\dot{N}=b-dN.
\end{aligned}
\end{array}
\end{equation*}
Hence, $\lim\limits_{t\rightarrow\infty}N(t)=\frac{b}{d}\equiv\Lambda$. Therefore, the reduced system of \eqref{1.2} is as follows
\begin{equation}
\begin{array}{ll}\label{1.3}
\left\{
\begin{aligned}
\dot{I}&=\beta I(1+\upsilon I^{k-1})(\Lambda-I-R)-(d+\mu)I,\\[2ex]
\dot{R}&=\mu I-(d+\delta)R,
\end{aligned}
\right.
\end{array}
\end{equation}
and
\begin{equation}
\begin{array}{ll}\label{1.4}
\begin{aligned}
\Omega=\{(I,R)|I\geq0, R\geq0, I+R\leq\Lambda\}
\end{aligned}
\end{array}
\end{equation}
is a positive invariant set of system \eqref{1.3}.

By the change of variables $I=\frac{d+\delta}{\beta}x$, $R=\frac{d+\delta}{\beta}y$ and $t=\frac{1}{d+\delta}\tau$ (still denote $\tau$ by $t$), system \eqref{1.3} can be transformed into
\begin{equation}
\begin{array}{ll}\label{1.5}
\left\{
\begin{aligned}
\dot{x}&=x(1+px^{k-1})(\Lambda_0-x-y)-\gamma x,\\[2ex]
\dot{y}&=\eta x-y,
\end{aligned}
\right.
\end{array}
\end{equation}
where
\begin{equation*}
\begin{array}{ll}
\begin{aligned}
p=(\frac{d+\delta}{\beta})^{k-1}\upsilon,\quad\Lambda_0=\frac{\beta}{d+\delta}\Lambda,\quad\gamma=\frac{d+\mu}{d+\delta}\quad\eta=\frac{\mu}{d+\delta},
\end{aligned}
\end{array}
\end{equation*}
satisfy $p>0$, $\Lambda_0>0$ and $\gamma>\eta>0$. Note that $R_0:=\frac{\Lambda_0}{\gamma}$ represents the basic reproduction number in the epidemic model. Obviously, the positive invariant set of system \eqref{1.5} is as follows
\begin{equation*}
\begin{array}{ll}
\begin{aligned}
\bar{\Omega}=\{(x,y)|x\geq0, y\geq0, x+y\leq\Lambda_0\},
\end{aligned}
\end{array}
\end{equation*}
and the parameter space of system \eqref{1.5} is
\begin{equation*}
\begin{array}{ll}
\begin{aligned}
\Gamma=\{(p,\Lambda_0,\gamma,\eta,k)|p>0,\Lambda_0>0,\gamma>\eta>0,k>0\}.
\end{aligned}
\end{array}
\end{equation*}

The organization of this paper is as follows. The existence and types of equilibria for system \eqref{1.5} is discussed in section 2. In section 3, we analyze the degenerate equilibria. In Section 4, we show that system \eqref{1.5} undergoes saddle-node bifurcation, Bogdanov-Takens bifurcation and Nilpotent focus bifurcation under certain parameter conditions. In section 5, we study Hopf bifurcation of system \eqref{1.5}. The paper ends with a brief discussion of the results in section 6.

\section{Existence of equilibria}

In this section, the equilibria of system \eqref{1.5} is investigated.

\subsection{Disease free equilibrium}

From system \eqref{1.5}, it follows that $E_0(0,0)$ is a unique disease free equilibrium. By Theorem 7.1 of Chapter 2 in \cite{ZTW}, the following statements hold.
\begin{lemma}\label{l1}
The disease free equilibrium $E_0(0,0)$ (i.e., boundary equilibrium) is a stable node (saddle) if $R_0<1$ ($R_0>1$), see Fig.\ref{1}(a) (Fig.\ref{1}(b)). If $R_0=1$, then $E_0$ is a saddle-node with the parabolic sector on the right half-plane, see Fig.\ref{1}(c).
\end{lemma}

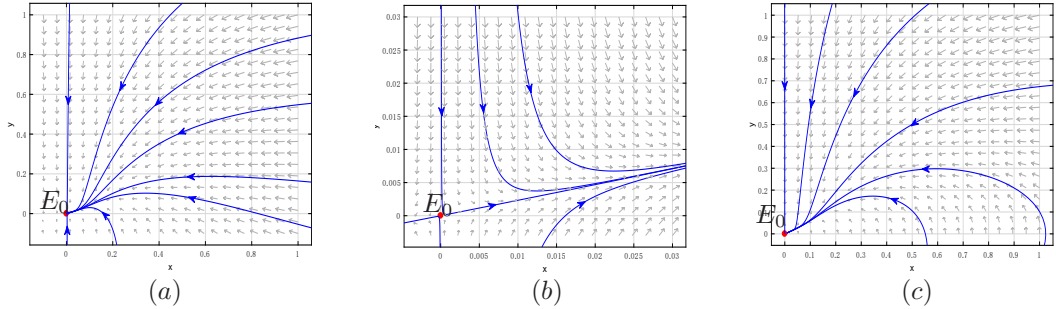
\begin{figure}[ht!]
\begin{center}
\begin{overpic}[scale=0.70]{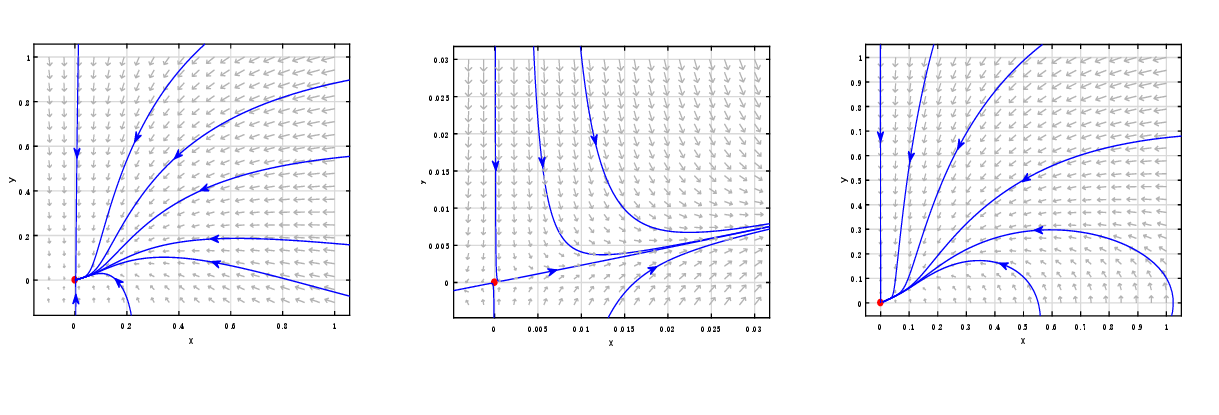}
\end{overpic}
\put(-396,47){$E_0$}
\put(-355,13){$(a)$}
\put(-251,46){$E_0$}
\put(-210,13){$(b)$}
\put(-125,42){$E_0$}
\put(-70,13){$(c)$}
\vspace{-5mm}
\end{center}
\caption{ The disease free equilibrium  $E_0$ of system \eqref{1.5}: (a) $E_0$ is a stable node for $R_0<1$; (b) $E_0$ is a saddle for $R_0>1$; (c) $E_0$ is a saddle-node for $R_0=1$.\label{1}}
\end{figure}

\subsection{Endemic equilibria}

Define
\begin{equation}
\begin{array}{ll}\label{2.2}
\begin{aligned}
H(x)\overset{\triangle}{=}(1+px^{k-1})\bigg(1-\frac{1+\eta}{\Lambda_0}x\bigg)-\frac{1}{R_0}=0,\quad x\in\big(0,\frac{\Lambda_0}{\eta+1}\big].
\end{aligned}
\end{array}
\end{equation}
Then
\begin{equation}
\begin{array}{ll}\label{2.4}
\begin{aligned}
H'(x)&=p(k-1)x^{k-2}-\frac{pk(1+\eta)}{\Lambda_0}x^{k-1}-\frac{1+\eta}{\Lambda_0},\\[2ex]
H''(x)&=p(k-1)x^{k-3}\big(k-2-\frac{k(1+\eta)}{\Lambda_0}x\big).
\end{aligned}
\end{array}
\end{equation}
In the following, we discuss the endemic equilibria (i.e. positive equilibria) of system \eqref{1.5} in 5 cases.
First, we consider the number of positive zeros of $H(x)$.

\textbf{Case 1:} $0<k<1$.
We have $H'(x)<0$ for all $x\in(0,\frac{\Lambda_0}{\eta+1})$. Besides, $\lim\limits_{x\rightarrow0^+}H(x)\rightarrow+\infty$ and $\lim\limits_{x\rightarrow\frac{\Lambda_0}{\eta+1}^-}H(x)\rightarrow-\frac{1}{R_0}$. Thus, $H(x)$ has a positive zero $x_0^*$.

\textbf{Case 2: $k=1$}.
We have $H(x)=(1+p)\big(1-\frac{1+\eta}{\Lambda_0}x\big)-\frac{1}{R_0}$ and $H(x)$ has a unique positive zero $x_1^*=\frac{\Lambda_0+p\Lambda_0-\gamma}{(1+p)(1+\eta)}$ when $R_0>\frac{1}{1+p}$.

\textbf{Case 3:} $1<k<2$.
For $1<k<2$, we have $\lim\limits_{x\rightarrow0^+}H'(x)=+\infty$ and  $\lim\limits_{x\rightarrow \frac{\Lambda_0}{\eta+1}^-}H'(x)=-\frac{(1+\eta)^{k-1}+p\Lambda_0^{k-1}}{\Lambda_0(1+\eta)^{k-2}}$. Thus, there exists $x_c\in(0,\frac{\Lambda_0}{1+\eta})$ such that $H'(x_c)=0$. Moreover, $\lim\limits_{x\rightarrow0^+}H(x)=1-\frac{1}{R_0}$ and $\lim\limits_{x\rightarrow\frac{\Lambda_0}{1+\eta}^-}H(x)=-\frac{1}{R_0}$. For $R_0<1$, $H(x)$ has a unique positive zero $x_c$  when $H(x_c)=0$. For $R_0<1$, $H(x)$ has two positive zeros $\bar{x}_1^*$ and $\bar{x}_2^*$ (where $\bar{x}_1^*<\bar{x}_2^*$) when $H(x_c)>0$. For $R_0\geq1$, $H(x)$ has a unique positive zero $x_2^*$  for $x_2^*\in (0,\frac{\Lambda_0}{1+\eta})$ when $H(x_c)>0$.

\textbf{Case 4:} $k=2$.
Define $\Delta:=(p\Lambda_0-\eta-1)^2-4(p+p\eta)(\gamma-\Lambda_0).$ For $R_0<1$, if $p>\frac{1+\eta}{\Lambda_0}$ and $\gamma<\frac{(1+\eta+p\Lambda_0)^2}{4p(1+\eta)}$, then $H(x)$ has two positive zeros $x_{01}^*=\frac{p\Lambda_0-\eta-1-\sqrt{\Delta}}{2p(1+\eta)}$ and $x_{02}^*=\frac{p\Lambda_0-\eta-1+\sqrt{\Delta}}{2p(1+\eta)}$.
For $R_0>1$, $H(x)$ has a unique positive zero $x_{02}^*$.
For $\gamma=\frac{(1+\eta+p\Lambda_0)^2}{4p(1+\eta)}$ and $p>\frac{1+\eta}{\Lambda_0}$, $H(x)$ has a unique positive zero $\bar{x}_{0}^*=\frac{p\Lambda_0-\eta-1}{2p(1+\eta)}$.

\begin{figure}[ht!]
\centering
\begin{subfigure}{0.325\linewidth}
\centering
\includegraphics[width=0.9\linewidth]{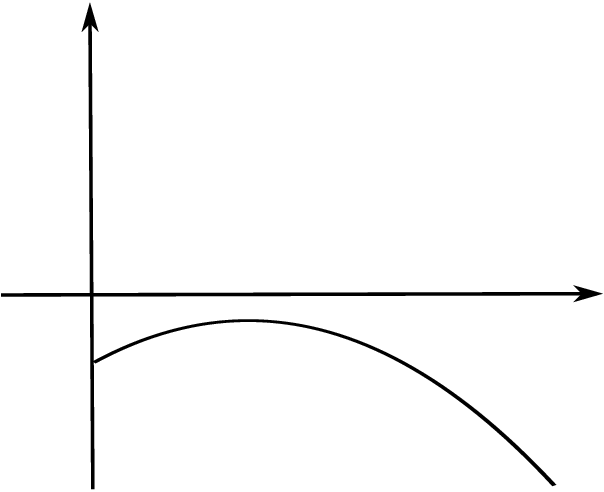}
\put(-112,33){$0$}
\put(-3,34){$x$}
\put(-113,95){$y$}
\caption{\centering{$H'(\bar{x}_c)<0$}}
\label{q1}
\end{subfigure}
\centering
\begin{subfigure}{0.325\linewidth}
\centering
\includegraphics[width=0.9\linewidth]{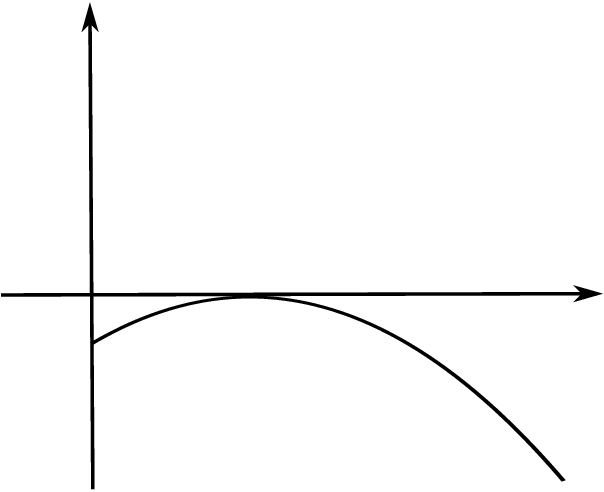}
\put(-113,33){$0$}
\put(-7,34){$x$}
\put(-114,95){$y$}
\put(-80,33){$\bar{x}_c$}
\caption{\centering{$H'(\bar{x}_c)=0$}}
\label{q2}
\end{subfigure}
\centering
\begin{subfigure}{0.325\linewidth}
\centering
\includegraphics[width=0.9\linewidth]{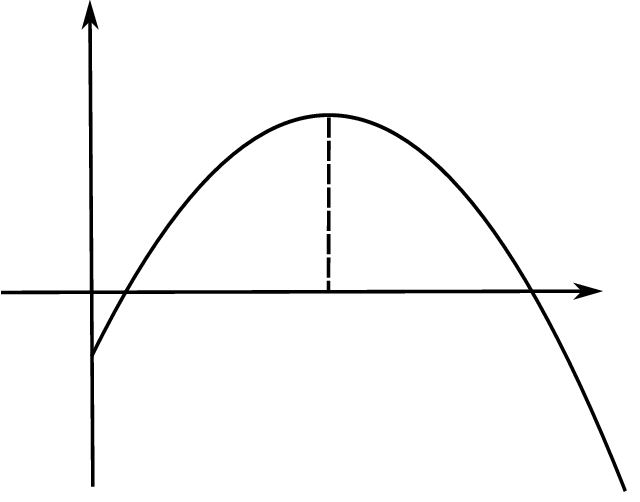}
\put(-113,33){$0$}
\put(-10,36){$x$}
\put(-114,91){$y$}
\put(-100,35){$x_{01}$}
\put(-33,35){$x_{02}$}
\put(-63,34){$\bar{x}_c$}
\caption{\centering{$H'(\bar{x}_c)>0$}}
\label{q3}
\end{subfigure}
\captionsetup{justification=centering}
\caption{The curve of $H'(x)$ when $k>2$.}
\label{L4}
\end{figure}
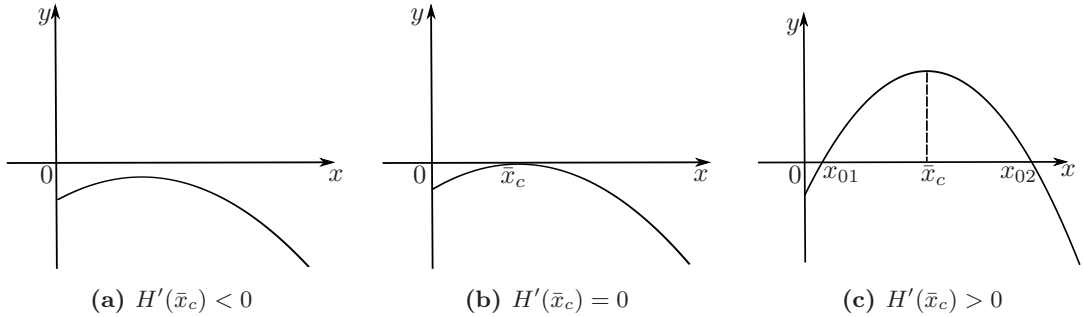

\textbf{Case 5:} $k>2$. From \eqref{2.4}, we have $\lim\limits_{x\rightarrow0^+}H'(x)=-\frac{1+\eta}{\Lambda_0}$, $\lim\limits_{x\rightarrow \frac{\Lambda_0}{1+\eta}^-}H'(x)=-\frac{(1+\eta)^{k-1}+p\Lambda_0^{k-1}}{\Lambda_0(1+\eta)^{k-2}}<0$, $H''(x)$ has a unique positive zero $\bar{x}_c=\frac{\Lambda_0}{1+\eta}\frac{k-2}{k}$
and $H''(x)>0$ if $x<\bar{x}_c$ and $H''(x)<0$ if $\bar{x}_c<x<\frac{\Lambda_0}{1+\eta}$. In addition,
\begin{equation}
\begin{array}{ll}\label{2.7}
\begin{aligned}
H'(\bar{x}_c)=p\bigg(\frac{\Lambda_0}{1+\eta}\bigg)^{k-2}\bigg(\frac{k-2}{k}\bigg)^{k-2}-\frac{1+\eta}{\Lambda_0}.
\end{aligned}
\end{array}
\end{equation}
If $H'(\bar{x}_c)<0$, then $H'(x)$ has no zero for $x\in (0,\frac{\Lambda_0}{1+\eta}]$, see Fig.\ref{L4}(a). If $H'(\bar{x}_c)=0$, then $H'(x)$ has a unique zero $\bar{x}_c$, see Fig.\ref{L4}(b) and  $H'(x)$ has two zeros $x_{01}$ and $x_{02}$ for $x\in (0,\frac{\Lambda_0}{1+\eta}]$ ($x_{01}<x_{02}$) as  $H'(\bar{x}_c)>0$, see Fig.\ref{L4}(c). Hence, $H(x)$ has at most three positive zeros for $x\in (0,\frac{\Lambda_0}{1+\eta}]$.

\begin{figure}[ht!]
\centering
\begin{subfigure}{0.325\linewidth}
\centering
\includegraphics[width=0.8\linewidth]{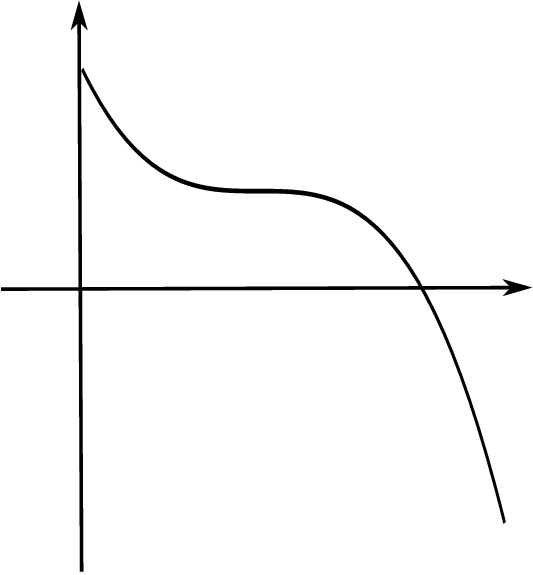}
\put(-101,51){$0$}
\put(-5,52){$x$}
\put(-102,113){$y$}
\put(-32,51){$x_4^*$}
\vspace{3mm}\caption{\centering{$R_0>1$}}
\label{w15}
\vspace{3mm}
\end{subfigure}
\centering
\begin{subfigure}{0.325\linewidth}
\centering
\includegraphics[width=0.8\linewidth]{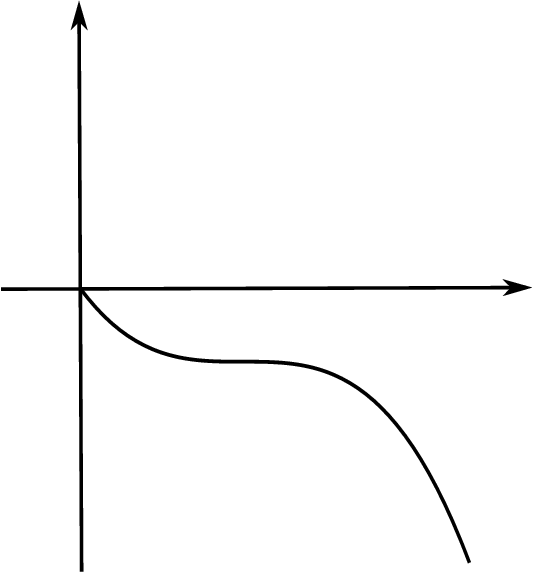}
\put(-101,51){$0$}
\put(-5,52){$x$}
\put(-102,113){$y$}
\vspace{4.3mm}\caption{\centering{$R_0=1$}}
\label{w16}
\vspace{3mm}
\end{subfigure}
\centering
\begin{subfigure}{0.325\linewidth}
\centering
\includegraphics[width=0.8\linewidth]{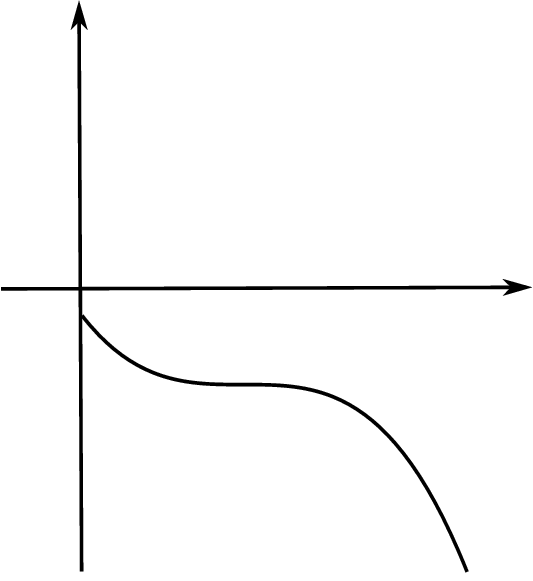}
\put(-101,52){$0$}
\put(-5,53){$x$}
\put(-102,113){$y$}
\vspace{3mm}\caption{\centering{$R_0<1$}}
\label{w17}
\vspace{3mm}
\end{subfigure}
\captionsetup{justification=centering}
\caption{The curve of $H(x)$ when $H'(\bar{x}_c)\leq0$.}
\label{L0}
\end{figure}
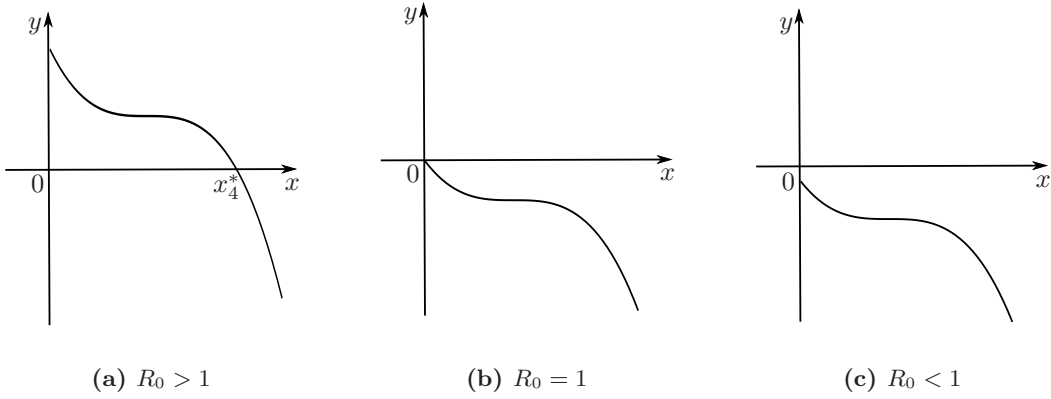

\textbf{(i):} For $H'(\bar{x}_c)\leq0  $ and $R_0>1$, $H(x)$ has a unique positive zero $x_4^*$, see Fig.\ref{L0}(a).

\begin{figure*}[ht!]
\centering
\begin{subfigure}{0.235\textwidth}
\centering
\includegraphics[width=1\linewidth]{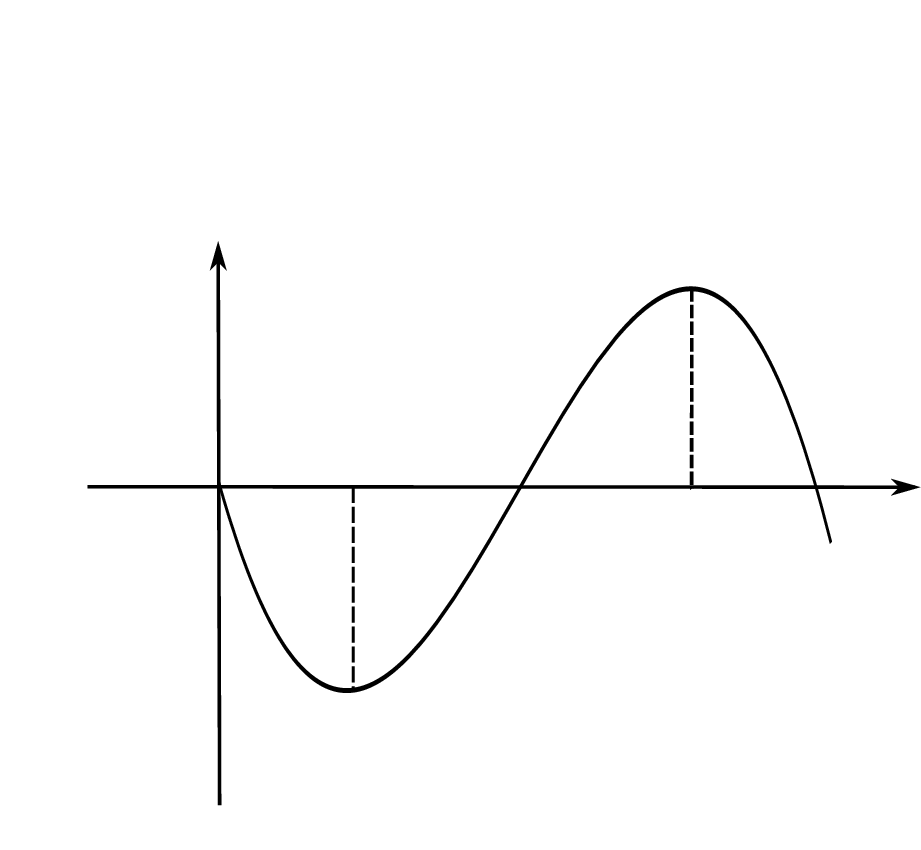}
\put(-82,32){$0$}
\put(-1,34){$x$}
\put(-83,63){$y$}
\put(-66,42){$x_{01}$}
\put(-33,34){$x_{02}$}
\put(-44,32){$\hat{x}_1^*$}
\put(-11,43){$\hat{x}_2^*$}
\caption{\centering{$R_0=1$}}
\vspace{2mm}
\end{subfigure}  
\begin{subfigure}{0.235\textwidth}
\centering
\includegraphics[width=1\linewidth]{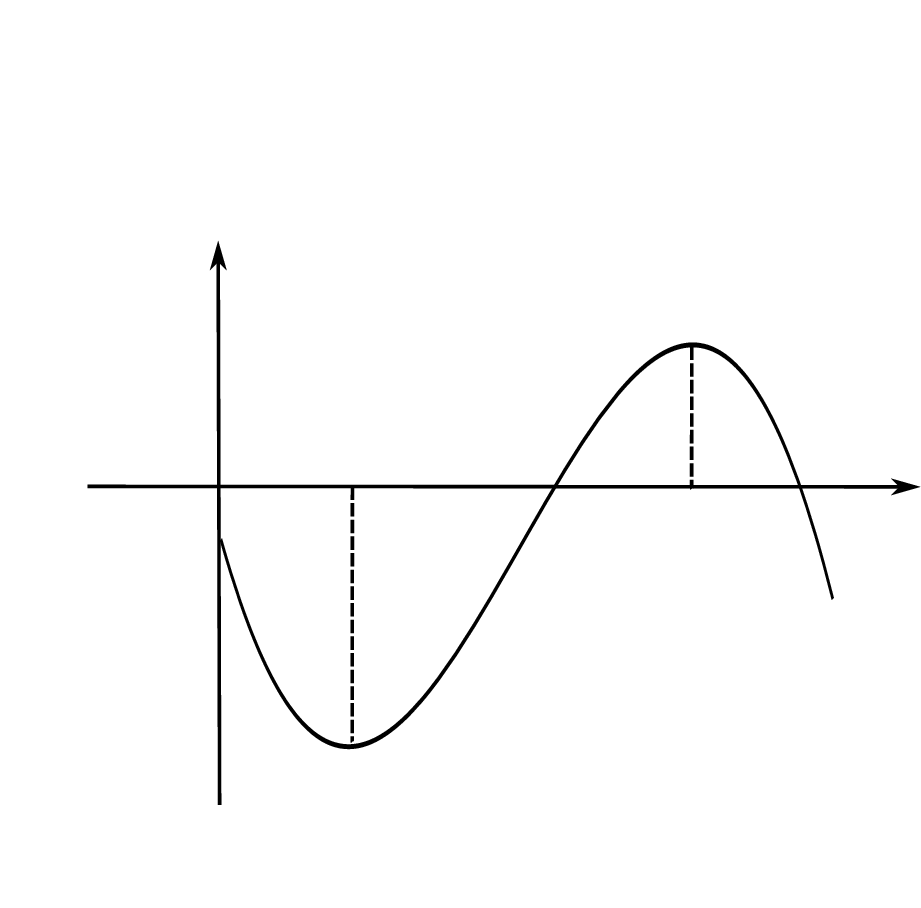}
\put(-82,38){$0$}
\put(-2,40){$x$}
\put(-83,70){$y$}
\put(-66,48){$x_{01}$}
\put(-31,40){$x_{02}$}
\put(-48,49){$\hat{x}_1^*$}
\put(-14,49){$\hat{x}_2^*$}
\caption{\centering{$R_0<1$}}
\vspace{3mm}
\end{subfigure}
\centering
\begin{subfigure}{0.235\textwidth}
\centering
\includegraphics[width=1\linewidth]{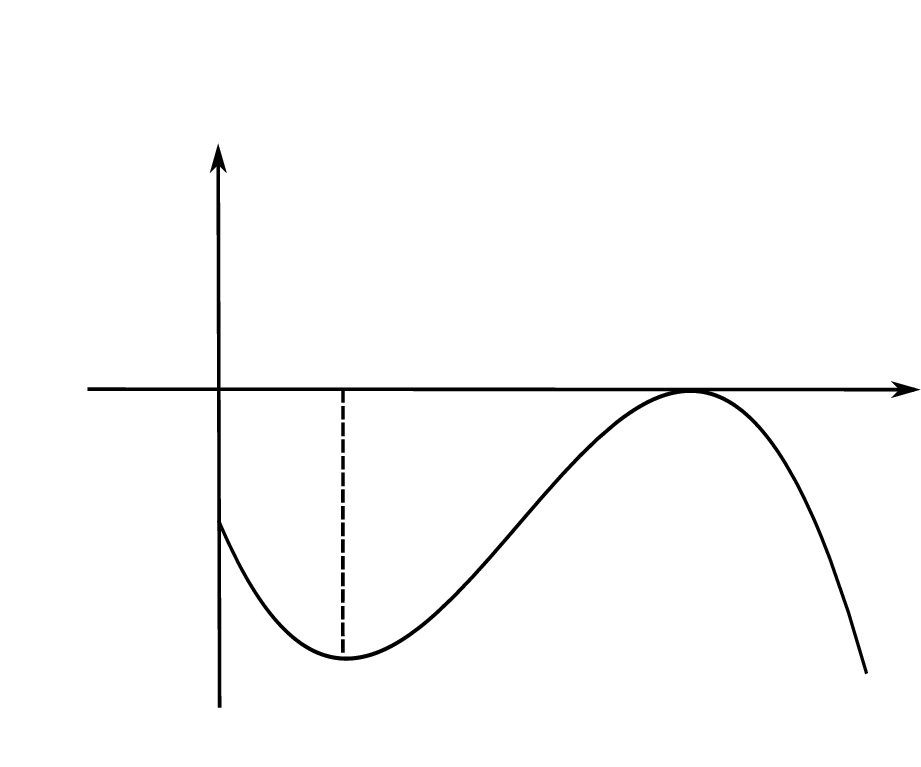}
\put(-82,33){$0$}
\put(-6,35){$x$}
\put(-83,67){$y$}
\put(-66,44){$x_{01}$}
\put(-48,45){$\hat{x}_5^*=x_{02}$}
\caption{\centering{$R_0<1$}}
\end{subfigure}   
\begin{subfigure}{0.235\textwidth}
\centering
\includegraphics[width=1\linewidth]{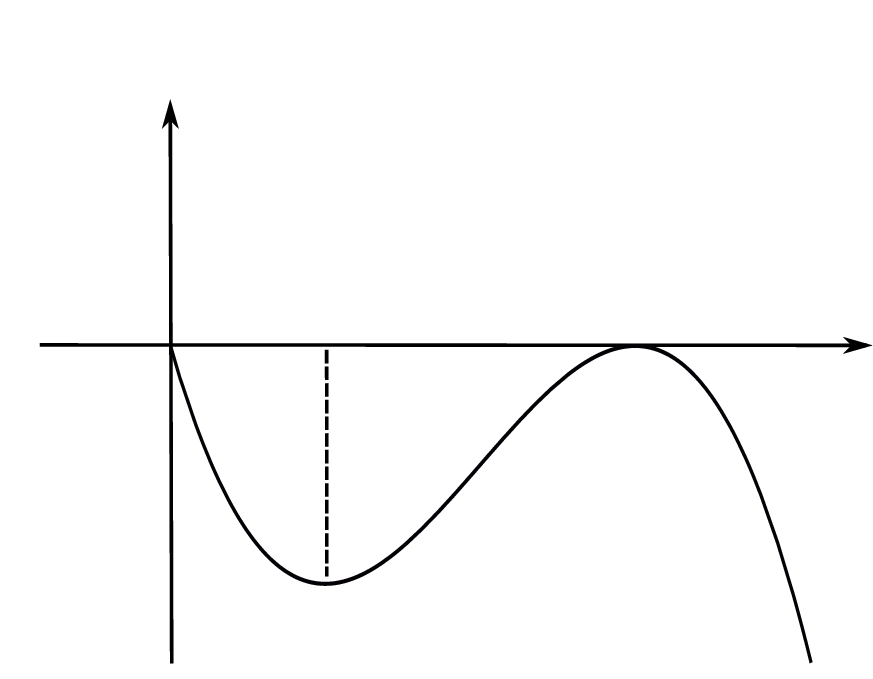}
\put(-86,31){$0$}
\put(-6,33){$x$}
\put(-86,64){$y$}
\put(-68,41){$x_{01}$}
\put(-48,41){$\hat{x}_5^*=x_{02}$}
\caption{\centering{$R_0=1$}}
\end{subfigure}
\captionsetup{justification=centering}
\caption{Graph of $H(x)$ for $H'(\bar{x}_c)>0$ and $R_0\leq1$.}
\label{L3}
\end{figure*}

\textbf{(ii):} For $H'(\bar{x}_c)>0$ and $R_0\leq1$, if $H(x_{01})<1-\frac{1}{R_0}\leq0<H(x_{02})$, then $H(x)$ has two positive zeros $\hat{x}_1^*$ and $\hat{x}_2^*$ ($\hat{x}_1^*<\hat{x}_2^*$), see Fig.\ref{L3}(a) and Fig.\ref{L3}(b). If $H(x_{01})<1-\frac{1}{R_0}\leq H(x_{02})=0$, then $H(x)$ has a unique positive zero $\hat{x}_5^*$, see Fig.\ref{L3}(c) and Fig.\ref{L3}(d).

\textbf{(iii):} For $H'({\bar{x}_c})>0$ and $R_0>1$, the following statements hold.

\quad\textbf{(iii-1):} If $H(x_{01})<0<H(x_{02})<1-\frac{1}{R_0}$ or $H(x_{01})<0<1-\frac{1}{R_0}<H(x_{02})$ or $H(x_{01})<0<1-\frac{1}{R_0}=H(x_{02})$, then $H(x)$ has three positive zeros $\bar{x}_3^*$, $\bar{x}_4^*$ and $\bar{x}_5^*$ ($\bar{x}_3^*<\bar{x}_4^*<\bar{x}_5^*$), see Fig.\ref{L1}(a), Fig.\ref{L1}(b) and Fig.\ref{L1}(c).

\quad\textbf{(iii-2):} If $H(x_{01})=0<H(x_{02})<1-\frac{1}{R_0}$ or $H(x_{01})=0<1-\frac{1}{R_0}<H(x_{02})$ or $H(x_{01})=0<1-\frac{1}{R_0}=H(x_{02})$, then $H(x)$ has two positive zeros $x_{30}^*$ and $\bar{x}_5^*$, see Fig.\ref{L1}(d), Fig.\ref{L1}(e) and Fig.\ref{L1}(f).

\quad\textbf{(iii-3):} If $H(x_{01})<0=H(x_{02})<1-\frac{1}{R_0}$, then $H(x)$ has two positive zeros $x_{40}^*$ and $\bar{x}_3^*$, see Fig.\ref{L1}(g).

\quad\textbf{(iii-4):} If $0<H(x_{01})<1-\frac{1}{R_0}=H(x_{02})$ or $0<H(x_{01})<1-\frac{1}{R_0}<H(x_{02})$ or $0<H(x_{01})<H(x_{02})<1-\frac{1}{R_0}$, then $H(x)$ has a unique positive zero $\hat{x}_3^*$, see Fig.\ref{L1}(h), Fig.\ref{L1}(i) and Fig.\ref{L1}(j).

\quad\textbf{(iii-5):} If $H(x_{01})<H(x_{02})<0<1-\frac{1}{R_0}$, then $H(x)$ has a unique positive zero $\hat{x}_4^*$, see Fig.\ref{L1}(k).

\begin{figure*}[ht!]
\centering
\begin{subfigure}{0.235\textwidth}
\centering
\includegraphics[width=1\linewidth]{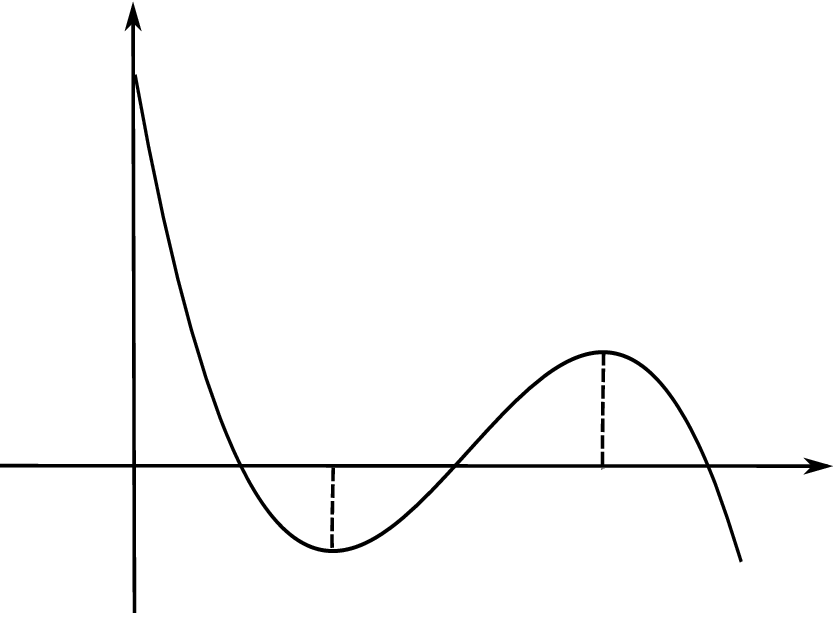}
\put(-91,12){$0$}
\put(-5,14){$x$}
\put(-91,71){$y$}
\put(-66,22){$x_{01}$}
\put(-35,14){$x_{02}$}
\put(-81,13){$\bar{x}_3^*$}
\put(-46,13){$\bar{x}_4^*$}
\put(-15,23){$\bar{x}_5^*$}
\put(-50,-6){$(a)$}
\end{subfigure}   
\begin{subfigure}{0.235\textwidth}
\centering
\includegraphics[width=\linewidth]{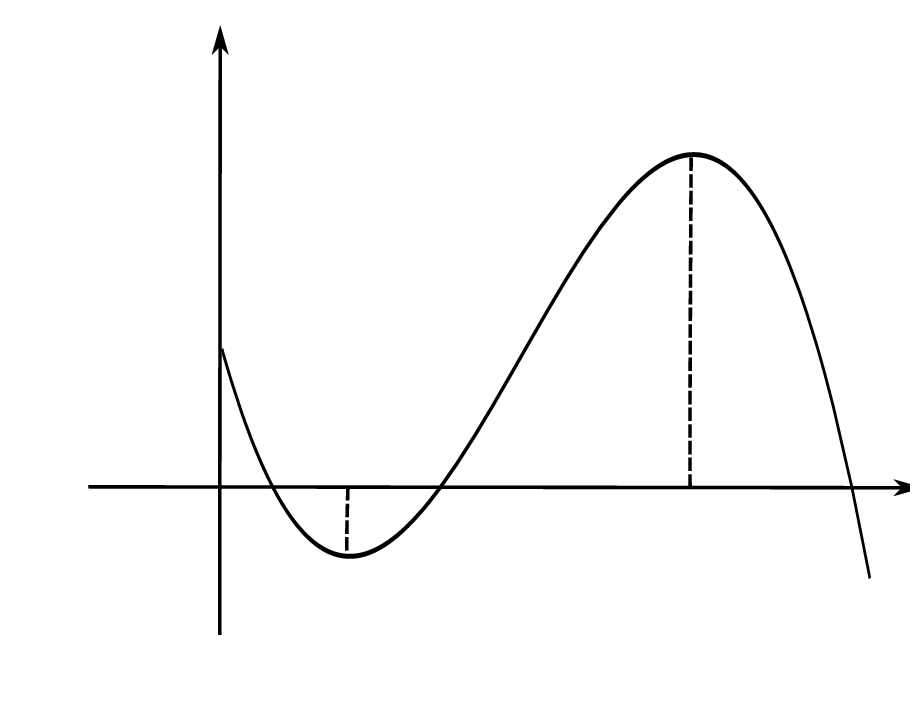}
\put(-82,17){$0$}
\put(-3,19){$x$}
\put(-82,71){$y$}
\put(-66,27){$x_{01}$}
\put(-33,19){$x_{02}$}
\put(-76,16){$\bar{x}_3^*$}
\put(-52,17){$\bar{x}_4^*$}
\put(-7,27){$\bar{x}_5^*$}
\put(-50,-6){$(b)$}
\vspace{-2mm}
\end{subfigure}
\centering
\begin{subfigure}{0.235\textwidth}
\centering
\includegraphics[width=1\linewidth]{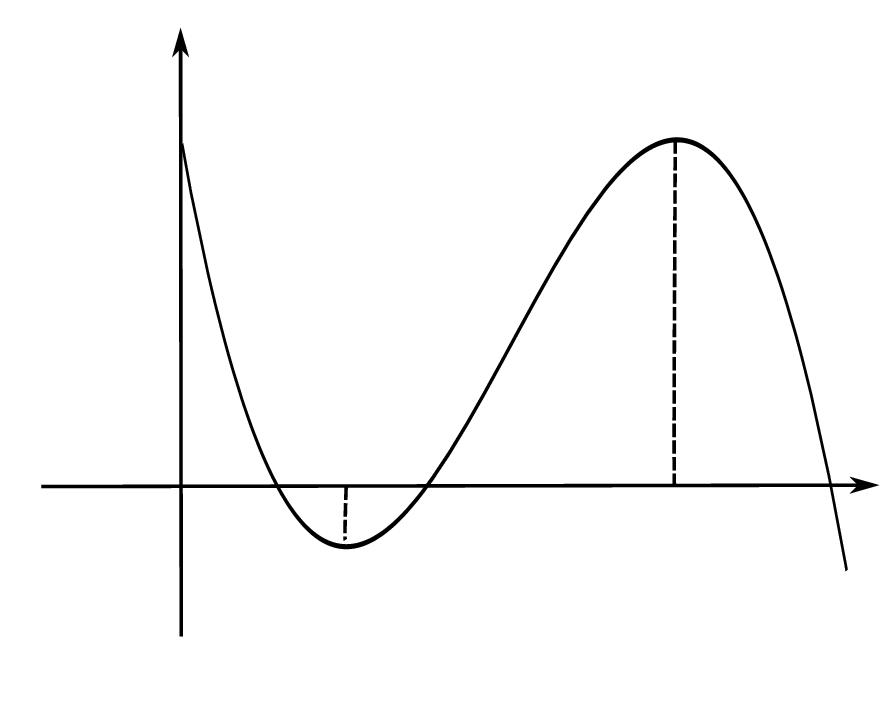}
\put(-86,17){$0$}
\put(-3,19){$x$}
\put(-86,71){$y$}
\put(-66,27){$x_{01}$}
\put(-33,19){$x_{02}$}
\put(-77,17){$\bar{x}_3^*$}
\put(-52,17){$\bar{x}_4^*$}
\put(-7,27){$\bar{x}_5^*$}
\put(-50,-6){$(c)$}
\end{subfigure}   
\begin{subfigure}{0.235\textwidth}
\centering
\includegraphics[width=\linewidth]{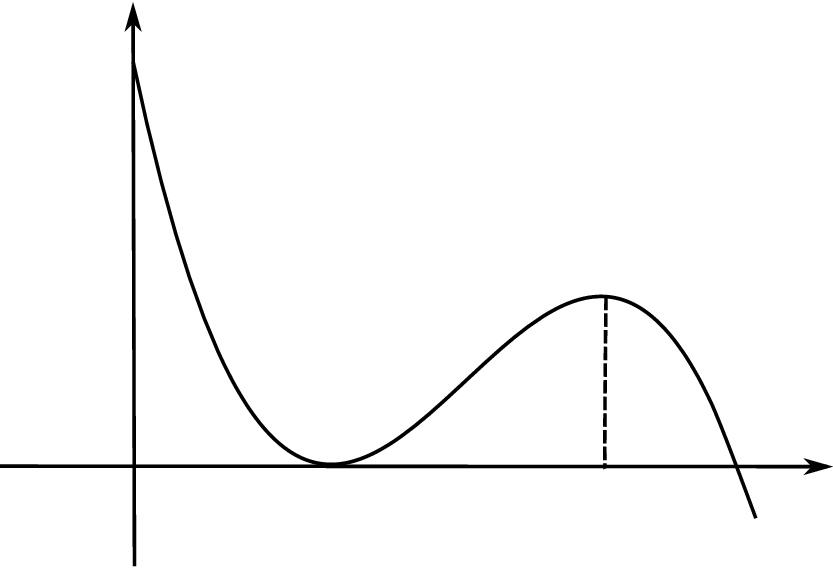}
\put(-90,5){$0$}
\put(-4,7){$x$}
\put(-90,65){$y$}
\put(-68,18){$x_{01}$}
\put(-35,7){$x_{02}$}
\put(-68,5){$x_{30}^*$}
\put(-13,16){$\bar{x}_5^*$}
\put(-50,-20){$(d)$}
\end{subfigure}

\begin{subfigure}{0.235\textwidth}
\centering
\includegraphics[width=\linewidth]{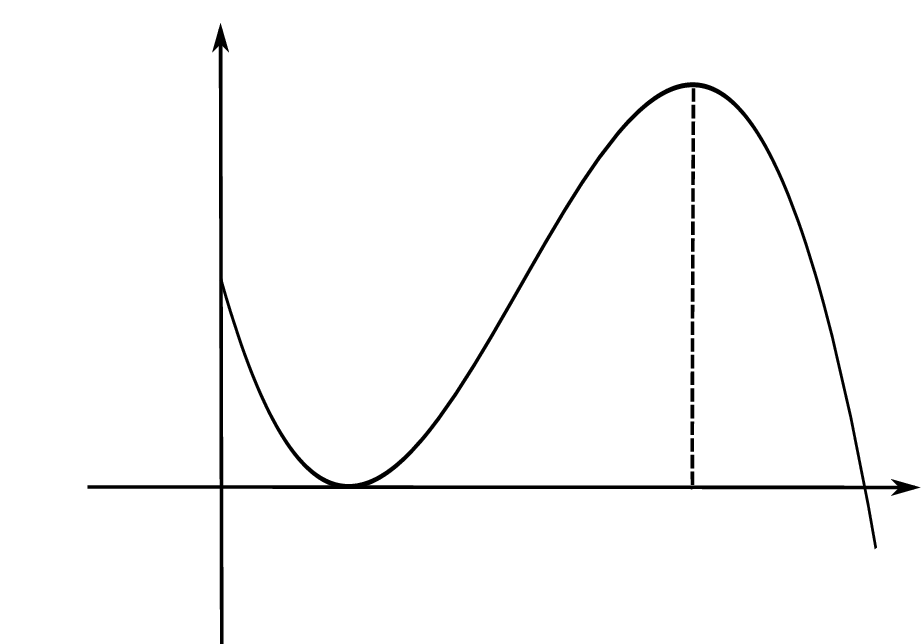}
\put(-83,9){$0$}
\put(-4,10){$x$}
\put(-83,62){$y$}
\put(-70,22){$x_{01}$}
\put(-33,11){$x_{02}$}
\put(-70,9){$x_{30}^*$}
\put(-7,21){$\bar{x}_5^*$}
\put(-50,-20){$(d)$}
\end{subfigure}
\begin{subfigure}{0.235\textwidth}
\centering
\includegraphics[width=\linewidth]{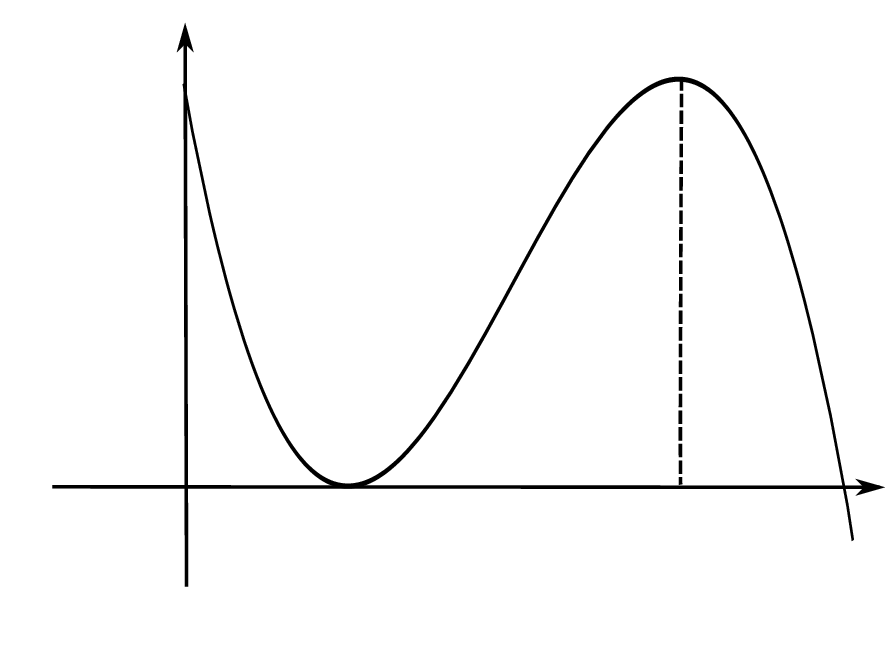}
\put(-85,10){$0$}
\put(-4,10){$x$}
\put(-85,66){$y$}
\put(-31,12){$x_{02}$}
\put(-68,23){$x_{01}$}
\put(-68,10){$x_{30}^*$}
\put(-5,21){$\bar{x}_5^*$}
\put(-50,-20){$(f)$}
\end{subfigure}
\begin{subfigure}{0.235\textwidth}
\centering
\includegraphics[width=\linewidth]{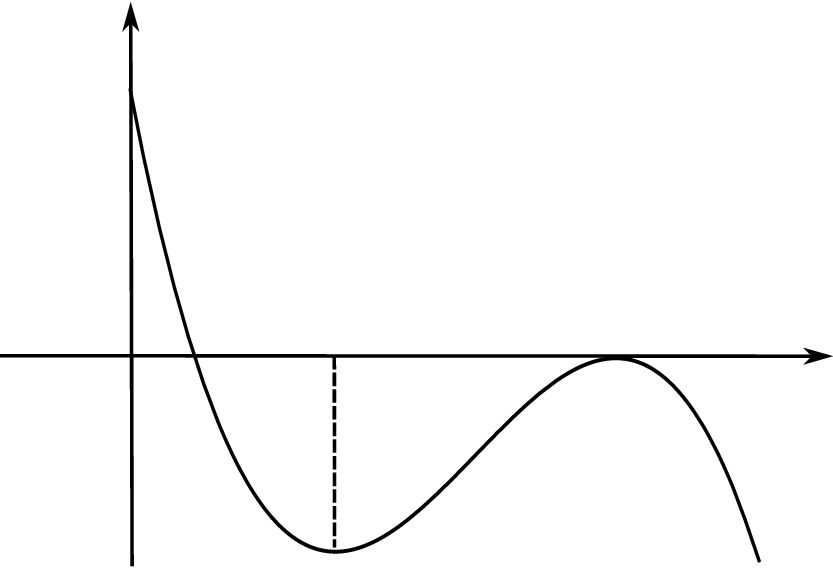}
\put(-90,18){$0$}
\put(-4,20){$x$}
\put(-91,62){$y$}
\put(-59,19){$x_{01}$}
\put(-34,18){$x_{02}$}
\put(-35,30){$x_{40}^*$}
\put(-77,28){$\bar{x}_3^*$}
\put(-50,-25){$(g)$}
\end{subfigure}
\begin{subfigure}{0.235\textwidth}
\centering
\includegraphics[width=\linewidth]{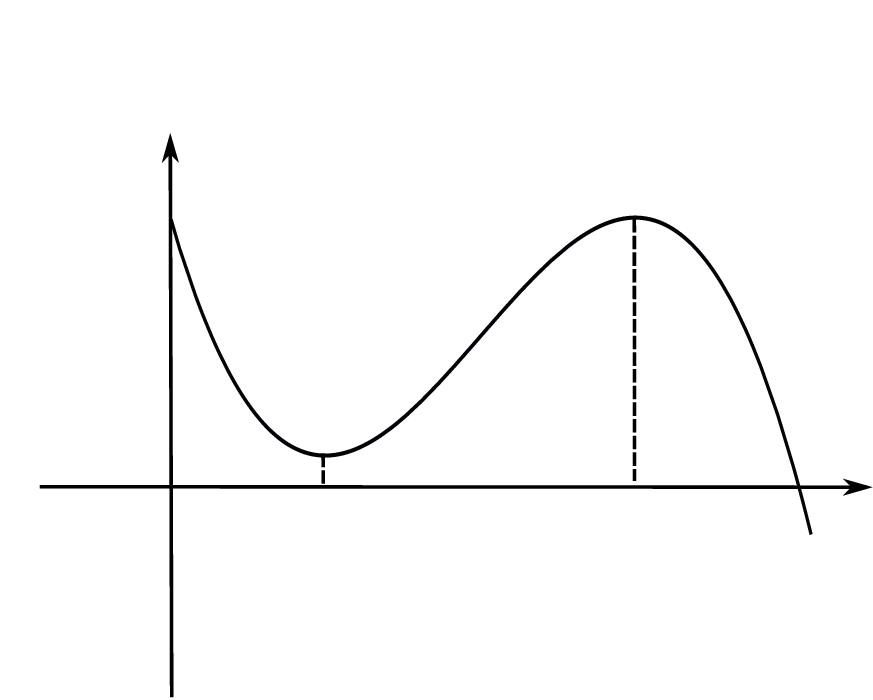}
\put(-86,17){$0$}
\put(-4,19){$x$}
\put(-86,62){$y$}
\put(-68,19){$x_{01}$}
\put(-34,19){$x_{02}$}
\put(-10,28){$\hat{x}_3^*$}
\put(-50,-15){$(h)$}
\end{subfigure}

\begin{subfigure}{0.235\textwidth}
\centering
\includegraphics[width=\linewidth]{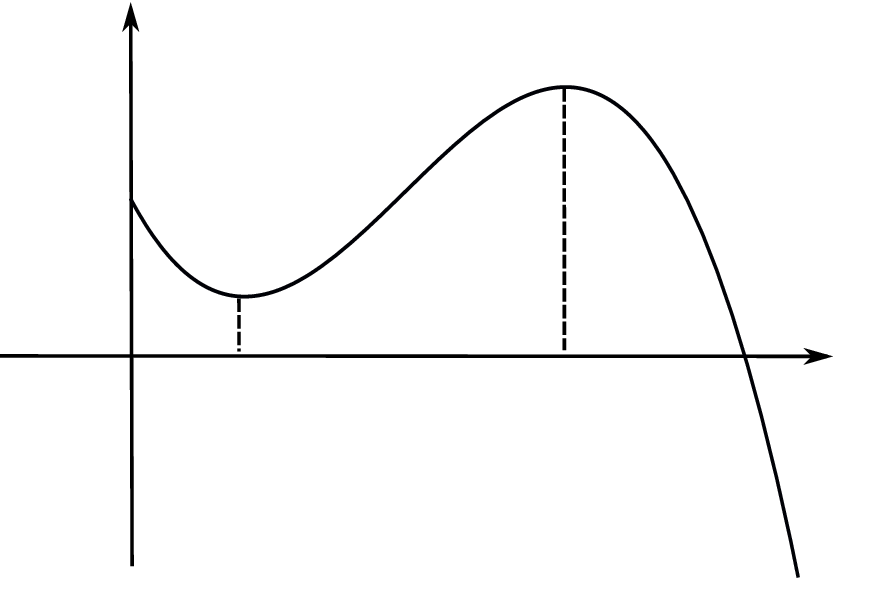}
\put(-91,19){$0$}
\put(-8,21){$x$}
\put(-92,62){$y$}
\put(-79,22){$x_{01}$}
\put(-43,22){$x_{02}$}
\put(-15,31){$\hat{x}_3^*$}
\put(-50,-20){$(i)$}
\end{subfigure}
\begin{subfigure}{0.235\textwidth}
\centering
\includegraphics[width=\linewidth]{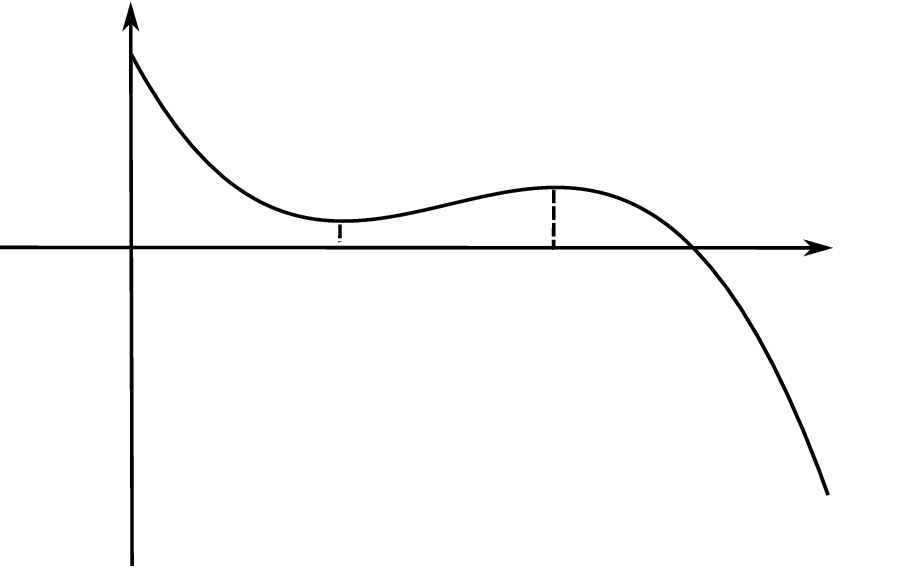}
\put(-92,28){$0$}
\put(-13,30){$x$}
\put(-92,59){$y$}
\put(-71,30){$x_{01}$}
\put(-43,30){$x_{02}$}
\put(-25,39){$\hat{x}_3^*$}
\put(-50,-25){$(j)$}
\end{subfigure}
\begin{subfigure}{0.235\textwidth}
\centering
\includegraphics[width=\linewidth]{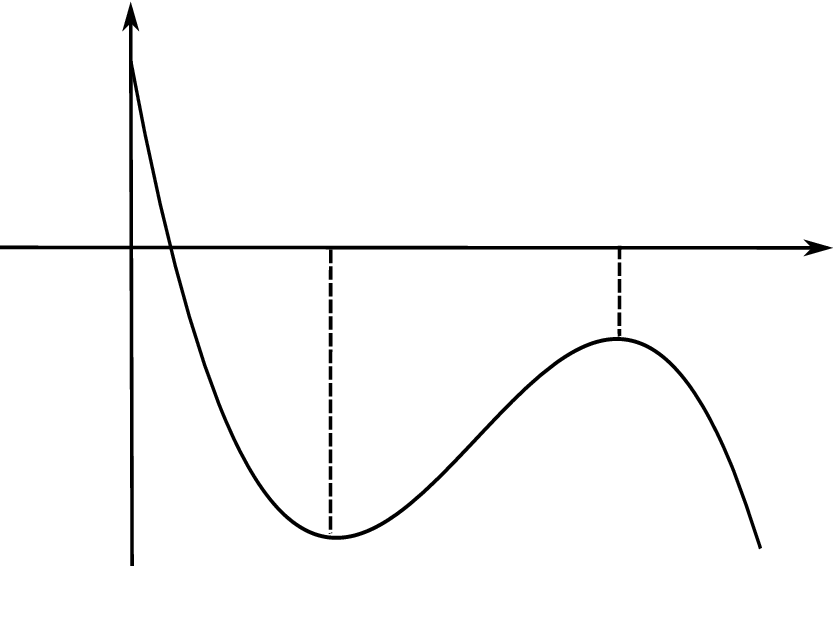}
\put(-90,37){$0$}
\put(-6,39){$x$}
\put(-91,70){$y$}
\put(-68,48){$x_{01}$}
\put(-35,48){$x_{02}$}
\put(-80,48){$\hat{x}_4^*$}
\put(-50,-18){$(k)$}
\end{subfigure}

\captionsetup{justification=centering}
\caption{Graph of $H(x)$ for $H'(\bar{x}_c)>0$ and $R_0>1$.}
\label{L1}
\end{figure*}
\begin{figure}[ht!]
\begin{center}
\begin{overpic}[scale=0.60]{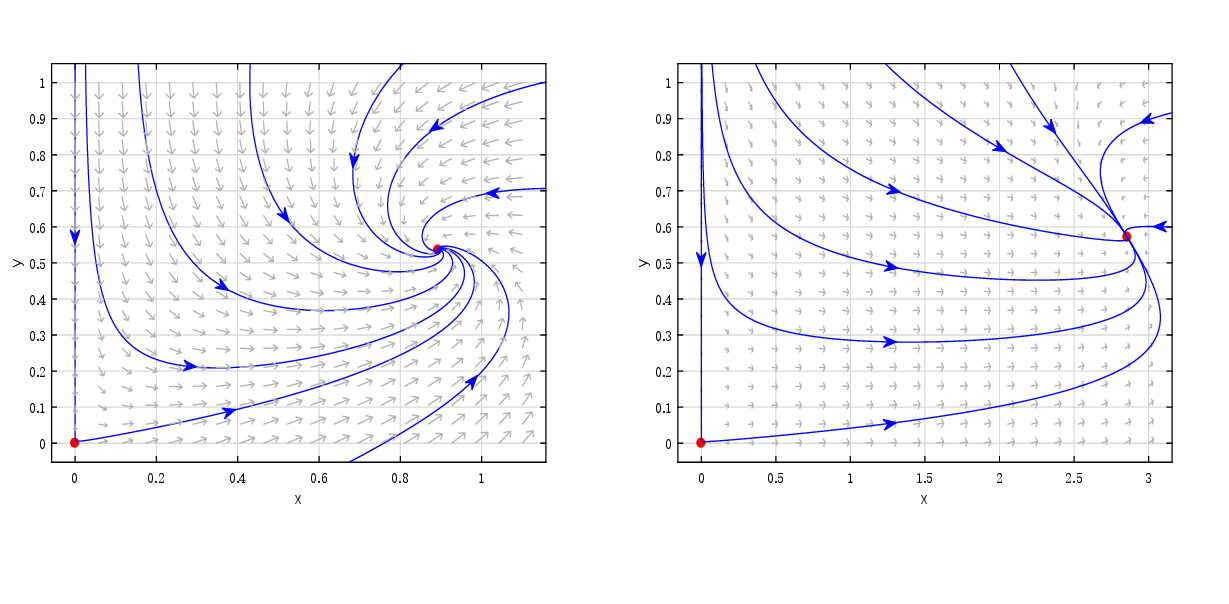}
\end{overpic}
\put(-330,46){$E_0$}
\put(-225,96){$E_1^*$}
\put(-271,13){$(a)$}
\put(-150,45){$E_0$}
\put(-29,97){$E_1^*$}
\put(-91,13){$(b)$}
\vspace{-5mm}
\end{center}
\caption{ When $k=1$, system \eqref{1.5} exists a boundary equilibrium $E_0$ and a unique positive equilibrium $E_1^*(x_1^*,\eta x_1^*)$. (a) $E_1^*$ is a stable focus; (b) $E_1^*$ is a stable node.\label{1-1}}
\end{figure}
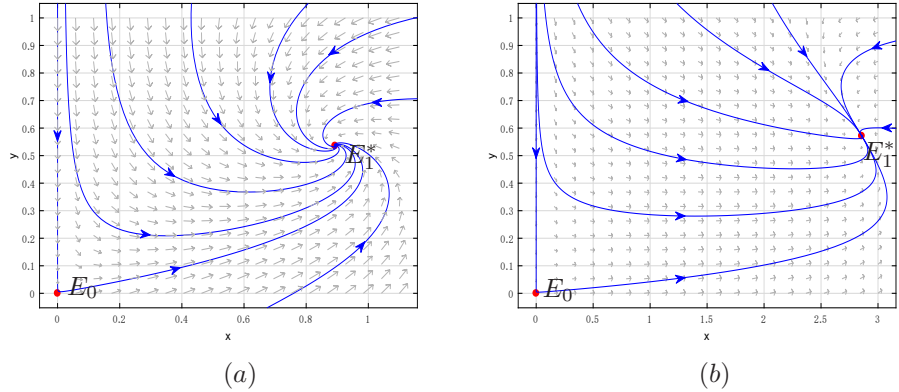
The Jacobian matrix of system \eqref{1.5} at the positive equilibrium $E(x,y)\ (H(x)=0)$ is
\begin{equation*}
J(E)=\left(
\begin{array}{cc}
\Lambda_0H(x)+\Lambda_0xH'(x)+\frac{\gamma\eta x}{\Lambda_0-(1+\eta)x}& -\frac{\gamma x}{\Lambda_0-(1+\eta)x}\big)\\
\eta& -1
\end{array}
\right)
\end{equation*}
and
\begin{equation*}
\begin{array}{ll}
\begin{aligned}
Tr(J(E))&=\Lambda_0H(x)+\Lambda_0xH'(x)+\frac{\gamma\eta x}{\Lambda_0-(1+\eta)x}-1,\\[2ex]
Det(J(E))&=-\Lambda_0H(x)-\Lambda_0xH'(x).
\end{aligned}
\end{array}
\end{equation*}

From above analysis, we have the following results.

\begin{lemma}\label{l2}
For $k\leq1$, the following statements for system.

(i) When $0<k<1$, system \eqref{1.5} has a unique positive equilibrium $E_0^*(x_0^*,\eta x_0^*)$, which is stable node.

(ii) When $k=1$ and $\Lambda_0>\frac{\gamma}{1+p}$, system \eqref{1.5} has unique equilibrium $E_1^*(x_1^*,\eta x_1^*)$, which is a stable focus or node, see Fig.\ref{1-1}.
\end{lemma}

\begin{lemma}\label{l3}
For $1<k\leq2$, the following statements hold.

(i) When $1<k<2$ and $R_0<1$,

\quad(i-1) if $H(x_c)=0$, then system \eqref{1.5} has a unique degenerate positive equilibrium $E_2^*(x_c,\eta x_c)$;

\quad(i-2) if $H(x_c)>0$, then system \eqref{1.5} has two positive equilibria $\bar{E}_1^*(\bar{x}_1^*,\eta\bar{x}_1^*)$ and $\bar{E}_2^*(\bar{x}_2^*,\eta\bar{x}_2^*)$, in which $\bar{E}_1^*$ is a saddle and $\bar{E}_2^*$ is a node or focus.

(ii) When $1<k<2$, $H(x_c)>0$ and $R_0\geq1$, system \eqref{1.5} has a unique positive equilibrium $\bar{E}_2^*$, which is a node or focus.

(iii) When $k=2$ and $\gamma\leq\frac{(1+\eta+p\Lambda_0)^2}{4p(1+\eta)}$,

\quad(iii-1) if $R_0<1$ and $p>\frac{1+\eta}{\Lambda_0}$, then system \eqref{1.5} has two positive equilibria $E_{01}^*$ and $E_{02}^*$, in which $E_{01}^*$ is a saddle and $E_{02}^*$ is a node or focus;

\quad(iii-2) if $R_0>1$ and $p>\frac{1+\eta}{\Lambda_0}$ or $p<\frac{1+\eta}{\Lambda_0}$, then system \eqref{1.5} has a unique positive equilibrium $E_{02}^*$, which is a node or focus;

\quad(iii-3) if $\gamma=\frac{(1+\eta+p\Lambda_0)^2}{4p(1+\eta)}$ and $\Lambda_0>\frac{1+\eta}{p}$, system \eqref{1.5} has a unique degenerate positive equilibrium $\bar{E}_0^*=\big(\frac{p\Lambda_0-\eta-1}{2p(1+\eta)},\frac{\eta(p\Lambda_0-\eta-1)}{2p(1+\eta)}\big)$.
\end{lemma}

\begin{lemma}\label{l4}
When $k>2$, the following statements hold.

\textbf{(i):} For $H'(\bar{x}_c)\leq0$ and  $R_0>1$,  system \eqref{1.5} has a unique positive equilibrium $E_4^*(x_4^*,\eta x_4^*)$, in which $E_4^*$ is a degenerate equilibrium when $H'(\bar{x}_c)=0$ and $x_4^*=\bar{x}_c$.

\textbf{(ii):} For $H'(\bar{x}_c)>0$ and $R_0\leq1$, if $H(x_{01})<1-\frac{1}{R_0}\leq0<H(x_{02})$, then system \eqref{1.5} has two positive equilibria $\hat{E}_1^*(\hat{x}_1^*,\eta\hat{x}_1^*)$ and $\hat{E}_2^*(\hat{x}_2^*,\eta\hat{x}_2^*)$, in which $\hat{E}_1^*$ is a saddle and $\hat{E}_2^*$ is a node or focus. If $H(x_{01})<1-\frac{1}{R_0}\leq H(x_{02})=0$, then system \eqref{1.5} has a unique degenerate positive equilibrium $\hat{E}_5^*(\hat{x}_5^*,\eta\hat{x}_5^*)$.

\textbf{(iii):} For $H'({\bar{x}_c})>0$ and $R_0>1$, the following statements hold.

\quad\textbf{(iii-1):} If $H(x_{01})<0<H(x_{02})<1-\frac{1}{R_0}$ or $H(x_{01})<0<1-\frac{1}{R_0}<H(x_{02})$ or $H(x_{01})<0<1-\frac{1}{R_0}=H(x_{02})$, then system \eqref{1.5} has three positive equilibria $\bar{E}_3^*(\bar{x}_3^*,\eta\bar{x}_3^*)$, $\bar{E}_4^*(\bar{x}_4^*,\eta\bar{x}_4^*)$ and $\bar{E}_5^*(\bar{x}_5^*,\eta\bar{x}_5^*)$, which are focus or node, saddle and node or focus, respectively.

\quad\textbf{(iii-2):} If $H(x_{01})=0<H(x_{02})<1-\frac{1}{R_0}$ or $H(x_{01})=0<1-\frac{1}{R_0}<H(x_{02})$ or $H(x_{01})=0<1-\frac{1}{R_0}=H(x_{02})$, then system \eqref{1.5} has two equilibria $E_{30}^*(x_{30}^*,\eta x_{30}^*)$ and $\bar{E}_5^*(\bar{x}_5^*,\eta\bar{x}_5^*)$ ($x_{30}^*<\bar{x}_5^*$), in which $E_{30}^*$ is a degenerate equilibrium.

\quad\textbf{(iii-3):} If $H(x_{01})<0=H(x_{02})<1-\frac{1}{R_0}$, then system \eqref{1.5} has two equilibria $\bar{E}_3^*(\bar{x}_3^*,\eta\bar{x}_3^*)$ and $E_{40}^*(x_{40}^*,\eta x_{40}^*)$ ($\bar{x}_3^*<x_{40}^*$), in which $E_{40}^*$ is a degenerate equilibrium.

\quad\textbf{(iii-4):} If $0<H(x_{01})<1-\frac{1}{R_0}=H(x_{02})$ or $0<H(x_{01})<1-\frac{1}{R_0}<H(x_{02})$ or $0<H(x_{01})<H(x_{02})<1-\frac{1}{R_0}$, then system \eqref{1.5} has a unique equilibrium $\hat{E}_3^*(\hat{x}_3^*,\eta\hat{x}_3^*)$, in which $\hat{E}_3^*$ is a stable (an unstable) node or focus when $\gamma<\frac{\big(1-\Lambda_0\hat{x}_3^*H'(\hat{x}_3^*)\big)\big(\Lambda_0-(1+\eta)\hat{x}_3^*\big)}{\eta \hat{x}_3^*}$ ($\gamma>\frac{\big(1-\Lambda_0\hat{x}_3^*H'(\hat{x}_3^*)\big)\big(\Lambda_0-(1+\eta)\hat{x}_3^*\big)}{\eta \hat{x}_3^*}$).

\quad\textbf{(iii-5):} If $H(x_{01})<H(x_{02})<0<1-\frac{1}{R_0}$, then system \eqref{1.5} has a unique equilibrium $\hat{E}_4^*(\hat{x}_4^*,\eta\hat{x}_4^*)$, in which $\hat{E}_4^*$ is a stable (an unstable) node or focus when $\gamma<\frac{\big(1-\Lambda_0\hat{x}_4^*H'(\hat{x}_4^*)\big)\big(\Lambda_0-(1+\eta)\hat{x}_4^*\big)}{\eta \hat{x}_4^*}$ ($\gamma>\frac{\big(1-\Lambda_0\hat{x}_4^*H'(\hat{x}_4^*)\big)\big(\Lambda_0-(1+\eta)\hat{x}_4^*\big)}{\eta \hat{x}_4^*}$).
\end{lemma}

\section{Degenerate equilibria of system \eqref{1.5}}
\begin{figure}[ht!]
\begin{center}
\begin{overpic}[scale=0.6]{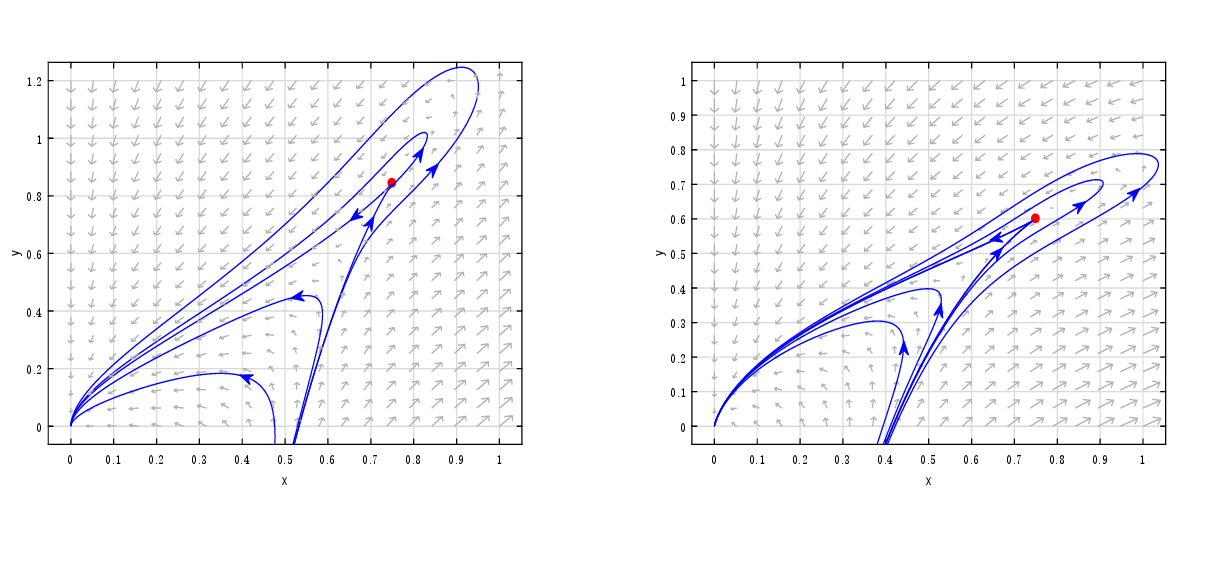}
\end{overpic}
\put(-280,13){$(a)$}
\put(-245,104){$E_2^*$}
\put(-90,13){$(b)$}
\put(-60,93){$\bar{E}_0^*$}
\vspace{-5mm}
\end{center}
\caption{ The phase portrait for system \eqref{1.5} when $1<k\leq2$. (a) $E_2^*$ is a saddle-node when $1<k<2$; (b) $\bar{E}_0^*$ is a saddle-node when $k=2$.\label{x2}}
\end{figure}
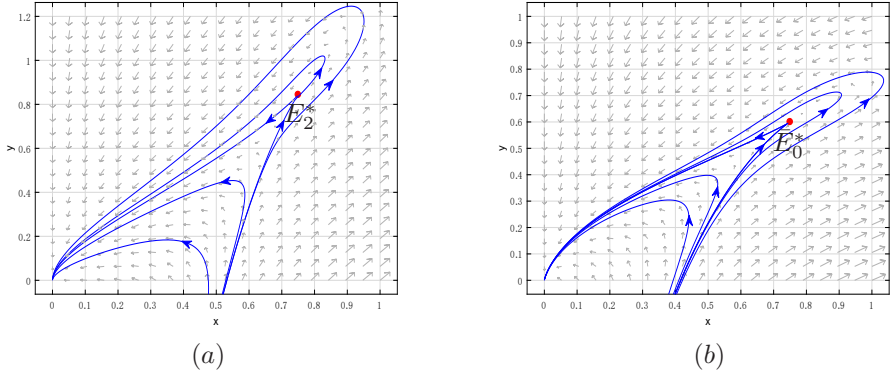
In this section, we consider the  degenerate equilibria for system \eqref{1.5}. From Lemma \ref{l3} and Lemma \ref{l4}, we know that $E_2^*$, $\bar{E}_0^*$, $E_{30}^*$, $E_{40}^*$, $E_4^*$ and $\hat{E}_5^*$ are degenerate equilibria.
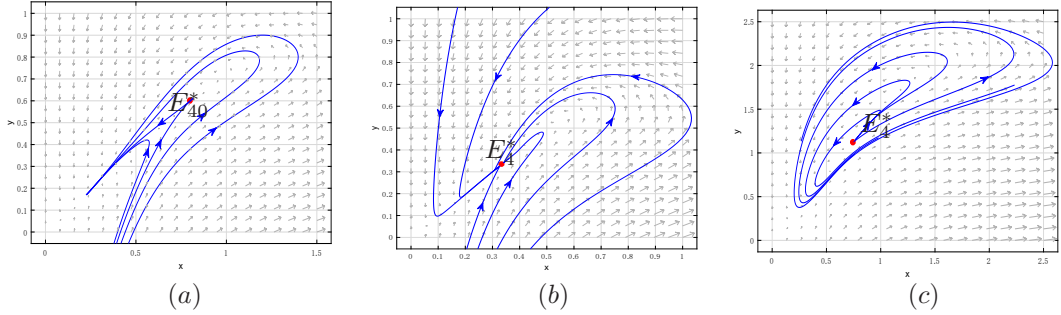
\begin{figure}[ht!]
\begin{center}
\begin{overpic}[scale=0.70]{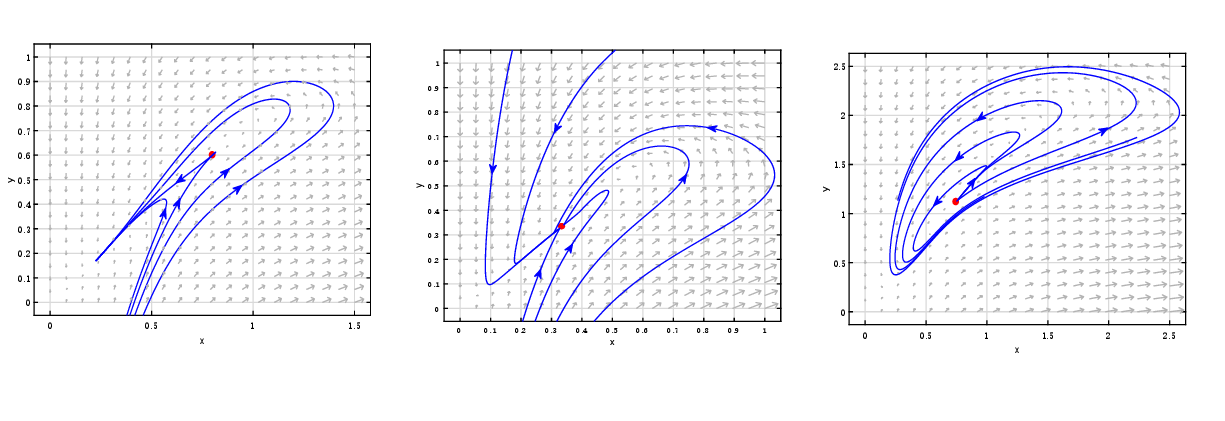}
\end{overpic}
\put(-350,13){$(a)$}
\put(-350,86){$E_{40}^*$}
\put(-210,13){$(b)$}
\put(-230,68){$E_4^*$}
\put(-70,13){$(c)$}
\put(-88,78){$E_4^*$}
\vspace{-5mm}
\end{center}
\caption{ The phase portrait for system \eqref{1.5} when $k>2$. (a) $E_{40}^*$ is a saddle-node; (b) $E_4^*$ is a stable degenerate node; (c) $E_4^*$ is an unstable degenerate node.\label{x1}}
\end{figure}
\begin{theorem}\label{t1}
For $\Lambda_0=\hat{\Lambda}_0:=\frac{(1+\eta)(z+kpz^k)}{z^{k-1}p(k-1)}$, $\gamma=\hat{\gamma}:=\frac{(1+\eta)(z+pz^k)^2}{z^{k}p(k-1)}>\eta$, $\hat{\eta}:=\frac{1}{z+pz^k}$ and $\bar{\eta}_1:=\frac{k}{2z(k-1)}$, the following statements hold.

(i)  $E_2^*$ is a saddle-node for $\eta\neq\hat{\eta}$ and $1<k<2$, see Fig.\ref{x2}(a);

(ii) $\bar{E}_0^*$ is a saddle-node for $\eta\neq\frac{1}{z+pz^2}$ and $k=2$, see Fig.\ref{x2}(b);

(iii) $E_{30}^*$ ($E_{40}^*$, $E_4^*$ and $\hat{E}_5^*$) is a saddle-node with a stable (or unstable) parabolic sector for $\eta<\hat{\eta}$ (or $\eta>\hat{\eta}$), $p\neq\hat{p}:=\frac{k-2}{kz^{k-1}}$ and $k>2$, see Fig.\ref{x1}(a);

(iv)  $E_4^*$ is a stable (an unstable) degenerate node for $p=\hat{p}$, $\eta<\bar{\eta}_1$ and $z>\frac{k(\sqrt{2k-3}-1)}{4(k-1)}$ or $z<\frac{k(\sqrt{2k-3}-1)}{4(k-1)}$ and $\eta<\frac{4z(k-1)}{k^2+4z-4kz-2k}$ ($\eta>\bar{\eta}_1$ and $z>\frac{k(k-2)}{4(k-1)}$ or $\bar{\eta}_1<\eta<\frac{4z(k-1)}{k^2+4z-4kz-2k}$ and $\frac{k(\sqrt{2k-3}-1)}{4(k-1)}<z<\frac{k(k-2)}{4(k-1)}$), see Fig.\ref{x1}(b) and Fig.\ref{x1}(c).
\end{theorem}
\begin{proof}
Denote the abscissa of the degenerate equilibria $E_2^*$ or $\bar{E}_0^*$ or $E_{30}^*$ or $E_{40}^*$ or $E_4^*$ or $\hat{E}_5^*$ by  $z$. For $H(z)=0$ and $H'(z)=0$, we have $\Lambda_0=\hat{\Lambda}_0$ and $\gamma=\hat{\gamma}$.
The Jacobian matrix of system \eqref{1.5} at $E_2^*$ (or $\bar{E}_0^*$ or $E_{30}^*$ or $E_{40}^*$ or $E_4^*$ or $\hat{E}_5^*$) is
\begin{equation*}
\left(
\begin{array}{cc}
\eta(z+pz^k)& -(z+pz^k)\\
\eta& -1
\end{array}
\right),
\end{equation*}
and the characteristic equation of system \eqref{1.5} at $E_2^*$ (or $\bar{E}_0^*$ or $E_{30}^*$ or $E_{40}^*$ or $E_4^*$ or $\hat{E}_5^*$) is
\begin{equation}
\begin{array}{ll}\label{3.1}
\begin{aligned}
Q(\lambda)\overset{\triangle}{=}\lambda^2+(1-\eta z-p\eta z^k)\lambda.
\end{aligned}
\end{array}
\end{equation}
From \eqref{3.1}, we know that equilibrium $E_2^*$ (or $\bar{E}_0^*$ or $E_{30}^*$ or $E_{40}^*$ or $E_4^*$ or $\hat{E}_5^*$) has at least one zero eigenvalue. Hence, $E_2^*$ (or $\bar{E}_0^*$ or $E_{30}^*$ or $E_{40}^*$ or $E_4^*$ or $\hat{E}_5^*$) is a degenerate positive equilibrium for system \eqref{1.5} when $\Lambda_0=\hat{\Lambda}_0$, $\gamma=\hat{\gamma}$ and $k>1$.

When $\eta\neq\hat{\eta}$, we have $\eta z+p\eta z^k-1\neq0$, i.e., $E_2^*$ (or $\bar{E}_0^*$ or $E_{30}^*$ or $E_{40}^*$ or $E_4^*$ or $\hat{E}_5^*$) has a unique zero eigenvalue. We make the following transformations successively
\begin{equation*}
\begin{array}{ll}
\begin{aligned}
x&=X+z,\quad y=Y+\eta z;\\[2ex]
X&=\frac{1}{\eta}u+(z+pz^k)v,\quad Y=u+v,\quad t=\frac{\tau}{p\eta z^k+\eta z-1},
\end{aligned}
\end{array}
\end{equation*}
and system \eqref{1.5} can be written as (still denote $\tau$ by $t$)
\begin{equation}
\begin{array}{ll}\label{3.2}
\left\{
\begin{aligned}
\dot{u}&=\xi_{20}u^2+\xi_{11}uv+\xi_{02}v^2+\xi_{30}u^3+\xi_{21}u^2v+\xi_{12}uv^2+\xi_{03}v^3+\mathcal{O}(|u,v|^{4}),\\[2ex]
\dot{v}&=v-\xi_{20}u^2-\xi_{11}uv-\xi_{02}v^2-\xi_{30}u^3-\xi_{21}u^2v-\xi_{12}uv^2-\xi_{03}v^3+\mathcal{O}(|u,v|^{3}),
\end{aligned}
\right.
\end{array}
\end{equation}
where $\xi_{20}=-\frac{(1+\eta)(k-2-kpz^{k-1})}{2\eta(\eta z+p\eta z^k-1)^2}$, $\xi_{30}=\frac{k(1+\eta)(2kpz^k-pz^k-kz+2z)}{6z^2\eta^2(p\eta z^k+\eta z-1)^2}$ and other $\xi_{ij}$ ($i+j\leq3$) are omitted here.

(i) For $\eta\neq\hat{\eta}$ and $1<k<2$, we have $\xi_{20}>0$;

(ii) For $k=2$, we have $\xi_{20}>0$;

(iii) For $p\neq\hat{p}$ and $k>2$, we have $\xi_{20}\neq0$.

From the center manifold theorem, it follows that
\begin{equation}
\begin{array}{ll}\label{3.3}
\begin{aligned}
\dot{u}=\xi_{20}u^2+O(|u|^3).
\end{aligned}
\end{array}
\end{equation}
By Theorem 7.1 of  Chapter 2 in \cite{ZTW}, $E_2^*$ (or $\bar{E}_0^*$ or $E_{30}^*$ or $E_{40}^*$ or $E_4^*$ or $\hat{E}_5^*$) is a saddle-node.

(iv) For $p=\hat{p}$, $k>2$ and $\eta\neq\bar{\eta}_1$, i.e., $\xi_{20}=0$ and $E_2^*$ (or $\bar{E}_0^*$ or $E_{30}^*$ or $E_{40}^*$ or $E_4^*$ or $\hat{E}_5^*$) has a unique zero eigenvalue. Besides, $\xi_{30}=\frac{(k-1)(k-2)(1+\eta)k^2}{6\eta^2z(2k\eta z-2\eta z-k)^2}>0$. Through the center manifold theorem, we have
\begin{equation}
\begin{array}{ll}\label{2.3}
\begin{aligned}
\dot{u}=\xi_{30}u^3+O(|u|^4).
\end{aligned}
\end{array}
\end{equation}
By Theorem 7.1 of Chapter 2 in \cite{ZTW}, $E_4^*$ is a degenerate node for $p=\hat{p}$ and $\eta\neq\bar{\eta}_1$.
\end{proof}

For $k>1$ and $\eta=\hat{\eta}$, $\Lambda_0=\check{\Lambda}_0:=\frac{(1+z+pz^k)(z+kpz^k)}{pz^{k-1}(k-1)(z+pz^k)}$ and $\gamma=\check{\gamma}:=\frac{(1+z+pz^k)(z+pz^k)}{pz^k(k-1)}>\hat{\eta}$. Through calculation, we detect that
 $J(E_2^*)$ (or $J(\bar{E}_0^*)$ or $J(E_{30}^*)$ or $J(E_{40}^*)$ or $J(E_4^*)$ or $J(\hat{E}_5^*)$) has two zero eigenvalues if and only if $\eta=\hat{\eta}$. Making the following transformations successively
\begin{equation*}
\begin{array}{ll}
\begin{aligned}
x&=u+z,\quad y=v+\frac{1}{z+pz^k}z;\\[2ex]
u&=(z+pz^k)X+(z+pz^k)Y,\quad v=X,
\end{aligned}
\end{array}
\end{equation*}
system \eqref{1.5} can be changed into
\begin{equation}
\begin{array}{ll}\label{3.4}
\left\{
\begin{aligned}
\dot{X}&=Y+\mathcal{O}(|X,Y|^{5}),\\[2ex]
\dot{Y}&=\sum\limits_{2\leq i+j\leq5}a_{ij}X^{i}Y^{j}+\mathcal{O}(|X,Y|^{5}),
\end{aligned}
\right.
\end{array}
\end{equation}
where $a_{ij}$ are omitted here for brevity.

Setting $X=X_1+\frac{a_{02}}{2}X_1^2$ and $Y=Y_1+a_{02}X_1Y_1$, system \eqref{3.4} can be written as
\begin{equation}
\begin{array}{ll}\label{3.5}
\left\{
\begin{aligned}
\dot{X}_1&=Y_1+\mathcal{O}(|X_1,Y_1|^{5}),\\[2ex]
\dot{Y}_1&=b_{20}X_1^2+b_{11}X_1Y_1+\sum\limits_{3\leq i+j\leq5}b_{ij}X_1^{i}Y_1^{j}+\mathcal{O}(|X_1,Y_1|^{5}),
\end{aligned}
\right.
\end{array}
\end{equation}
where $b_{20}=\frac{(1+z+pz^k)(kz-2z-kpz^k)}{2z}$, $b_{11}=\frac{(k-1)z+(k-2)z^2-kp^2z^{2k}-2pz^{k+1}}{z}$.
\begin{theorem}\label{t2}
If $b_{20}b_{11}\neq0$, then $E_2^*$ (or $\bar{E}_0^*$ or $E_{30}^*$ or $E_{40}^*$ or $E_4^*$ or $\hat{E}_5^*$)  is a cusp of codimension 2.
\end{theorem}

If $b_{20}b_{11}=0$, the following statement holds for $E_2^*$ (or $\bar{E}_0^*$ or $E_{30}^*$ or $E_{40}^*$ or $E_4^*$ or $\hat{E}_5^*$).
\begin{lemma}\label{l5}
System \eqref{3.5} is locally topologically equivalent to
\begin{equation}
\begin{array}{ll}\label{3.6}
\left\{
\begin{aligned}
\dot{X}_4&=Y_4+\mathcal{O}(|X_4,Y_4|^{5}),\\[2ex]
\dot{Y}_4&=k_{20}X_4^2+k_{11}X_4Y_4+k_{30}X_4^3+k_{21}X_4^2Y_4+k_{40}X_4^4+k_{31}X_4^3Y_4\\[2ex]
&\quad+k_{50}X_4^5+k_{41}X_4^4Y_4+\mathcal{O}(|X_4,Y_4|^{5}),
\end{aligned}
\right.
\end{array}
\end{equation}
where $k_{ij}$ can be expressed by $b_{ij}$.
\end{lemma}
\begin{proof}
Letting $X_1=X_2+\frac{b_{12}}{6}X_2^3+\frac{b_{03}}{2}X_2^2Y_2$ and $Y_1=Y_2+\frac{b_{12}}{2}X_2^2Y_2+b_{03}X_2Y_2^2$, system \eqref{3.5} can be transformed into
\begin{equation}
\begin{array}{ll}\label{3.7}
\left\{
\begin{aligned}
\dot{X}_2&=Y_2+c_{40}X_2^4+c_{31}X_2^3Y_2+c_{50}X_2^5+c_{41}X_2^4Y_2+\mathcal{O}(|X_2,Y_2|^{5}),\\[2ex]
\dot{Y}_2&=d_{20}X_2^2+d_{11}X_2Y_2+d_{30}X_2^3+d_{21}X_2^2Y_2+\sum\limits_{4\leq i+j\leq5}d_{ij}X_2^{i}Y_2^{j}+\mathcal{O}(|X_2,Y_2|^{5}),
\end{aligned}
\right.
\end{array}
\end{equation}
where
\begin{equation*}
\begin{array}{ll}
\begin{aligned}
c_{40}&=-\frac{b_{20}b_{03}}{2},\quad c_{31}=-\frac{b_{11}b_{03}}{2},\quad c_{50}=-\frac{b_{30}b_{03}}{2},\quad c_{41}=-\frac{b_{21}b_{03}}{2},\quad d_{20}=b_{20},\quad d_{11}=b_{11},\\[2ex]
d_{30}&=b_{30},\quad d_{21}=b_{21},\quad d_{40}=b_{40}-\frac{b_{20}b_{12}}{6},\quad d_{13}=b_{13},\quad d_{31}=b_{31}-b_{20}b_{03}+\frac{b_{11}b_{12}}{6},\\[2ex]
d_{22}&=b_{22}-\frac{b_{11}b_{03}}{2},\quad d_{04}=b_{04},\quad d_{50}=b_{50},\quad d_{41}=b_{41}-\frac{b_{30}b_{03}}{2}+\frac{b_{21}b_{12}}{3},\quad d_{14}=b_{14}+3b_{03}^2,\\[2ex]
d_{32}&=b_{32}+\frac{7b_{12}^2}{6},\quad d_{23}=b_{23}+4b_{03}b_{12},\quad d_{05}=b_{05}.
\end{aligned}
\end{array}
\end{equation*}

Making $X_2=X_3+\frac{3c_{31}+d_{22}}{12}X_3^4+\frac{d_{13}}{6}X_3^3Y_3+\frac{d_{04}}{2}X_3^2Y_2^2$ and $Y_2=Y_3-c_{40}X_3^4+\frac{d_{22}}{3}X_3^3Y_3+d_{04}X_3Y_3^3+\frac{d_{13}}{2}X_3^2Y_3^2$, system \eqref{3.7} becomes
\begin{equation}
\begin{array}{ll}\label{3.8}
\left\{
\begin{aligned}
\dot{X}_3&=Y_3+e_{50}X_3^5+e_{41}X_3^4Y_3+e_{32}X_3^3Y_3^2+\mathcal{O}(|X_3,Y_3|^{5}),\\[2ex]
\dot{Y}_3&=h_{20}X_3^2+h_{11}X_3Y_3+h_{30}X_3^3+h_{21}X_3^2Y_3+h_{40}X_3^4+h_{31}X_3^3Y_3\\[2ex]
&\quad+\sum\limits_{i+j=5}h_{ij}X_3^{i}Y_3^{j}+\mathcal{O}(|X_3,Y_3|^{5}),
\end{aligned}
\right.
\end{array}
\end{equation}
where
\begin{equation*}
\begin{array}{ll}
\begin{aligned}
e_{50}&=c_{50}-\frac{d_{20}d_{13}}{6},\quad e{41}=c_{41}-d_{20}d_{04}-\frac{d_{11}d_{13}}{6},\quad e_{32}=-d_{11}d_{04},\quad h_{20}=d_{20},\\[2ex]
h_{11}&=d_{11},\quad h_{30}=d_{30},\quad h_{21}=d_{21},\quad h_{40}=d_{40},\quad h_{31}=d_{31}+4c_{40},\quad h_{05}=d_{05},\\[2ex]
h_{50}&=d_{50}-c_{40}d_{11}-\frac{d_{20}}{6}(3c_{31}-d_{22}),\quad h_{41}=d_{41}+\frac{d_{11}}{12}(3c_{31}+d_{22})-\frac{2d_{20}d_{13}}{3},\\[2ex]
h_{14}&=d_{14},\quad h_{23}=d_{23}-\frac{3d_{11}d_{04}}{2},\quad h_{32}=d_{32}-2d_{20}d_{04}-\frac{d_{11}d_{13}}{3}.
\end{aligned}
\end{array}
\end{equation*}

Setting $X_3=X_4+\frac{4e_{41}+h_{32}}{20}X_4^5+\frac{3e_{32}+h_{23}}{12}X_4^4Y_4+\frac{h_{05}}{4}X_4^2Y_4^3+\frac{h_{14}}{6}X_4^3Y_4^2$ and $Y_3=Y_4-e_{50}X_4^5+\frac{h_{32}}{4}X_4^4Y_4+\frac{h_{14}}{2}X_4^2Y_4^3+\frac{h_{23}}{3}X_4^3Y_4^2+h_{05}X_4Y_4^4$, system \eqref{3.8} can be written as
\begin{equation}
\begin{array}{ll}\label{3.9}
\left\{
\begin{aligned}
\dot{X}_4&=Y_4+\mathcal{O}(|X_4,Y_4|^{5}),\\[2ex]
\dot{Y}_4&=k_{20}X_4^2+k_{11}X_4Y_4+k_{30}X_4^3+k_{21}X_4^2Y_4+k_{40}X_4^4+k_{31}X_4^3Y_4\\[2ex]
&\quad+k_{50}X_4^5+k_{41}X_4^4Y_4+\mathcal{O}(|X_4,Y_4|^{5}),
\end{aligned}
\right.
\end{array}
\end{equation}
where
\begin{equation*}
\begin{array}{ll}
\begin{aligned}
k_{20}&=h_{20},\quad k_{11}=h_{11},\quad k_{30}=h_{30},\quad k_{21}=h_{21},\quad k_{40}=h_{40},\\[2ex]
k_{31}&=h_{31},\quad k_{50}=h_{50},\quad k_{41}=h_{41}+5e_{50}.
\end{aligned}
\end{array}
\end{equation*}
\end{proof}

For $b_{20}\neq0$ and $b_{11}=0$, i.e. $p=\check{p}:=\frac{\sqrt{z(k-z+kz)(k-1)}-z}{kz^k}$ and $k>\frac{1+2z}{1+z}$, $E_2^*$ (or $\bar{E}_0^*$ or $E_{30}^*$ or $E_{40}^*$ or $E_4^*$ or $\hat{E}_5^*$) is a nilpotent cusp of codimension at least 3 for system \eqref{1.5}.

Let
\begin{equation*}
\begin{array}{ll}
\begin{aligned}
\mathcal{G}&=-\frac{(k-1)\big((k-1)\sqrt{z}+\sqrt{(k-1)(k-z+kz)}\big)\sqrt[4]{z(k-1)(k-z+kz)}(\mathcal{G}_1+\mathcal{G}_2)}{3\sqrt{2}k^3z\big(k-z+kz-\sqrt{(k-1)z(k-z+kz)}\big)},\\[2ex]
\mathcal{F}&=-\frac{(k-1)^3\big((k-1)z+\sqrt{(k-1)z(k-z+kz)}\big)^2\sqrt[4]{z(k-1)(k-z+kz)}(\mathcal{F}_1+\mathcal{F}_2)}{36\sqrt{2}k^3z^{3/2}\big(k-z+kz-\sqrt{(k-1)z(k-z+kz)}\big)\big(z-kz+\sqrt{(k-1)z(k-z+kz)}\big)^3},\\[2ex]
\tilde{z}&=\frac{2-7k+7k^2-2k^3+\sqrt{2}\sqrt{2-17k+58k^2-101k^3+94k^4-44k^5+8k^6}}{6(k^3-4k^2+5k-2)},
\end{aligned}
\end{array}
\end{equation*}
where
\begin{equation*}
\begin{array}{ll}
\begin{aligned}
\mathcal{G}_1&=(k-1)\big((2z^2-2)k^3-(10z^2+5z-1)k^2+(16z^2+4z)k-8z^2\big),\\[2ex]
\mathcal{G}_2&=\sqrt{z(k-1)(k-z+kz)}\big((2z+2)k^3-(10z+1)k^2+16kz-8z\big),\\[2ex]
\mathcal{F}_1&=-z\big((26z^2+68z+42)k^5+(8z^2-79z-71)k^4-(214z^2+59z+35)k^3\\[2ex]
&\quad+(380z^2+106z-5)k^2-(280z^2+36z)k+80z^2\big),\\[2ex]
\mathcal{F}_2&=\sqrt{z(k-1)(k-z+kz)}\big((46z^2+62z+16)k^4-(106z^2+67z+22)k^3\\[2ex]
&\quad+(180z^2+10z+7)k^2-(200z^2-4z)k+80z^2\big).
\end{aligned}
\end{array}
\end{equation*}
\begin{lemma}\label{l6}
For $p=\check{p}$, then $\Lambda_0=\tilde{\Lambda}_0:=\frac{kz(k-z+kz)}{(k-1)\big(\sqrt{(k-1)(k-z+kz)z}-z\big)}$, $\eta=\tilde{\eta}:=\frac{k}{\sqrt{(k-1)(k-z+kz)z}-z+kz}$ and $\gamma=\tilde{\gamma}:=\frac{\big(kz-z+\sqrt{(k-1)(k-z+kz)z}\big)\big(k-z+kz+\sqrt{(k-1)(k-z+kz)z}\big)}{(k-1)k\big(\sqrt{(k-1)(k-z+kz)z}-z\big)}$, system \eqref{1.5} is locally topologically equivalent to
\begin{equation}
\begin{array}{ll}\label{3.10}
\left\{
\begin{aligned}
\dot{x}&=y,\\[2ex]
\dot{y}&=x^2+\mathcal{G}x^3y+\mathcal{F}x^4y+\mathcal{O}(|x,y|^{5}),
\end{aligned}
\right.
\end{array}
\end{equation}
where $\mathcal{G}$ and $\mathcal{F}$ are expressed by $b_{ij}$.
\end{lemma}
\begin{proof}
Step 1: $k_{20}=b_{20}\neq0$ when $p=\check{p}$. Setting $X_4=x+\frac{k_{21}}{3k_{20}}xy+\frac{5k_{21}^2}{54k_{20}}x^4$, $Y_4=y+\frac{k_{21}}{3k_{20}}y^2+\frac{k_{21}}{3}x^3+\frac{k_{21}k_{30}}{3k_{20}}x^4+\frac{10k_{21}^2}{27k_{20}}x^3y$, system \eqref{3.9} can be transformed into
\begin{equation}
\begin{array}{ll}\label{3.11}
\left\{
\begin{aligned}
\dot{x}&=y+r_{50}x^5+r_{41}x^4y+O(|x,y|^5),\\[2ex]
\dot{y}&=s_{20}x^2+s_{30}x^3+s_{40}x^4+s_{31}x^3y+s_{50}x^5\\[2ex]
&\quad+s_{41}x^4y+s_{32}x^3y^2+s_{23}x^2y^3+O(|x,y|^5),
\end{aligned}
\right.
\end{array}
\end{equation}
where
\begin{equation*}
\begin{array}{ll}
\begin{aligned}
r_{50}&=-\frac{k_{11}k_{21}^2+3k_{21}k_{40}}{9k_{20}},\quad r_{41}=\frac{k_{21}(k_{21}k_{30}-k_{20}k_{31})}{3k_{20}^2},\quad s_{31}=k_{31}-\frac{k_{21}k_{30}}{k_{20}},\\[2ex]
s_{20}&=k_{20},\quad s_{30}=k_{30},\quad s_{50}=k_{50}+\frac{k_{11}k_{21}k_{30}}{3k_{20}}+\frac{4k_{21}^2}{27},\quad s_{41}=k_{41}+\frac{2k_{21}k_{40}}{3k_{20}},\\[2ex]
s_{40}&=k_{40},\quad s_{32}=\frac{k_{21}(3k_{21}k_{30}+2k_{20}k_{31})}{3k_{20}^2},\quad s_{23}=\frac{k_{21}^3}{3k_{20}^2}.
\end{aligned}
\end{array}
\end{equation*}

Step 2: Letting $x=X+\frac{4r_{41}+s_{32}}{20}X^5+\frac{s_{23}}{12}X^4Y$, $y=Y-r_{50}X^5+\frac{s_{32}}{4}X^4Y+\frac{s_{23}}{3}X^3Y^2$, system \eqref{3.11} can be written as
\begin{equation}
\begin{array}{ll}\label{3.12}
\left\{
\begin{aligned}
\dot{X}&=Y+O(|X,Y|^5),\\[2ex]
\dot{Y}&=w_{20}X^2+w_{30}X^3+w_{40}X^4+w_{31}X^3Y+w_{50}X^5+w_{41}X^4Y+O(|X,Y|^5),
\end{aligned}
\right.
\end{array}
\end{equation}
where
\begin{equation*}
\begin{array}{ll}
\begin{aligned}
w_{20}&=s_{20},\quad w_{30}=s_{30},\quad w_{40}=s_{40},\quad w_{31}=s_{31},\quad w_{50}=s_{50},\quad w_{41}=s_{41}+5r_{50}.
\end{aligned}
\end{array}
\end{equation*}

Step 3: Making the following transformation
\begin{equation*}
\begin{array}{ll}
\begin{aligned}
X&=x-\frac{w_{30}}{4w_{20}}x^2+\frac{15w_{30}^2-16w_{20}w_{40}}{80w_{20}^2}x^3-\frac{175w_{30}^3-336w_{20}w_{30}w_{40}+160w_{20}^2w_{50}}{960w_{20}^3}x^4,\\[2ex]
Y&=y,\\[2ex] t&=(1-\frac{w_{30}}{2w_{20}}x+\frac{45w_{30}^2-48w_{20}w_{40}}{80w_{20}^2}x^2-\frac{175w_{30}^3-336w_{20}w_{30}w_{40}+160w_{20}^2w_{50}}{240w_{20}^3}x^3)\tau.
\end{aligned}
\end{array}
\end{equation*}
System \eqref{3.12} can be changed into (still denote $\tau$ by $t$)
\begin{equation}
\begin{array}{ll}\label{3.13}
\left\{
\begin{aligned}
\dot{x}&=y+O(|x,y|^5),\\[2ex]
\dot{y}&=p_{20}x^2+p_{31}x^3y+p_{41}x^4y+O(|x,y|^5),
\end{aligned}
\right.
\end{array}
\end{equation}
where
\begin{equation*}
\begin{array}{ll}
\begin{aligned}
v_{20}=w_{20},\quad v_{31}=w_{31},\quad v_{41}=w_{41}-\frac{5w_{30}w_{31}}{4w_{20}}.
\end{aligned}
\end{array}
\end{equation*}

Step 4: From above analysis, we have $v_{20}=b_{20}=-\frac{\sqrt{z(kz+k-z)(k-1)}}{2z}<0$, setting $x=-X$,$y=-\sqrt{-v_{20}}Y$ and $t=\frac{1}{\sqrt{-v_{20}}}\tau$, system \eqref{3.13} can be written as (still denote $\tau$ by $t$)
\begin{equation}
\begin{array}{ll}\label{3.14}
\left\{
\begin{aligned}
\dot{X}&=Y+O(|X,Y|^5),\\[2ex]
\dot{Y}&=X^2+\mathcal{G}X^3Y+\mathcal{F}X^4Y+O(|X,Y|^5),
\end{aligned}
\right.
\end{array}
\end{equation}

(i) For $E_2^*$, we have $\mathcal{G}>0$ when $\frac{1+2z}{1+z}<k<2$, thus $E_2^*$ is a cusp of codimension 3.

(ii) For $\bar{E}_0^*$, we have $\mathcal{G}=\frac{1}{2\sqrt{2z}\sqrt[4]{z(2+z)}}>0$, thus $\bar{E}_0^*$ is a cusp of codimension 3.

(iii) For $E_{30}^*$ (or $E_{40}^*$ or $E_4^*$ or $\hat{E}_5^*$), if $\mathcal{G}\neq0$, then $E_{30}^*$ (or $E_{40}^*$ or $E_4^*$ or $\hat{E}_5^*$) is a a cusp of codimension. If $\mathcal{G}=0$, i.e., $z=\tilde{z}$, then $\mathcal{F}<0$ when $2<k<5.5745$, i.e., $E_{30}^*$ (or $E_{40}^*$ or $E_4^*$ or $\hat{E}_5^*$) is a cusp of codimension 4.
\end{proof}
\begin{theorem}\label{t3}
For system \eqref{1.5}, the following statements hold.

(i) If $\frac{1+2z}{1+z}<k\leq2$, then  $E_2^*$ ($\bar{E}_0^*$) is a cusp of codimension 3;

(ii) If $k>2$ and $z\neq\tilde{z}$, then $\mathcal{G}\neq0$, i.e., $E_{30}^*$ (or $E_{40}^*$ or $E_4^*$ or $\hat{E}_5^*$) is a cusp of codimension 3. If $z=\tilde{z}$, then $\mathcal{G}=0$ and $\mathcal{F}<0$ when $2<k<5.5745$, i.e., $E_{30}^*$ (or $E_{40}^*$ or $E_4^*$ or $\hat{E}_5^*$) is a cusp of codimension 4.
\end{theorem}
When $k>2$, $\eta=\bar{\eta}_1$ and $b_{20}=0$, i.e., $p=\hat{p}$, we further investigate the degenerate equilibrium $E_4^*$.

\begin{theorem}\label{t4}
For $\Lambda_0=\tilde{\bar{\Lambda}}_0:=\frac{k(k-2x_4^*+2kx_4^*)}{2(k-2)(k-1)}$, $\gamma=\tilde{\bar{\gamma}}:=\frac{2(k-2x_4^*+2kx_4^*)}{k(k-2)}$, $p=\hat{p}$, $\eta=\bar{\eta}_1$, $k>2$ and $x_4^*>\frac{k(\sqrt{2k-3}-1)}{4(k-1)}$, system \eqref{1.5} has a unique positive equilibrium $E_4^*$ and system \eqref{1.5} is locally topologically equivalent to
\begin{equation}
\begin{array}{ll}\label{3.19}
\left\{
\begin{aligned}
\dot{x}&=y+O(|x,y|^5),\\[2ex]
\dot{y}&=\mathcal{M}xy-x^3+x^2y+\mathcal{N}x^3y+O(|x,y|^4),
\end{aligned}
\right.
\end{array}
\end{equation}
where
\begin{equation*}
\begin{array}{ll}
\begin{aligned}
\mathcal{M}&=-\frac{\sqrt{3}k}{\sqrt{(k-2)(k-2x_4^*+2kx_4^*)}},\\[2ex]
\mathcal{N}&=\frac{\sqrt{(k-2)(k-2x_4^*+2kx_4^*)}(195k^3x_4^*-665k^2x_4^*+630kx_4^*-160x_4^*+44k^3-87k^2-6k)}{3\sqrt{3}(10k^2x_4^*-30kx_4^*+20x_4^*+2k^2-3k)^2}.
\end{aligned}
\end{array}
\end{equation*}
Besides,

(i) For $2<k\leq k_0$ ($k_0\approx 2.80425$) and $x_4^*>\frac{16k-5k^2}{16k^2-48k+32}$ or $k>k_0$ and $x_4^*>\frac{k(\sqrt{2k-3}-1)}{4(k-1)}$,  $E_4^*$ is a nilpotent focus of codimension 3;

(ii) For $2<k\leq k_0$ and $\frac{k(\sqrt{2k-3}-1)}{4(k-1)}<x_4^*<\frac{16k-5k^2}{16k^2-48k+32}$,  $E_4^*$  is a nilpotent elliptic equilibrium of codimension 3;

(iii) For $x_4^*=\frac{k(16-5k)}{(k-1)(k-2)}$ and $2<k<k_0$,  $E_4^*$ is a nilpotent elliptic equilibrium of codimension at least 4, the phase portrait is given in Fig.\ref{w10}.
\end{theorem}
\begin{figure}
\begin{center}
\begin{overpic}[scale=0.50]{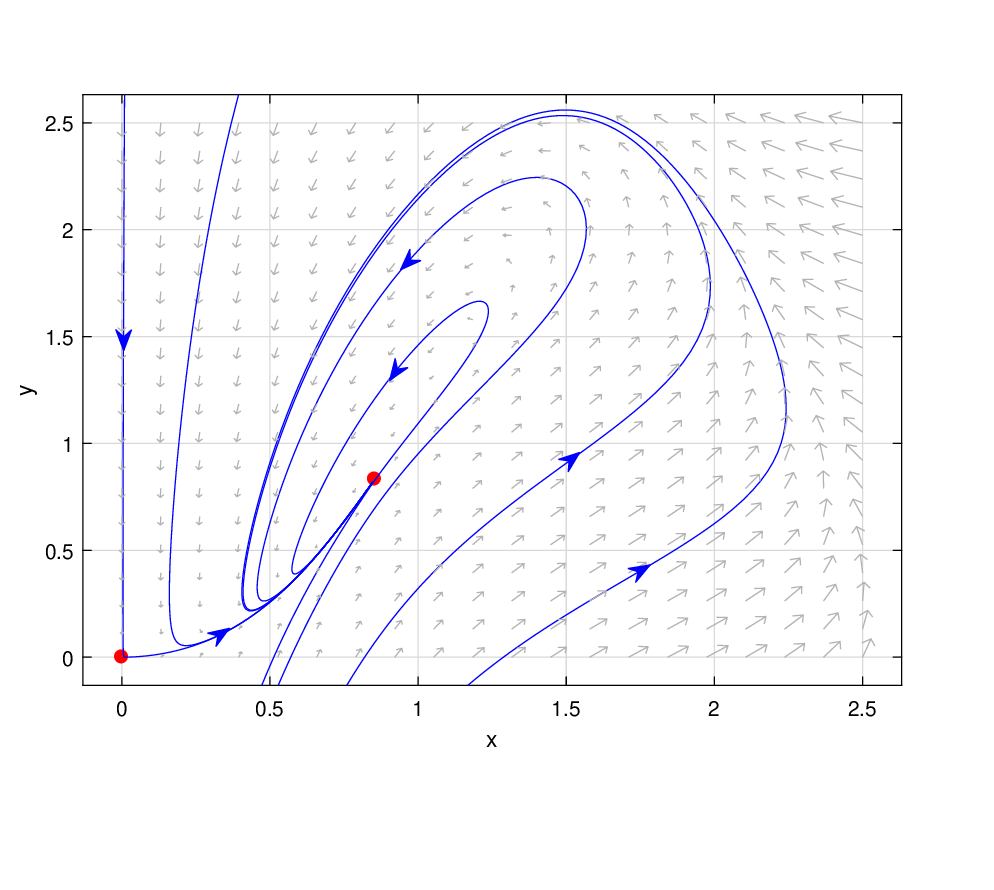}
\end{overpic}
\put(-220,61){$E_0$}
\put(-150,93){$E_4^*$}
\vspace{-15mm}
\end{center}
\caption{The local dynamics of system \eqref{1.5} with an elliptic type endemic equilibrium $E_4^*$ of codimension at least 4, where $p=\frac{192\sqrt{3}}{175\sqrt{35}}$, $\Lambda_0=\frac{125}{16}$, $\gamma=\frac{15}{2}$, $\eta=\frac{8}{7}$ and $k=\frac{5}{2}$.\label{w10}}
\end{figure}
\begin{proof}
If $H(x_4^*)=H'(x_4^*)=H''(x_4^*)=0$ and $b_{20}=0$, we have $\Lambda_0=\tilde{\bar{\Lambda}}_0$, $\gamma=\tilde{\bar{\gamma}}_0$, $p=\hat{p}$ and $\eta=\bar{\eta}_1$. Making the following transformations successively
\begin{equation*}
\begin{array}{ll}
\begin{aligned}
x&=u+x_4^*,\quad y=v+\frac{kx_4^*}{2x_4^*(k-1)};\\[2ex]
u&=\frac{2x_4^*(k-1)}{k}X+\frac{2x_4^*(k-1)}{k}Y,\quad v=X,
\end{aligned}
\end{array}
\end{equation*}
system \eqref{1.5} can be changed into
\begin{equation}
\begin{array}{ll}\label{3.20}
\left\{
\begin{aligned}
\dot{X}&=Y+\mathcal{O}(|X,Y|^{5}),\\[2ex]
\dot{Y}&=\sum\limits_{2\leq i+j\leq5}\tilde{a}_{ij}X^{i}Y^{j}+\mathcal{O}(|X,Y|^{5}),
\end{aligned}
\right.
\end{array}
\end{equation}
where $\tilde{a}_{ij}$ are omitted here for brevity.

Similar to Lemma \ref{l5}, system \eqref{3.20} can be transformed into system \eqref{3.6}. From above analysis, we know that $k_{30}=b_{30}\neq0$, letting $X_4=X-\frac{k_{40}}{5k_{30}}$, $Y_4=Y$ and $t=(1-\frac{2k_{40}}{5k_{30}})\tau$ (still denote $\tau$ by $t$), system \eqref{3.6} can be transformed into
\begin{equation}
\begin{array}{ll}\label{3.21}
\left\{
\begin{aligned}
\dot{X}&=Y+O(|X,Y|^4),\\[2ex]
\dot{Y}&=q_{11}XY+q_{30}X^3+q_{21}X^2Y+q_{31}X^3Y+O(|X,Y|^4),
\end{aligned}
\right.
\end{array}
\end{equation}
where
\begin{equation*}
\begin{array}{ll}
\begin{aligned}
q_{11}=k_{11},\quad q_{30}=k_{30}-\frac{4k_{20}k_{40}}{5k_{30}},\quad q_{21}=k_{21}-\frac{3k_{11}k_{40}}{5k_{30}},\quad q_{31}=k_{31}-\frac{4k_{21}k_{40}}{5k_{30}}+\frac{2k_{11}k_{40}^2}{25k_{30}^2}.
\end{aligned}
\end{array}
\end{equation*}

Notice $q_{30}<0$ and $q_{21}\neq0$, making $X=\frac{\sqrt{-q_{30}}}{q_{21}}x$, $Y=\frac{q_{30}\sqrt{-q_{30}}}{q_{21}^2}y$ and $t=\frac{q_{21}}{q_{30}}\tau$ (still denote $\tau$ by $t$), system \eqref{3.21} becomes
\begin{equation}
\begin{array}{ll}\label{2.22}
\left\{
\begin{aligned}
\dot{x}&=y+O(|x,y|^4),\\[2ex]
\dot{y}&=\mathcal{M}xy-x^3+x^2y+\mathcal{N}x^3y+O(|x,y|^4),
\end{aligned}
\right.
\end{array}
\end{equation}
where
\begin{equation*}
\begin{array}{ll}
\begin{aligned}
\mathcal{M}&=-\frac{\sqrt{3}k}{\sqrt{(k-2)(k-2x_4^*+2kx_4^*)}},\\[2ex]
\mathcal{N}&=\frac{\sqrt{(k-2)(k-2x_4^*+2kx_4^*)}(195k^3x_4^*-665k^2x_4^*+630kx_4^*-160x_4^*+44k^3-87k^2-6k)}{3\sqrt{3}(10k^2x_4^*-30kx_4^*+20x_4^*+2k^2-3k)^2}.
\end{aligned}
\end{array}
\end{equation*}
From above analysis, we have $p=\hat{p}$ when $b_{20}=0$. Through calculation, we obtain $b_{11}=k-1\neq0$, $b_{30}=\frac{(k-1)^2(k-2)(2x_4^*-k-2kx_4^*)}{3k^2}<0$ and $b_{11}^2+8b_{30}=-\frac{(k-1)^2(16k^2x_4^*-48kx_4^*+32x_4^*+5k^2-16k)}{3k^2}$ when $k>2$ and $x_4^*>\frac{k(\sqrt{2k-3}-1)}{4(k-1)}$.

(i) For $2<k\leq k_0$ ($k_0\approx 2.80425$) and $x_4^*>\frac{16k-5k^2}{16k^2-48k+32}$ or $k>k_0$ and $x_4^*>\frac{k(\sqrt{2k-3}-1)}{4(k-1)}$, we have $b_{30}<0$ and $b_{11}^2+8b_{30}<0$;

(ii) For $2<k\leq k_0$ and $\frac{k(\sqrt{2k-3}-1)}{4(k-1)}<x_4^*<\frac{16k-5k^2}{16k^2-48k+32}$, we have $b_{30}<0$ and $b_{11}^2+8b_{30}>0$;

(iii) For $x_4^*=\frac{k(16-5k)}{(k-1)(k-2)}$ and $2<k<k_0$, we have $b_{30}=-\frac{1}{8}(k-1)^2<0$ and $b_{11}^2+8b_{30}=0$.

By Khibnik, Krauskopf and Rousseau \cite{CB}, Dumortier, Fiddelaers and Li \cite{CP} and Lemma 3.1 in \cite{DG}, we obtain the results.
\end{proof}

\section{Nilpotent bifurcations}
In this section, we mainly focus on nilpotent bifurcations  for system \eqref{1.5}.
\subsection{Saddle-node bifurcation}
It follows from Theorem \ref{t1} that
\begin{equation*}
\begin{array}{ll}
\begin{aligned}
SN_1&=\bigg\{(\Lambda_0,\gamma,\eta,p,z,k):\Lambda_0=\hat{\Lambda}_0:=\frac{(1+\eta)(z+kpz^k)}{z^{k-1}p(k-1)}, \gamma=\hat{\gamma}:=\frac{(1+\eta)(z+pz^k)^2}{z^{k}p(k-1)}>\eta,\\[2ex]
&\qquad\eta\neq\frac{1}{z+pz^k}, z>0, p>0, 1<k<2\bigg\}
\end{aligned}
\end{array}
\end{equation*}
\begin{equation*}
\begin{array}{ll}
\begin{aligned}
SN_2&=\bigg\{(\Lambda_0,\gamma,\eta,p,z,k):\Lambda_0=\hat{\Lambda}_0:=\frac{(1+\eta)(z+2pz^2)}{pz}, \gamma=\hat{\gamma}:=\frac{(1+\eta)(z+pz^2)^2}{pz}>\eta,\\[2ex]
&\qquad\eta\neq\frac{1}{z+pz^2}, z>0, p>0, k=2\bigg\}
\end{aligned}
\end{array}
\end{equation*}
\begin{equation*}
\begin{array}{ll}
\begin{aligned}
SN_3&=\bigg\{(\Lambda_0,\gamma,\eta,p,z,k):\Lambda_0=\hat{\Lambda}_0:=\frac{(1+\eta)(z+kpz^k)}{z^{k-1}p(k-1)}, \gamma=\hat{\gamma}:=\frac{(1+\eta)(z+pz^k)^2}{z^{k}p(k-1)}>\eta,\\[2ex]
&\qquad\eta\neq\frac{1}{z+pz^k}, z>0, p\neq\frac{k-2}{kz^{k-1}}, k>2\bigg\}
\end{aligned}
\end{array}
\end{equation*}
are saddle-node bifurcation surfaces. When the parameters vary from one side of the surface $SN_i\ (i=1,2,3)$ to the other side, the number of equilibria of system \eqref{1.5} changes from zero to two, and the two equilibria are a hyperbolic saddle and a node. When $\hat{\Lambda}_0$ and $\hat{\gamma}$ vary, infectious diseases will disappear when $R_0=\frac{\hat{\Lambda}_0}{\hat{\gamma}}<1$ and infectious diseases will form endemic diseases and persist when $\frac{\hat{\Lambda}_0}{\hat{\gamma}}>1$.
\begin{theorem}\label{l7}
System \eqref{1.5} can undergo saddle-node bifurcation as $(\Lambda_0,\gamma,\eta,p,z,k)$ varies near
 $SN_1$ or $SN_2$ or $SN_3$.
\end{theorem}

\subsection{Bogdanov-Takens bifurcation}
From Theorem \ref{t3} we can know that system \eqref{1.5} may exhibit cusp type Bogdanov-Takens bifurcation of codimension 3 around $E_2^*$ (or $\bar{E}_0^*$) when $1<k\leq2$ and cusp type Bogdanov-Takens bifurcation of codimension 4 around $E_{30}^*$ (or $E_{40}^*$ or $E_4^*$ or $\hat{E}_5^*$) when $k>2$. Here, we have omitted the analysis of the cusp type Bogdanov-Takens bifurcation of codimension 3. In order to make sure if such a bifurcation can be fully unfolding inside the class of system \eqref{1.5}, we choose $\Lambda_0$, $\gamma$, $\eta$ and $p$ as bifurcation parameters, and consider the following system
\begin{equation}
\begin{array}{ll}\label{4.1}
\left\{
\begin{aligned}
\dot{x}&=x(1+(\check{p}+\lambda_1)x^{k-1})(\tilde{\Lambda}_0+\lambda_2-x-y)-(\tilde{\gamma}+\lambda_3)x,\\[2ex]
\dot{y}&=(\tilde{\eta}+\lambda_4)x-y,
\end{aligned}
\right.
\end{array}
\end{equation}
where $\lambda=(\lambda_1,\lambda_2,\lambda_3,\lambda_4)=(0,0,0,0)$. If we can transform system \eqref{3.1} into the following form
\begin{equation}
\begin{array}{ll}\label{4.2}
\left\{
\begin{aligned}
\dot{x}&=y,\\[2ex]
\dot{y}&=\eta_1+\eta_2y+\eta_3xy+\eta_4x^3y+x^2-x^4y+R(x,y,\lambda),
\end{aligned}
\right.
\end{array}
\end{equation}
where
\begin{equation}
\begin{array}{ll}\label{4.3}
\begin{aligned}
R(x,y,\lambda)&=y^2O(|x,y|^2)+O(|x,y|^6)+O(\lambda)\big(O(y^2)+O(|x,y|^3)\big)\\[2ex]
&\quad+O(\lambda^2)O(|x,y|),
\end{aligned}
\end{array}
\end{equation}
and check $\frac{\partial(\eta_1,\eta_2,\eta_3,\eta_4)}{\partial(\lambda_1,\lambda_2,\lambda_3,\lambda_4)}\big|_{\lambda=0}\neq0$, then we claim that system \eqref{1.5} undergoes cusp type Bogdanov-Takens bifurcation of codimension 4 around $E_{30}^*$ (or $E_{40}^*$ or $E_4^*$ or $\hat{E}_5^*$) when $k>2$. The bifurcation diagram of system \eqref{4.1} is sketched in Fig.\ref{w9} and phase portraits of system \eqref{1.5} are given in Fig.\ref{w2} and Fig.\ref{w3}.

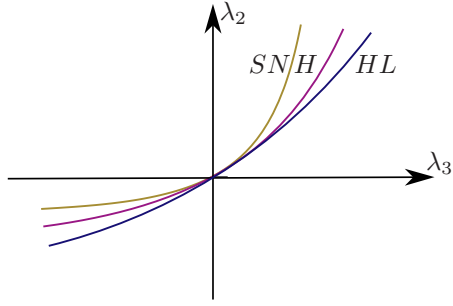
\begin{figure}
\begin{center}
\begin{overpic}[scale=2]{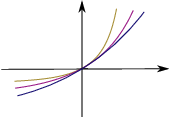}
\end{overpic}
\put(-83,106){$\lambda_2$}
\put(-72,86){$SN$}
\put(-55,86){$H$}
\put(-32,86){$HL$}
\put(-5,49){$\lambda_3$}
\vspace{-4mm}
\end{center}
\centering{\caption{ The bifurcation diagram of system \eqref{1.5} with $\lambda_1=0$, $\lambda_2=0$, $k=3$ and $p=1$. The yellow line is saddle-node bifurcation curve, the pink line is Hopf bifurcation curve and the blue line is homoclinic bifurcation curve.\label{w9}}}
\end{figure}
\begin{theorem}\label{t5}
For $k>2$, system \eqref{1.5} can undergo cusp type Bogdanov-Takens bifurcation of codimension 4 around $E_{30}^*(\tilde{z},\tilde{\eta}\tilde{z})$ (or $E_{40}^*$ or $E_4^*$ or $\hat{E}_5^*$) as $(\Lambda_0,\gamma,\eta,p)$ varies near $(\tilde{\Lambda}_0,\tilde{\gamma},\tilde{\eta},\check{p})$. There exist a series of bifurcation with codimension 2, 3 and 4 originating from $E_{30}^*$ (or $E_{40}^*$ or $E_4^*$ or $\hat{E}_5^*$).

(i) If $\Lambda_0=\frac{(1+z+pz^k)(z+kpz^k)}{pz^{k-1}(k-1)(z+pz^k)}$, $\gamma=\frac{(1+z+pz^k)(z+pz^k)}{pz^k(k-1)}$ and $\eta=\frac{1}{z+pz^k}$, then system \eqref{1.5} can undergo the cusp  bifurcation of codimension 2 near $E_{30}^*$ (or $E_{40}^*$ or $E_4^*$ or $\hat{E}_5^*$) for $p\neq\frac{k-2}{kz^{k-1}}$ or $p\neq\frac{\sqrt{z(k-z+kz)(k-1)}-z}{kz^k}$.

(ii) If $\Lambda_0=\frac{kz(k-z+kz)}{(k-1)\big(\sqrt{(k-1)(k-z+kz)z}-z\big)}$, $\gamma=\frac{\big(kz-z+\sqrt{(k-1)(k-z+kz)z}\big)\big(k-z+kz+\sqrt{(k-1)(k-z+kz)z}\big)}{(k-1)k\big(\sqrt{(k-1)(k-z+kz)z}-z\big)}$, $\eta=\frac{k}{\sqrt{(k-1)(k-z+kz)z}-z+kz}$ and $p=\frac{\sqrt{z(k-z+kz)(k-1)}-z}{kz^k}$, then system \eqref{1.5} can undergo the cusp  bifurcation of codimension 3 near $E_{30}^*$ (or $E_{40}^*$ or $E_4^*$ or $\hat{E}_5^*$) for $z\neq\tilde{z}$.

(iii) If $\Lambda_0=\frac{kz(k-z+kz)}{(k-1)\big(\sqrt{(k-1)(k-z+kz)z}-z\big)}$, $\gamma=\frac{\big(kz-z+\sqrt{(k-1)(k-z+kz)z}\big)\big(k-z+kz+\sqrt{(k-1)(k-z+kz)z}\big)}{(k-1)k\big(\sqrt{(k-1)(k-z+kz)z}-z\big)}$, $\eta=\frac{k}{\sqrt{(k-1)(k-z+kz)z}-z+kz}$ and $p=\frac{\sqrt{z(k-z+kz)(k-1)}-z}{kz^k}$, then system \eqref{1.5} can undergo the cusp  bifurcation of codimension 4 near $E_{30}^*$ (or $E_{40}^*$ or $E_4^*$ or $\hat{E}_5^*$) for $z=\tilde{z}$ and $2<k<5.5745$.
\end{theorem}
\begin{proof}
Firstly, by the transformation $x=X+z$,  $y=Y+\tilde{\eta}z$ and Taylor series theorem, system \eqref{4.1} can be changed into
\begin{equation}
\begin{array}{ll}\label{4.4}
\left\{
\begin{aligned}
\dot{X}&=\bar{a}_{00}+\bar{a}_{10}X+\bar{a}_{01}Y+\bar{a}_{20}X^2+\bar{a}_{11}XY+\bar{a}_{30}X^3+\bar{a}_{21}X^2Y\\[2ex]
&\quad+\bar{a}_{40}X^4+\bar{a}_{31}X^3Y+\bar{a}_{50}X^5+\bar{a}_{41}X^4Y+O(|X,Y|^5),\\[2ex]
\dot{Y}&=\bar{b}_{00}+\bar{b}_{10}X+\bar{b}_{01}Y+O(|X,Y|^5),
\end{aligned}
\right.
\end{array}
\end{equation}
where $\bar{a}_{ij}$ and $\bar{b}_{ij}$ are omitted here for the sake of brevity.

Secondly, setting $u=X$ and $v=\dot{X}$, system \eqref{4.4} can be written as
\begin{equation}
\begin{array}{ll}\label{4.5}
\left\{
\begin{aligned}
\dot{u}&=v,\\[2ex]
\dot{v}&=\bar{c}_{00}+\bar{c}_{10}u+\bar{c}_{01}v+\bar{c}_{20}u^2+\bar{c}_{11}uv+\bar{c}_{02}v^2+\bar{c}_{30}u^3\\[2ex]
&\quad+\bar{c}_{21}u^2v+\bar{c}_{12}uv^2+\bar{c}_{40}u^4+\bar{c}_{31}u^3v+\bar{c}_{22}u^2v^2\\[2ex]
&\quad+\bar{c}_{50}u^5+\bar{c}_{32}u^3v^2+\bar{c}_{41}u^4v+O(|u,v|^5),
\end{aligned}
\right.
\end{array}
\end{equation}
where $\bar{c}_{ij}$ are smooth functions whose long expressions are omitted here.

Thirdly, letting $u=x+\frac{\bar{c}_{02}}{2}x^2$ and $v=y+\bar{c}_{02}xy$, system \eqref{4.5} can be transformed into
\begin{equation}
\begin{array}{ll}\label{4.6}
\left\{
\begin{aligned}
\dot{x}&=y,\\[2ex]
\dot{y}&=\bar{d}_{00}+\bar{d}_{10}x+\bar{d}_{01}y+\bar{d}_{20}x^2+\bar{d}_{11}xy+\bar{d}_{30}x^3+\bar{d}_{12}xy^2+\bar{d}_{40}x^4+\bar{d}_{50}x^5\\[2ex]
&\quad+\bar{d}_{21}x^2y+\bar{d}_{31}x^3y+\bar{d}_{22}x^2y^2+\bar{d}_{32}x^3y^2+\bar{d}_{41}x^4y+O(|x,y|^5),
\end{aligned}
\right.
\end{array}
\end{equation}
where
\begin{equation*}
\begin{array}{ll}
\begin{aligned}
\bar{d}_{00}&=\bar{c}_{00},\quad \bar{d}_{10}=\bar{c}_{10}-\bar{c}_{00}\bar{c}_{02},\quad \bar{d}_{01}=\bar{c}_{01},\quad \bar{d}_{20}=\bar{c}_{20}+\bar{c}_{00}\bar{c}_{02}^2-\frac{\bar{c}_{02}\bar{c}_{10}}{2},\\[2ex]
\bar{d}_{30}&=\bar{c}_{30}-\bar{c}_{00}\bar{c}_{02}^3+\frac{\bar{c}_{02}^2\bar{c}_{10}}{2},\quad \bar{d}_{21}=\bar{c}_{21}+\frac{\bar{c}_{02}\bar{c}_{11}}{2},\quad \bar{d}_{12}=\bar{c}_{12}+2\bar{c}_{02}^2,\quad \bar{d}_{11}=\bar{c}_{11},\\[2ex]
\bar{d}_{40}&=\bar{c}_{40}+\bar{c}_{00}\bar{c}_{02}^4-\frac{\bar{c}_{02}(\bar{c}_{02}^2\bar{c}_{10}-\bar{c}_{30})}{2}+\frac{\bar{c}_{02}^2\bar{c}_{20}}{4},\quad \bar{d}_{22}=\bar{c}_{22}-\bar{c}_{02}^3+\frac{3\bar{c}_{02}\bar{c}_{12}}{2},\\[2ex]
\bar{d}_{50}&=\bar{c}_{50}-\bar{c}_{02}(\bar{c}_{00}\bar{c}_{02}^4-\bar{c}_{40})+\frac{1}{4}\bar{c}_{02}^2(2\bar{c}_{02}^2\bar{c}_{10}-\bar{c}_{02}\bar{c}_{20}+\bar{c}_{30}),\quad \bar{d}_{31}=\bar{c}_{31}+\bar{c}_{02}\bar{c}_{21},\\[2ex] \bar{d}_{41}&=\bar{c}_{41}+\frac{1}{4}\bar{c}_{02}(\bar{c}_{02}\bar{c}_{21}+6\bar{c}_{31}),\quad \bar{d}_{32}=\bar{c}_{32}+2\bar{c}_{02}\bar{c}_{22}+\bar{c}_{02}^4+\frac{\bar{c}_{02}^2\bar{c}_{12}}{2}.
\end{aligned}
\end{array}
\end{equation*}

Fourthly, making $x=x_1+\frac{\bar{d}_{12}}{6}x_1^3$ and $y=y_1+\frac{\bar{d}_{12}}{2}x_1^2y_1$, system \eqref{4.6} can be written as
\begin{equation}
\begin{array}{ll}\label{4.7}
\left\{
\begin{aligned}
\dot{x}_1&=y_1,\\[2ex]
\dot{y}_1&=\bar{e}_{00}+\bar{e}_{10}x_1+\bar{e}_{01}y_1+\bar{e}_{20}x_1^2+\bar{e}_{11}x_1y_1+\bar{e}_{30}x_1^3+\bar{e}_{21}x_1^2y_1+\bar{e}_{31}x_1^3y_1\\[2ex]
&\quad+\bar{e}_{40}x_1^4+\bar{e}_{22}x_1^2y_1^2+\bar{e}_{50}x_1^5+\bar{e}_{32}x_1^3y_1^2+\bar{e}_{41}x_1^4y_1+O(|x_1,y_1|^5),
\end{aligned}
\right.
\end{array}
\end{equation}
where
\begin{equation*}
\begin{array}{ll}
\begin{aligned}
\bar{e}_{00}&=\bar{d}_{00},\quad \bar{e}_{10}=\bar{d}_{10},\quad \bar{e}_{01}=\bar{d}_{01},\quad \bar{e}_{20}=\bar{d}_{20}-\frac{\bar{d}_{00}\bar{d}_{12}}{2},\quad \bar{e}_{11}=\bar{d}_{11},\quad \bar{e}_{22}=\bar{d}_{22},\\[2ex]
\bar{e}_{30}&=\bar{d}_{30}-\frac{\bar{d}_{10}\bar{d}_{12}}{3},\quad \bar{e}_{40}=\bar{d}_{40}-\frac{\bar{d}_{00}\bar{d}_{12}^2}{4}-\frac{\bar{d}_{12}\bar{d}_{20}}{6},\quad \bar{e}_{31}=\bar{d}_{31}+\frac{\bar{d}_{11}\bar{d}_{12}}{6},\\[2ex]
\bar{e}_{50}&=\bar{d}_{50}+\frac{\bar{d}_{10}\bar{d}_{12}^2}{5},\quad \bar{e}_{41}=\bar{d}_{41}+\frac{\bar{d}_{12}\bar{d}_{21}}{3},\quad \bar{e}_{32}=\bar{d}_{32}+\frac{7\bar{d}_{12}^2}{6},\quad \bar{e}_{21}=\bar{d}_{21}.
\end{aligned}
\end{array}
\end{equation*}

Fifthly, letting $x_1=u+\frac{\bar{e}_{22}}{12}u^4$ and $y_1=v+\frac{\bar{e}_{22}}{3}u^3v$, \eqref{4.7} can be changed into
\begin{equation}
\begin{array}{ll}\label{4.8}
\left\{
\begin{aligned}
\dot{u}&=v,\\[2ex]
\dot{v}&=\bar{h}_{00}+\bar{h}_{10}u+\bar{h}_{01}v+\bar{h}_{20}u^2+\bar{h}_{11}uv+\bar{h}_{30}u^3+\bar{h}_{21}u^2y_1+\bar{h}_{31}u^3v\\[2ex]
&\quad+\bar{h}_{40}u^4+\bar{h}_{50}u^5+\bar{h}_{32}u^3v^2+\bar{h}_{41}u^4v+O(|u,v|^5),
\end{aligned}
\right.
\end{array}
\end{equation}
where
\begin{equation*}
\begin{array}{ll}
\begin{aligned}
\bar{h}_{00}&=\bar{e}_{00},\quad \bar{h}_{10}=\bar{e}_{10},\quad \bar{h}_{01}=\bar{e}_{01},\quad \bar{h}_{20}=\bar{e}_{20},\quad \bar{h}_{11}=\bar{e}_{11},\\[2ex]
\bar{h}_{30}&=\bar{e}_{30}-\frac{\bar{e}_{00}\bar{e}_{22}}{3},\quad \bar{h}_{21}=\bar{e}_{21},\quad \bar{h}_{40}=\bar{e}_{40}-\frac{\bar{e}_{10}\bar{e}_{22}}{4},\quad \bar{h}_{31}=\bar{e}_{31},\\[2ex]
\bar{h}_{50}&=\bar{e}_{50}-\frac{\bar{e}_{20}\bar{e}_{22}}{5},\quad \bar{h}_{41}=\bar{e}_{41}+\frac{\bar{e}_{11}\bar{e}_{22}}{12},\quad \bar{h}_{32}=\bar{e}_{32}.
\end{aligned}
\end{array}
\end{equation*}

Sixthly, setting $u=x+\frac{\bar{h}_{32}}{20}x^5$ and $v=y+\frac{\bar{h}_{32}}{4}x^4y$, then system \eqref{4.8} can be transformed into
\begin{equation}
\begin{array}{ll}\label{4.9}
\left\{
\begin{aligned}
\dot{x}&=y,\\[2ex]
\dot{y}&=\bar{k}_{00}+\bar{k}_{10}x+\bar{k}_{01}y+\bar{k}_{20}x^2+\bar{k}_{11}xy+\bar{k}_{30}x^3+\bar{k}_{21}x^2y+\bar{k}_{31}x^3y\\[2ex]
&\quad+\bar{k}_{40}x^4+\bar{k}_{50}x^5+\bar{k}_{41}x^4y+O(|x,y|^5),
\end{aligned}
\right.
\end{array}
\end{equation}
where
\begin{equation*}
\begin{array}{ll}
\begin{aligned}
\bar{k}_{00}&=\bar{h}_{00},\quad \bar{k}_{10}=\bar{h}_{10},\quad \bar{k}_{01}=\bar{h}_{01},\quad \bar{k}_{20}=\bar{h}_{20},\quad \bar{k}_{11}=\bar{h}_{11},\quad \bar{k}_{30}=\bar{h}_{30},\\[2ex]
\bar{k}_{21}&=\bar{h}_{21},\quad \bar{k}_{40}=\bar{h}_{40}-\frac{\bar{h}_{00}\bar{h}_{32}}{4},\quad \bar{k}_{31}=\bar{h}_{31},\quad \bar{k}_{50}=\bar{h}_{50}-\frac{\bar{h}_{10}\bar{h}_{32}}{5},\quad \bar{k}_{41}=\bar{h}_{41}.
\end{aligned}
\end{array}
\end{equation*}

Seventhly, making
\begin{equation*}
\begin{array}{ll}
\begin{aligned}
x&=X-\frac{\bar{k}_{30}}{4\bar{k}_{20}}X^2+\frac{15\bar{k}_{30}^2-16\bar{k}_{20}\bar{k}_{40}}{80\bar{k}_{20}^2}X^3
+\frac{336\bar{k}_{20}\bar{k}_{30}\bar{k}_{40}-175\bar{k}_{30}^3-160\bar{k}_{20}^2\bar{k}_{50}}{960\bar{k}_{20}^3}X^4,\\[2ex]
y&=Y,\\[2ex]
t&=(1-\frac{\bar{k}_{30}}{2\bar{k}_{20}}X+\frac{45\bar{k}_{30}^2-48\bar{k}_{40}}{80\bar{k}_{20}^2}X^2
+\frac{336\bar{k}_{20}\bar{k}_{30}\bar{k}_{40}-175\bar{k}_{30}^3-160\bar{k}_{20}^2\bar{k}_{50}}{240\bar{k}_{20}^3}X^3)\tau,
\end{aligned}
\end{array}
\end{equation*}
system \eqref{4.9} can be written as (still denote $\tau$ by $t$)
\begin{equation}
\begin{array}{ll}\label{4.10}
\left\{
\begin{aligned}
\dot{X}&=Y,\\[2ex]
\dot{X}&=\bar{l}_{00}+\bar{l}_{10}X+\bar{l}_{01}Y+\bar{l}_{20}X^2+\bar{l}_{11}XY+\bar{l}_{30}X^3+\bar{l}_{21}X^2Y+\bar{l}_{40}X^4\\[2ex]
&\quad+\bar{l}_{31}X^3Y+\bar{l}_{50}X^5+\bar{l}_{41}X^4Y+O(|X,Y|^5),
\end{aligned}
\right.
\end{array}
\end{equation}
where
\begin{equation*}
\begin{array}{ll}
\begin{aligned}
\bar{l}_{20}&=\bar{k}_{20}+\frac{9\bar{k}_{00}\bar{k}_{30}^2}{16\bar{k}_{20}^2}-\frac{3\bar{k}_{00}\bar{k}_{40}}{5\bar{k}_{20}}-\frac{3\bar{k}_{10}\bar{k}_{30}}{4\bar{k}_{20}},\quad \bar{l}_{40}=-\frac{\bar{k}_{10}(55\bar{k}_{30}^3-96\bar{k}_{20}\bar{k}_{30}\bar{k}_{40}+40\bar{k}_{20}^2\bar{k}_{50})}{48\bar{k}_{20}^3},\\[2ex]
\bar{l}_{21}&=\bar{k}_{21}-\frac{3\big(5\bar{k}_{30}(4\bar{k}_{11}\bar{k}_{20}-3\bar{k}_{01}\bar{k}_{30})+16\bar{k}_{01}\bar{k}_{20}\bar{k}_{40}\big)}{80\bar{k}_{20}^2},\quad \bar{l}_{01}=\bar{k}_{01},\quad \bar{l}_{11}=\bar{k}_{11}-\frac{\bar{k}_{01}\bar{k}_{30}}{2\bar{k}_{20}},\\[2ex]
\bar{l}_{30}&=\frac{6\bar{k}_{10}\bar{k}_{20}(35\bar{k}_{30}^2-32\bar{k}_{20}\bar{k}_{40})-\bar{k}_{00}(175\bar{k}_{30}^3-336\bar{k}_{20}\bar{k}_{30}\bar{k}_{40}+160\bar{k}_{20}^2\bar{k}_{50})}{240\bar{k}_{20}^3},\quad \bar{l}_{00}=\bar{k}_{00},\\[2ex]
\bar{l}_{31}&=\frac{1}{240\bar{k}_{20}^3}\big(16\bar{k}_{20}^2(15\bar{k}_{20}\bar{k}_{31}-12\bar{k}_{11}\bar{k}_{40}-10\bar{k}_{01}\bar{k}_{50})-175\bar{k}_{01}\bar{k}_{30}^3\\[2ex]
&\quad-6\bar{k}_{20}\bar{k}_{30}(40\bar{k}_{20}\bar{k}_{21}-35\bar{k}_{11}\bar{k}_{30}-56\bar{k}_{01}\bar{k}_{40})\big),\quad l_{10}=\bar{k}_{10}-\frac{\bar{k}_{00}\bar{k}_{30}}{2\bar{k}_{20}},\\[2ex]
\bar{l}_{50}&=\frac{\bar{k}_{10}\big(2425\bar{k}_{30}^4+768\bar{k}_{20}^2\bar{k}_{40}^2+1600\bar{k}_{20}\bar{k}_{30}(\bar{k}_{20}\bar{k}_{50}-3\bar{k}_{30}\bar{k}_{40})\big)}{6400\bar{k}_{20}^4},\\[2ex]
\bar{l}_{41}&=\frac{1}{48\bar{k}_{20}^3}\big(8\bar{k}_{20}^2(6\bar{k}_{20}\bar{k}_{41}-5\bar{k}_{11}\bar{k}_{50})-48\bar{k}_{20}\bar{k}_{40}(\bar{k}_{20}\bar{k}_{21}-2\bar{k}_{11}\bar{k}_{30})\\[2ex]
&\quad-5\bar{k}_{30}(11\bar{k}_{11}\bar{k}_{30}^2+12\bar{k}_{20}^2\bar{k}_{31}-12\bar{k}_{20}\bar{k}_{21}\bar{k}_{30})\big).
\end{aligned}
\end{array}
\end{equation*}

Eighthly, letting $X=x_2$, $Y=y_2+\frac{\bar{l}_{21}}{3\bar{l}_{20}}y_2^2+\frac{\bar{l}_{21}^2}{36\bar{l}_{20}^2}y_2^3$ and $t=(1+\frac{\bar{l}_{21}}{3\bar{l}_{20}}y_2+\frac{\bar{l}_{21}^2}{36\bar{l}_{20}^2}y_2^2)\tau$, system \eqref{4.10} can be transformed into (still denote $\tau$ by $t$)
\begin{equation}
\begin{array}{ll}\label{4.11}
\left\{
\begin{aligned}
\dot{x_2}&=y_2,\\[2ex]
\dot{y_2}&=\bar{m}_{00}+\bar{m}_{10}x_2+\bar{m}_{01}y_2+\bar{m}_{20}x_2^2+\bar{m}_{11}x_2y_2\\[2ex]
&\quad+\bar{m}_{31}x_2^3y_2+\bar{m}_{41}x_2^4y_2+O(|x_2,y_2|^5),
\end{aligned}
\right.
\end{array}
\end{equation}
where
\begin{equation*}
\begin{array}{ll}
\begin{aligned}
\bar{m}_{00}&=l_{00},\quad \bar{m}_{10}=\bar{l}_{10},\quad \bar{m}_{01}=\bar{l}_{01}-\frac{\bar{l}_{00}\bar{l}_{21}}{\bar{l}_{20}},\quad \bar{m}_{20}=\bar{l}_{20},\\[2ex]
\bar{m}_{11}&=\bar{l}_{11}-\frac{\bar{l}_{10}\bar{l}_{21}}{\bar{l}_{20}},\quad \bar{m}_{31}=\bar{l}_{31}-\frac{\bar{l}_{21}\bar{l}_{30}}{\bar{l}_{20}},\quad \bar{m}_{41}=\bar{l}_{41}-\frac{\bar{l}_{21}\bar{l}_{40}}{\bar{l}_{20}}.
\end{aligned}
\end{array}
\end{equation*}

Ninthly, setting $x_2=\bar{m}_{20}^{\frac{1}{7}}\bar{m}_{41}^{-\frac{2}{7}}u$, $y_2=-\bar{m}_{20}^{\frac{5}{7}}\bar{m}_{41}^{-\frac{3}{7}}v$ and $t=-\bar{m}_{20}^{-\frac{4}{7}}\bar{m}_{41}^{\frac{1}{7}}\tau$, system \eqref{4.11} can be changed into
\begin{equation}
\begin{array}{ll}\label{4.12}
\left\{
\begin{aligned}
\dot{u}&=v,\\[2ex]
\dot{v}&=\bar{n}_{00}+\bar{n}_{10}u+\bar{n}_{01}v+\bar{n}_{11}uv+\bar{n}_{31}u^3v+u^2-u^4v+O(|u,v|^5),
\end{aligned}
\right.
\end{array}
\end{equation}
where
\begin{equation*}
\begin{array}{ll}
\begin{aligned}
\bar{n}_{00}&=\frac{\bar{m}_{00}\bar{m}_{41}^{\frac{4}{7}}}{\bar{m}_{20}^{\frac{9}{7}}},\quad \bar{n}_{10}=\frac{\bar{m}_{10}\bar{m}_{41}^{\frac{2}{7}}}{\bar{m}_{20}^{\frac{8}{7}}},
\quad \bar{n}_{01}=-\frac{\bar{m}_{01}\bar{m}_{41}^{\frac{1}{7}}}{\bar{m}_{20}^{\frac{4}{7}}},\quad\bar{n}_{11}=-\frac{\bar{m}_{11}}{\bar{m}_{20}^{\frac{3}{7}}\bar{m}_{41}^{\frac{1}{7}}},\quad \bar{n}_{31}=-\frac{\bar{m}_{31}}{\bar{m}_{20}^{\frac{1}{7}}\bar{m}_{41}^{\frac{5}{7}}}.
\end{aligned}
\end{array}
\end{equation*}

Finally, making $u=x-\frac{\bar{n}_{10}}{2}$ and $v=y$, system \eqref{4.12} can be written as
\begin{equation}
\begin{array}{ll}\label{4.13}
\left\{
\begin{aligned}
\dot{x}&=y,\\[2ex]
\dot{y}&=\bar{\chi}_1+\bar{\chi}_2y+\bar{\chi}_3xy+\bar{\chi}_4x^3y+x^2-x^4y+O(|x,y|^5),
\end{aligned}
\right.
\end{array}
\end{equation}
where
\begin{equation*}
\begin{array}{ll}
\begin{aligned}
\bar{\chi}_1&=\bar{n}_{00}-\frac{\bar{n}_{10}^2}{4},\quad \bar{\chi}_2=\bar{n}_{01}-\frac{\bar{n}_{10}^4}{16}-\frac{\bar{n}_{10}^3\bar{n}_{31}}{8}-\frac{\bar{n}_{10}\bar{n}_{11}}{2},\\[2ex]
\bar{\chi}_3&=\bar{n}_{11}+\frac{\bar{n}_{10}^3}{2}+\frac{3\bar{n}_{10}^2\bar{n}_{31}}{4},\quad \bar{\chi}_4=2\bar{n}_{10}+\bar{n}_{31}.
\end{aligned}
\end{array}
\end{equation*}

Through calculate, we have
\begin{equation}
\begin{array}{ll}\label{4.14}
\begin{aligned}
\bigg|\frac{\partial(\bar{\chi}_1,\bar{\chi}_2,\bar{\chi}_3,\bar{\chi}_4)}{\partial(\lambda_1,\lambda_2,\lambda_3,\lambda_4)}\bigg|_{\lambda=0}\neq0.
\end{aligned}
\end{array}
\end{equation}
Thus, we claim that system \eqref{1.5} can undergo cusp type Bogdanov-Takens bifurcation of codimension 4.
\end{proof}
\begin{figure}
\begin{center}
\begin{overpic}[scale=0.50]{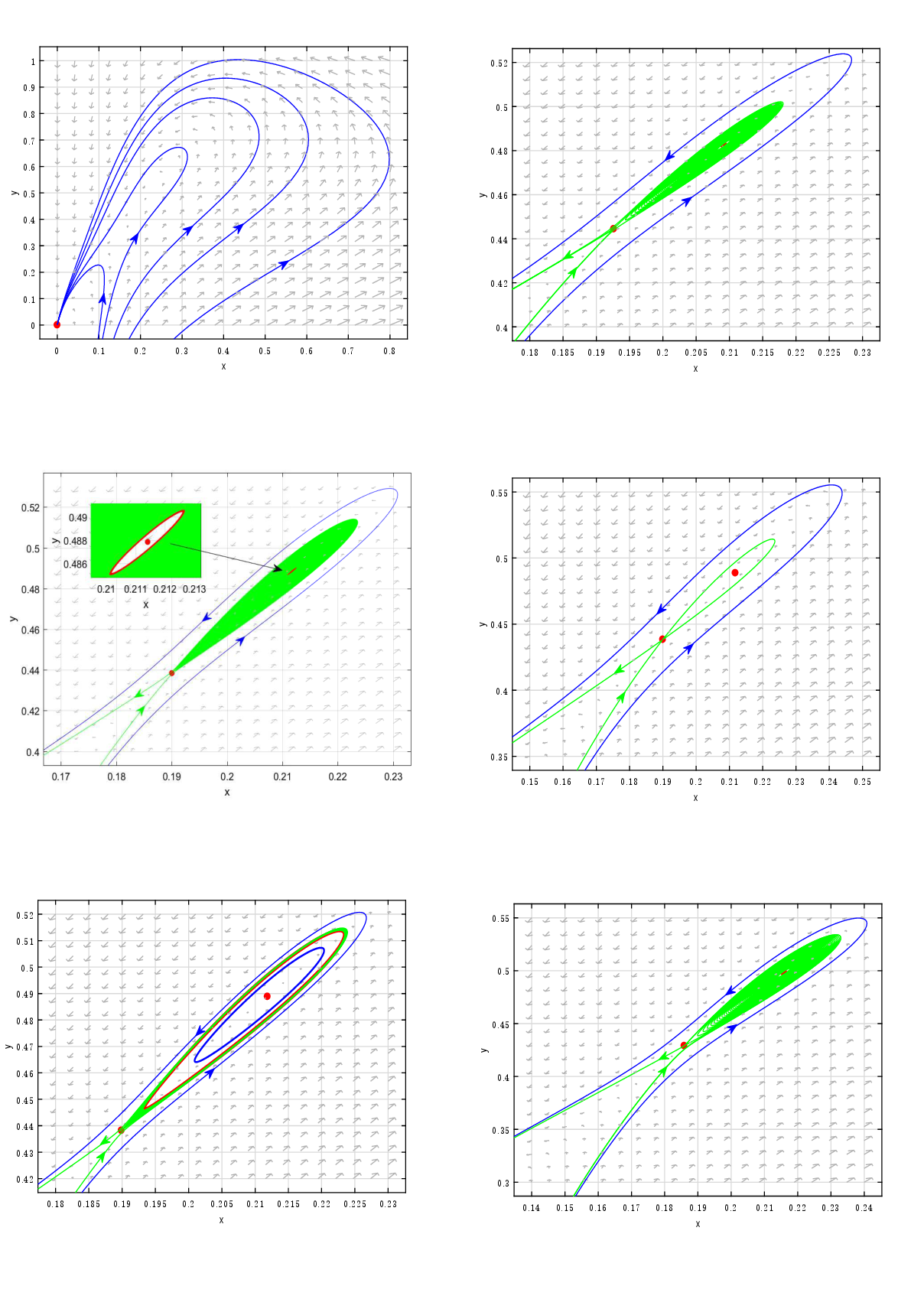}
\end{overpic}
\put(-228,278){$(a)$}
\put(-78,278){$(b)$}
\put(-228,140){$(c)$}
\put(-78,140){$(d)$}
\put(-228,6){$(e)$}
\put(-78,6){$(f)$}
\vspace{-6mm}
\end{center}
\centering{\caption{ The phase portraits of system system \eqref{1.5} with $k=2$, $\Lambda_0=1.8993$, $p=5.7966$ and $\eta=2.3072$. (a) $\gamma=2.6731$; (b) $\gamma=2.6717$; (c) $\gamma=2.6712$; (d) $\gamma=2.6708$; (e) $\gamma=2.6706$; (f) $\gamma=6687$. When $\gamma$ decreases, system \eqref{1.5} undergoes successively saddle-node bifurcation, Hopf bifurcation, homoclinic bifurcation and double limit cycle bifurcation.\label{w2}}}
\end{figure}

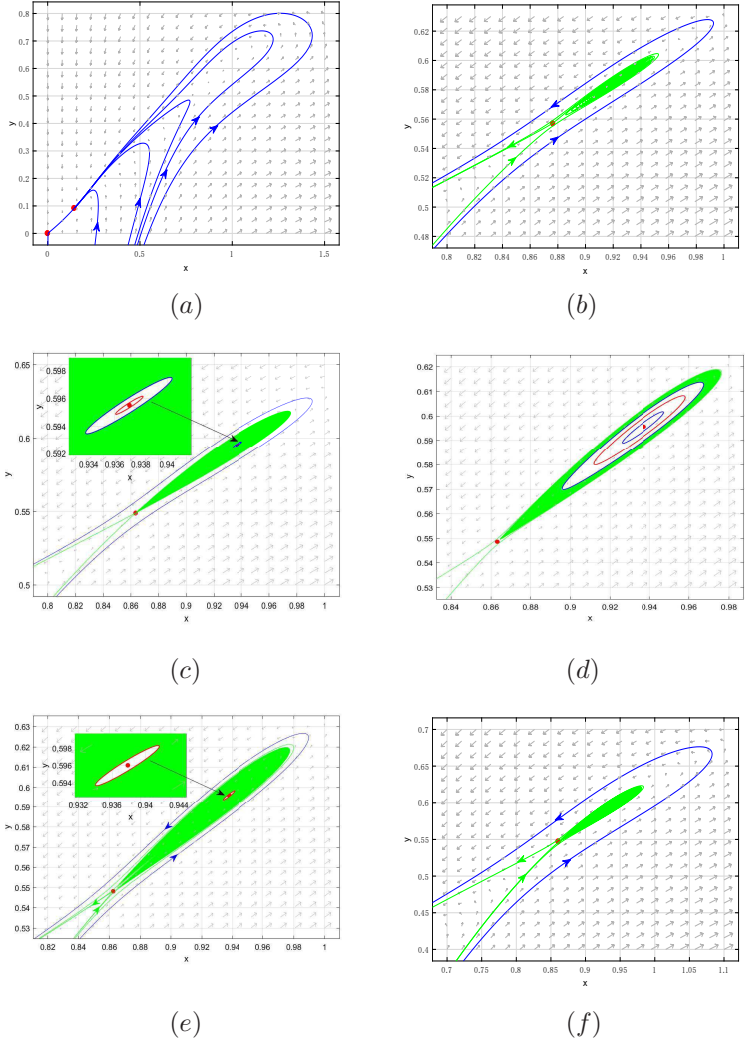
\begin{figure}
\begin{center}
\begin{overpic}[scale=0.50]{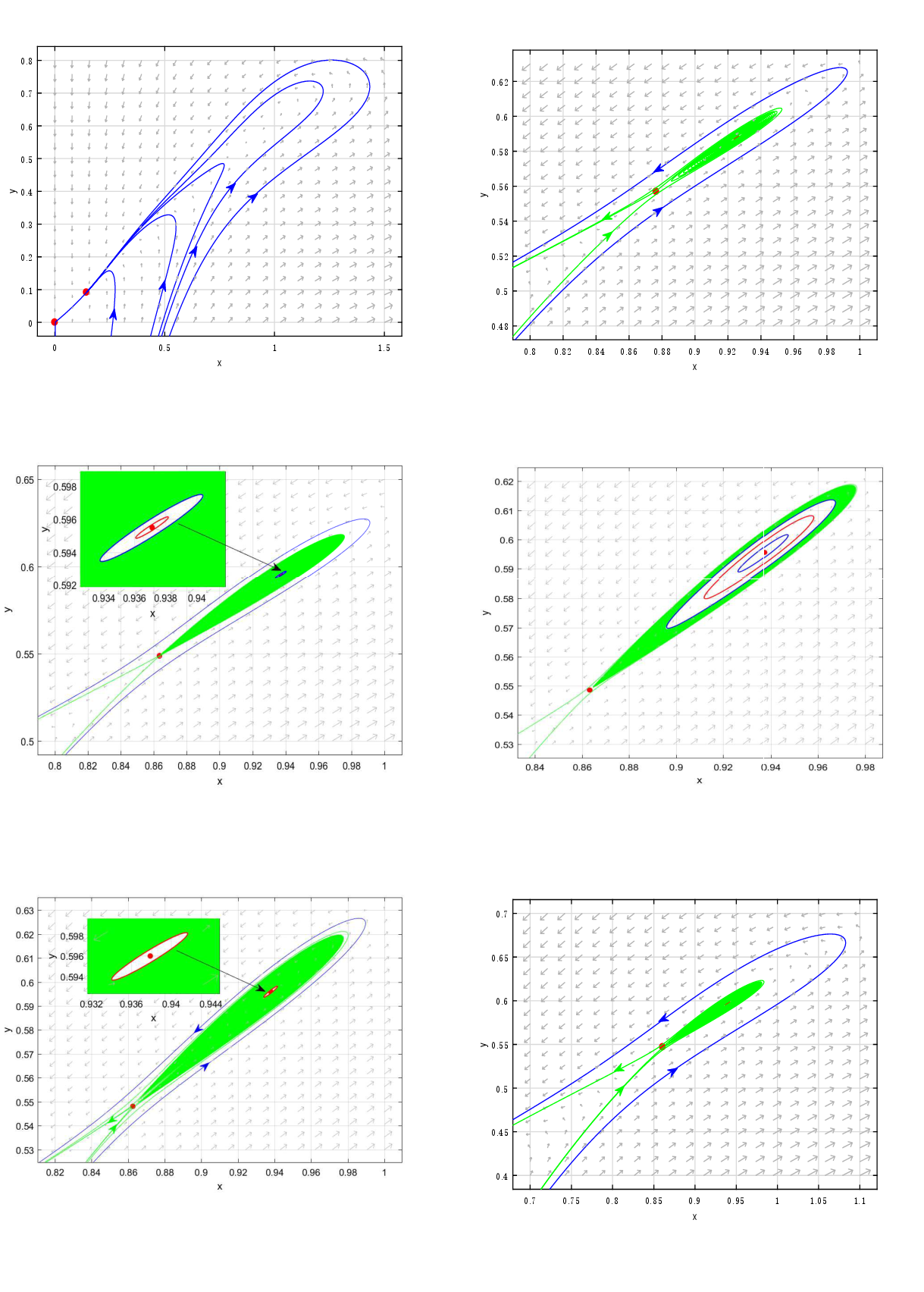}
\end{overpic}
\put(-228,278){$(a)$}
\put(-78,278){$(b)$}
\put(-228,140){$(c)$}
\put(-78,140){$(d)$}
\put(-228,6){$(e)$}
\put(-78,6){$(f)$}
\vspace{-6mm}
\end{center}
\centering{\caption{ The phase portraits of system system \eqref{1.5} with $k=3$, $\Lambda_0=3.1832$, $p=0.9331$ and $\eta=0.6355$. (a) $\gamma=3.0047$; (b) $\gamma=3.004$; (c) $\gamma=3.00312$; (d) $\gamma=3.00311$; (e) $\gamma=3.00305$; (f) $\gamma=3.0029$. When $\gamma$ decreases, system \eqref{1.5} undergoes successively saddle-node bifurcation, double limit cycle bifurcation, three limit cycle bifurcation and Hopf bifurcation.\label{w3}}}
\end{figure}
Similar to Theorem \ref{t5}, we have the following results.
\begin{theorem}
For $1<k\leq2$, system \eqref{1.5} can undergo cusp type Bogdanov-Takens bifurcation of codimension 3 around $E_2^*$ (or $\hat{E}_0^*$) as $(\Lambda_0,\gamma,\eta,p)$ varies near $(\tilde{\Lambda}_0,\tilde{\gamma},\tilde{\eta},\check{p})$, see Fig.\ref{w2}. There exist a series of bifurcation with codimension 2, and 3 originating from $E_2^*$ (or $\hat{E}_0^*$).

(i) If $\Lambda_0=\frac{(1+z+pz^k)(z+kpz^k)}{pz^{k-1}(k-1)(z+pz^k)}$, $\gamma=\frac{(1+z+pz^k)(z+pz^k)}{pz^k(k-1)}$ and $\eta=\frac{1}{z+pz^k}$, then system \eqref{1.5} can undergo the cusp  bifurcation of codimension 2 near $E_2^*$ (or $\hat{E}_0^*$) for $1<k\leq2$.

(ii) If $\Lambda_0=\frac{kz(k-z+kz)}{(k-1)\big(\sqrt{(k-1)(k-z+kz)z}-z\big)}$, $\gamma=\frac{\big(kz-z+\sqrt{(k-1)(k-z+kz)z}\big)\big(k-z+kz+\sqrt{(k-1)(k-z+kz)z}\big)}{(k-1)k\big(\sqrt{(k-1)(k-z+kz)z}-z\big)}$, $\eta=\frac{k}{\sqrt{(k-1)(k-z+kz)z}-z+kz}$ and $p=\frac{\sqrt{z(k-z+kz)(k-1)}-z}{kz^k}$, then system \eqref{1.5} can undergo the cusp  bifurcation of codimension 3 near $E_2^*$ (or $\hat{E}_0^*$) for $\frac{1+2z}{1+z}<k\leq2$.
\end{theorem}

\section{Hopf bifurcation}

In this section, we consider Hopf bifurcation of system \eqref{1.5} around positive equilibrium $E(x,y)$, which satisfied $Tr(J(E))=0$ and $Det(J(E))>0$ when $k>1$. To simplify the notation, we let $E=(z,\eta z)$. From \eqref{2.2} and $Tr(J(E))=0$, $\Lambda_0$ and $\gamma$ can be expresses by $p$, $\eta$, $z$ and $k$ as follows
\begin{equation}\label{5.1}
\begin{array}{ll}
\begin{aligned}
\Lambda_0&=\breve{\Lambda}_0:=\frac{1+z+pz^k(k-\eta+k\eta)}{pz^{k-1}(k-1)},\\[2ex]
\gamma&=\breve{\gamma}:=\frac{z(z+1)+pz^k(pz^k+1)+2pz^{k+1}}{pz^k(k-1)}.
\end{aligned}
\end{array}
\end{equation}

It is not difficult to show that $\breve{\Lambda}_0>0$, $\breve{\gamma}>0$ and $Det(J(E))>0$ if and only if $(k,z,p,\eta)\in\Omega^*$, where
\begin{equation}
\begin{array}{ll}\label{5.2}
\begin{aligned}
\Omega^*:&=\bigg\{(k,z,p,\eta)\in R_+^4\bigg|\frac{1}{z+pz^k}<\eta<\frac{(z+pz^k)(1+z+pz^k)}{pz^k(k-1)}, p>0, z>0, k>1\bigg\}.
\end{aligned}
\end{array}
\end{equation}

Now, we compute the first four focal values of system \eqref{1.5}. Make the following linear transformations successively
\begin{equation*}
\begin{array}{ll}
\begin{aligned}
x&=X+z,\quad y=Y+\eta z;\\[2ex]
X&=\frac{1}{\eta}u+\frac{\sqrt{p\eta z^k+\eta z-1}}{\eta}v,\quad Y=u,\quad t=\frac{\tau}{\sqrt{\mathcal{D}}},
\end{aligned}
\end{array}
\end{equation*}
where $\mathcal{D}=Det(J(E))=p\eta z^k+\eta z-1$ when $k>1$, then system \eqref{1.5} can be transformed into (still $\tau$ by $t$)
\begin{equation}
\begin{array}{ll}\label{5.3}
\left\{
\begin{aligned}
\dot{u}&=v,\\[2ex]
\dot{v}&=-u+\sum\limits_{2\leq i+j\leq9}\omega_{ij}u^{i}v^{j}+O(|u,v|^{10}),
\end{aligned}
\right.
\end{array}
\end{equation}
where $\omega_{ij}$ are omitted for brevity.

According to formal series method \cite{ZTW}, we can get the first four focal values as follows
\begin{equation*}
\begin{array}{ll}
\begin{aligned}
\vartheta_1&=\frac{f_1}{8\eta z^2(p\eta z^k+\eta z-1)^{3/2}},\quad\vartheta_2=\frac{f_2}{1152\eta^3z^4(p\eta z^k+\eta z-1)^{7/2}},\\[2ex]
\vartheta_3&=\frac{f_3}{442368\eta^5z^6(p\eta z^k+\eta z-1)^{11/2}},\quad\vartheta_4=\frac{f_4}{1061683200\eta^7z^8(p\eta z^k+\eta z-1)^{15/2}}.
\end{aligned}
\end{array}
\end{equation*}
where
\begin{equation*}
\begin{array}{ll}
\begin{aligned}
f_1&=(4-k+3\eta-3k\eta)kp^2z^{2k+1}+(2-k+\eta-k\eta)(k+2)pz^{k+2}+(3-k+\eta-k\eta)kpz^{k+1}\\[2ex]
&\quad+(k-k\eta+\eta)kp^3z^{3k}+(2k-1)kp^2z^{2k}+(k-2-\eta+k\eta)(k-2)z^3\\[2ex]
&\quad-(2-k+\eta-k\eta)(k-2)z^2,
\end{aligned}
\end{array}
\end{equation*}
and $f_2$, $f_3$, $f_4$ are omitted here. Let the algebraic variety $V(\xi_1,\xi_2,\cdots,\xi_n)$ denote the set of common zeros of $\xi_i$  ($i=1, 2,\cdots, n$), $\mathrm{Res}(f,g,x)$ denotes the Sylvester resultant of $f$ and $g$ with respect to $x$. Through calculate, we have
\begin{equation}\label{5.4}
\begin{array}{ll}
\begin{aligned}
r_{12}:&=\mathrm{Res}(f_1,f_2,p)=128k^4\eta^3z^{6+21k}(1+\eta)^3(1+z)(2-k)(k-1)^7R_0R_1R_2^2R_3^2g_1,\\[2ex]
r_{13}:&=\mathrm{Res}(f_1,f_3,p)=4096k^5\eta^3z^{6+33k}(1+\eta)^4(1+z)(2-k)(k-1)^7R_0R_1R_2^3R_3^2g_2,\\[2ex]
r_{14}:&=\mathrm{Res}(f_1,f_4,p)=16384k^6\eta^3z^{6+45k}(1+\eta)^5(1+z)(2-k)(k-1)^7R_0R_1R_2^4R_3^2g_3,\\[2ex]
r_{23}:&=\mathrm{Res}(g_1,g_2,\eta)=C_1k^{34}z^{13}(2-k)^3(k-1)^{15}(k+1)^2(2k-1)^4l_1l_2l_3l_4S_1,\\[2ex]
r_{24}:&=\mathrm{Res}(g_1,g_3,\eta)=C_2k^{55}z^{21}(2-k)^3(k-1)^{24}(k+1)^2(2k-1)^4l_1l_2l_3l_4S_2,\\[2ex]
\end{aligned}
\end{array}
\end{equation}
where
\begin{equation*}
\begin{array}{ll}
\begin{aligned}
R_0&=(k\eta-\eta-k),\quad R_1=k\eta+k-\eta-2,\quad R_2=2k\eta z-2\eta z-k,\\[2ex]
R_3&=k\eta^2z-\eta^2z+2k\eta z-2\eta z-k,\quad l_1=10k^2z-30kz+20z+2k^2-3k,\\[2ex]
l_2&=12(k-1)^2(k-1)z^24(k-1)(k-2)(2k-1)z-k(2k-1)^2,\\[2ex]
l_3&=12(k-1)^4z^4+8(k-1)^3(5k+2)z^3+(k-1)^2(2k-1)(23k+18)z^2\\[2ex]
&\quad+2kz(k-1)(2k-1)(5k+2)+k^2(k+1)(2k-1),\\[2ex]
l_4&=300(k-2)^2(k-1)^3z^4+4(k-1)^2(k-2)(127k^2-181k-158)z^3\\[2ex]
&\quad+2(k-1)(k-1)(2k-1)(66k^2+111k-535)z^2+(2k-1)^2(51k^3-60k^2-407k+648)z\\[2ex]
&\quad+(k-2)(k+1)(2k-1)^2(37k-81).
\end{aligned}
\end{array}
\end{equation*}
For brevity, we omit expressions of $g_i$ ($i=1, 2, 3$) and $S_i$ ($i=1, 2, 3, 4$). Based on the preceding examination, we know that all factors, except $g_i$ ($i=1, 2, 3$) in $r_{1j}$ ($j=2, 3, 4$) and $R_i$ ($i=0, 1, 2, 3$) are not zero when $1<k<2$ and $k>2$.

\subsection{$1<k<2$}

In this section, we consider Hopf bifurcation of system \eqref{1.5} around equilibrium $\bar{E}_2^*$ when $1<k<2$.

\subsubsection{$R_i=0$ ($i=0, 1, 2, 3$)}

For $R_i=0$, the following statements hold.
\begin{theorem}\label{k1}
Let $(z,p,\eta)\in\Omega^*$ and \eqref{4.1} hold, if $R_i=0$, then the following statements hold.

(i) For $R_0=0$, $0<z<\frac{3}{4}$ and $1<k<2$ or $z>\frac{3}{4}$ and $1<k<\frac{4z-1}{4z-2}$, $\bar{E}_2^*$ is a weak focus of order at most 2;

(ii) For $R_1=0$, $\bar{E}_2^*$ is an unstable weak focus of order 1;

(iii) For $R_2=0$, $\frac{1}{3}<z<0.465571$ and $1<k<\frac{2z}{1+2z}\big(1+\frac{1}{1-z-2z^2}\big)$ (or $0.465571<z<\frac{1}{2}$ and $1<k<2$ or $z>\frac{1}{2}$ and $\frac{4z}{1+2z}<k<2$), $\bar{E}_2^*$ is a weak focus of order at most 2;

(iv) For $R_3=0$, $\bar{E}_2^*$ is a weak focus of order at most 2.
\end{theorem}
\begin{proof}
For $R_0=0$, we have $\eta=\eta_1:=\frac{k}{k-1}$. By substituting $\eta_1$ into $f_1$ and $f_2$, we can obtain
\begin{equation*}
\begin{array}{ll}
\begin{aligned}
f_1=h_{11},\quad f_{2}=\frac{h_{22}}{(k-1)^3},
\end{aligned}
\end{array}
\end{equation*}
where
\begin{equation*}
\begin{array}{ll}
\begin{aligned}
h_{11}&=2(k-2)(k-1)z^2(1+z)+k(2k-1)p^2z^{2k}-k(2k-3)pz^{k+1}\\[2ex]
&\quad-2(k-1)(2+k)pz^{k+2}-4(k-1)kp^2z^{2k+1}
\end{aligned}
\end{array}
\end{equation*}
and we  omit the expression of $h_{22}$ for brevity. We treat $h_{11}$ as a quadratic functions of $p$, whose discriminant is
\begin{equation*}
\begin{array}{ll}
\begin{aligned}
\bar{\Delta}_p&=z^{2+2k}(16k-47k^2+44k^3-12k^4-24kz+76k^2z-76k^3z+24k^4z+16z^2-80kz^2\\[2ex]
&\quad+148k^2z^2-120k^3z^2+36k^4z^2).
\end{aligned}
\end{array}
\end{equation*}
We find that $\bar{\Delta}_p>0$ for $1<k<2$. Thus, $h_{11}$ has two real roots $p_{10}$ and $p_{20}$ ($p_{10}<0<p_{20}$) when $0<z<\frac{3}{4}$ and $1<k<2$ or $z>\frac{3}{4}$ and $1<k<\frac{4z-1}{4z-2}$, where
\begin{equation*}
\begin{array}{ll}
\begin{aligned}
p_{10}&=\frac{z^{k+1}(2k^2z+2kz-4z+2k^2-3k)-\sqrt{\bar{\Delta}_p}}{2kz^{2k}(2k+4z-1-4kz)},\\[2ex] p_{20}&=\frac{z^{k+1}(2k^2z+2kz-4z+2k^2-3k)+\sqrt{\bar{\Delta}_p}}{2kz^{2k}(2k+4z-1-4kz)}.
\end{aligned}
\end{array}
\end{equation*}
Because of $\mathrm{Lcoeff}(h_{11},p)=-kz^{2k}(1-2k-4z+4kz)>0$, then $h_{11}<0$ when $0<p<p_{20}$ and $h_{11}>0$ when $p>p_{20}$. If $z>\frac{3}{4}$ and $\frac{4z-1}{4z-2}<k<2$, then $p_{10}<p_{20}<0$ and $\mathrm{Lcoeff}(h_{11},p)=-kz^{2k}(1-2k-4z+4kz)<0$, i.e., $h_{11}<0$, $\bar{E}_2^*$ is a stable weak focus of system \eqref{1.5}. For $z>\frac{3}{4}$, $\frac{4z-1}{4z-2}<k<2$ and $p=p_{20}$ (i.e., $f_1=0$), we have $f_2>0$. Thus, $\bar{E}_2^*$ is a weak focus of order at most 2. The proof of (ii)-(iv) are similar to (i), which are omitted.
\end{proof}

\subsubsection{$R_i\neq0$ ($i=0, 1, 2, 3$)}

For $R_i\neq0$ and $1<k<2$, we have $l_2<0$, $l_3>0$ and $l_4>0$. If $l_1=0$ (i.e., $z=\breve{z}:=\frac{k(2k-3)}{10(k-1)(2-k)}$ when $\frac{3}{2}<k<2$), then the following statements hold.
\begin{figure}
\begin{center}
\begin{overpic}[scale=0.60]{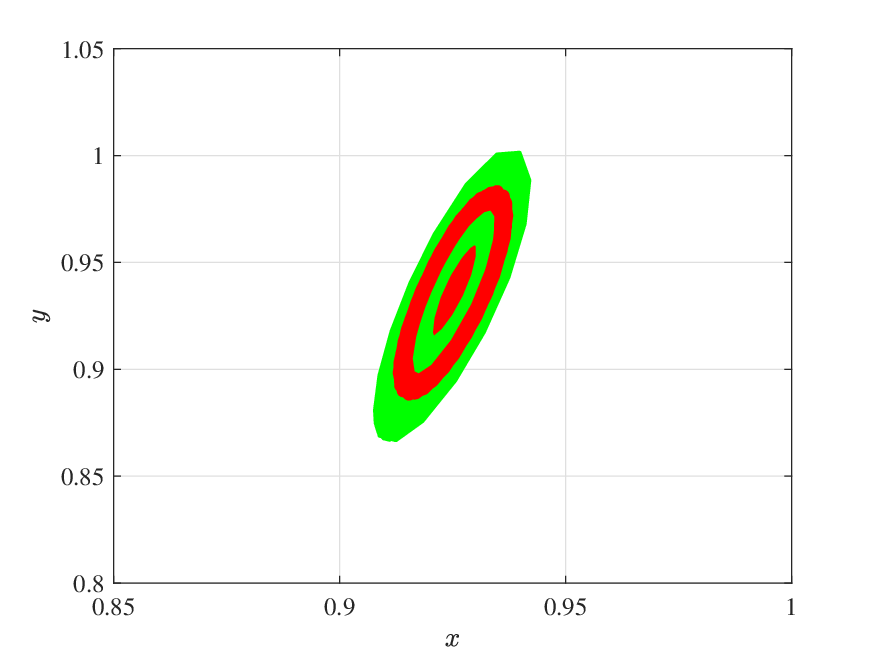}
\end{overpic}
\vspace{-1mm}
\end{center}
\caption{ For $\frac{3}{2}<k<2$, system \eqref{1.5} can undergo at most three limit cycles bifurcated by Hopf bifurcation when $l_1=0$ and $k\neq1.9839$.\label{w8}}
\end{figure}
\begin{theorem}\label{k2}
Let $(z,p,\eta)\in\Omega^*$ and \eqref{4.1} hold. For $l_1=0$, $\frac{3}{2}<k<2$ and $k\neq1.9839$, $\bar{E}_2^*$ can be a weak focus of order at most 3, see Fig.\ref{w8}.
\end{theorem}
\begin{proof}
For $l_1=0$ and $\frac{3}{2}<k<2$, we have $z=\breve{z}$. By substituting $\breve{z}$ into $f_1$, $f_2$ and $f_3$, we can obtain
\begin{equation*}
\begin{array}{ll}
\begin{aligned}
f_1&=\frac{k}{10^{3+3k}(2-k)^2(k-1)^3}\mathcal{L}_1,\quad f_{2}=\frac{2k}{10^{7+7k}(2-k)^5(k-1)^7}\mathcal{L}_2,\\[2ex]
f_3&=\frac{k}{10^{11+11k}(2-k)^8(k-1)^{11}}\mathcal{L}_3,
\end{aligned}
\end{array}
\end{equation*}
where $\mathcal{L}_i$ ($i=1, 2, 3$) are omitted here for brevity. Through calculate, we have
\begin{equation*}
\begin{array}{ll}
\begin{aligned}
\breve{r}_{12}:&=\mathrm{Res}(\mathcal{L}_1,\mathcal{L}_2,p)=2^{35+21k}5^{27+21k}\eta^3k^6(1+\eta)^3(2-k)^{15}(k-1)^{42}(2k-3)^6(27k-8k^2-20)R_0R_1q_1^3q_2^2\mathcal{R}_{10},\\[2ex]
\breve{r}_{13}:&=\mathrm{Res}(\mathcal{L}_1,\mathcal{L}_3,p)=2^{60+33k}5^{43+33k}\eta^3k^9(1+\eta)^4(2-k)^{33}(k-1)^{66}(2k-3)^6(27k-8k^2-20)R_0R_1q_1^4q_2^2\mathcal{R}_{20}.
\end{aligned}
\end{array}
\end{equation*}
If $q_1=0$ i.e., $\eta=\breve{\eta}:=\frac{5(2-k)}{2k-3}$, then
\begin{equation*}
\begin{array}{ll}
\begin{aligned}
f_1&=100(k-1)^2(2-k)(7k^2-18k+10)\bigg(\frac{k(3-2k)}{k^2-3k+2}\bigg)^{2k}p^2\\[2ex]
&\quad+2^{k+2}5^{k+1}k(k-1)(2k-3)^2\bigg(\frac{k(3-2k)}{k^2-3k+2}\bigg)^kp+100^kk(2k-3)(3k-2)(8k^2-27k+20).
\end{aligned}
\end{array}
\end{equation*}
We treat $f_1$ as a quadratic functions of $p$, whose discriminant is
\begin{equation*}
\begin{array}{ll}
\begin{aligned}
\breve{\Delta}_p&=4^{k+2}25^{k+1}k(k-1)^4(2k-3)(800-2547k+2624k^2-1111k^3+168k^4)\bigg(\frac{k(3-2k)}{k^2-3k+2}\bigg)^{2k}.
\end{aligned}
\end{array}
\end{equation*}
We find that $\bar{\Delta}_p>0$ for $1.74557<k<2$ and $\bar{\Delta}_p<0$ for $\frac{3}{2}<k<1.74557$. Thus, $f_1$ has two real roots $p_{10}$ and $p_{20}$ ($0<p_{10}<p_{20}$) when $1.74557<k<1.75952$ or $p_{20}<0<p_{10}$ when $1.75952<k<2$, where
\begin{equation*}
\begin{array}{ll}
\begin{aligned}
p_{10}&=\frac{-2^{k+2}5^{k+1}k(k-1)(2k-3)^2\bigg(\frac{k(3-2k)}{k^2-3k+2}\bigg)^k-\sqrt{\breve{\Delta}_p}}{200\bigg(\frac{k(3-2k)}{k^2-3k+2}\bigg)^{2k}(2-k)(k-1)^2(10-18k+7k^2)},\\[2ex]
p_{20}&=\frac{-2^{k+2}5^{k+1}k(k-1)(2k-3)^2\bigg(\frac{k(3-2k)}{k^2-3k+2}\bigg)^k+\sqrt{\breve{\Delta}_p}}{200\bigg(\frac{k(3-2k)}{k^2-3k+2}\bigg)^{2k}(2-k)(k-1)^2(10-18k+7k^2)}.
\end{aligned}
\end{array}
\end{equation*}

For $p=p_{10}$ or $p=p_{20}$, we always have $f_2>0$, i.e., $\bar{E}_2^*$ is a stable weak focus of order 2 when $z=\breve{z}$ and $\eta=\breve{\eta}$. (For $q_2=0$ is similar to $q=q_1$, which is omitted here.) When $q_i\neq0$ ($i=1, 2$), the first seven factors in $\breve{r}_{12}$ and $\breve{r}_{13}$ are all positive $\frac{3}{2}<k<2$, then
\begin{equation*}
\begin{array}{ll}
\begin{aligned}
V(f_1,f_2,f_3,l_1)\cap\Omega^*=V(\mathcal{L}_1,\mathcal{L}_2,\mathcal{L}_3)\cap\Omega^*=V(\mathcal{L}_1,\mathcal{L}_2,\mathcal{L}_3,\mathcal{R}_{10},\mathcal{R}_{20})\cap\Omega^*.
\end{aligned}
\end{array}
\end{equation*}
Next, we have
\begin{equation*}
\begin{array}{ll}
\begin{aligned}
\mathrm{Res}(\mathcal{R}_{10},\mathcal{R}_{20},\eta)\neq0.\\[2ex]
\end{aligned}
\end{array}
\end{equation*}
when $k\neq1.9839$. Hence, $V(\mathcal{R}_{10},\mathcal{R}_{20})\cap\Omega^*=\emptyset$, then $V(f_1,f_2,f_3,l_1)\cap\Omega^*=V(\mathcal{L}_1,\mathcal{L}_2,\mathcal{L}_3)\cap\Omega^*=\emptyset$, and $\bar{E}_2^*$ is a weak focus of order at most 3 when $z=\breve{z}$ and $k\neq1.9839$.
\end{proof}

For $R_i\neq0$ ($i=0, 1, 2, 3$) and $l_1\neq0$, we will consider it in the future.

\subsection{$k=2$}

In this section, we consider Hopf bifurcation of system \eqref{1.5} around equilibrium $E_{02}^*$. When $k=2$, according to formal series method \cite{ZTW}, we can get the first two focal values as follows
\begin{equation*}
\begin{array}{ll}
\begin{aligned}
f_1=\frac{pzL_{11}}{4\eta(p\eta z^2+\eta z-1)^{3/2}},\quad f_2=\frac{pL_{22}}{96\eta^3z(p\eta z^2+\eta z-1)^{7/2}},
\end{aligned}
\end{array}
\end{equation*}
where
\begin{equation*}
\begin{array}{ll}
\begin{aligned}
L_{11}&=1+3pz+2pz^2+2p^2z^3-\eta-2z\eta-3pz^2\eta-p^2z^3\eta,\\[2ex]
L_{22}&=-72-216pz-36pz^2+180p^2z^3+144p^2z^4-144p^3z^6-144p^4z^7+72\eta+284z\eta\\[2ex]
&\quad+632pz^2\eta-177pz^3\eta+384p^2z^3\eta-682p^2z^4\eta-152p^2z^5\eta-493p^3z^5\eta+80p^3z^6\eta\\[2ex]
&\quad+372p^3z^7\eta+232p^4z^7\eta+744p^4z^8\eta+372p^5z^9\eta-140z\eta^2-345z^2\eta^2-96pz^2\eta^2\\[2ex]
&\quad-659pz^3\eta^2+485pz^4\eta^2-527p^2z^4\eta^2+1230p^2z^5\eta^2-193p^3z^5\eta^2-230p^2z^6\eta^2\\[2ex]
&\quad+1087p^3z^6\eta^2-876p^3z^7\eta^2+342p^4z^7\eta^2-248p^3z^8\eta^2-1062p^4z^8\eta^2-744p^4z^9\eta^2\\[2ex]
&\quad-416p^5z^9\eta^2-744p^5z^{10}\eta^2-248p^6z^{11}\eta^2+65z^2\eta^3+127z^3\eta^3+69pz^3\eta^3+189pz^4\eta^3\\[2ex]
&\quad+15p^2z^4\eta^3-272pz^5\eta^3+58p^2z^5\eta^3-721p^2z^6\eta^3-14p^3z^6\eta^3+248p^2z^7\eta^3\\[2ex]
&\quad-715p^3z^7\eta^3-10p^4z^7\eta^3+868p^3z^8\eta^3-355p^4z^8\eta^3+1116p^4z^9\eta^3-89p^5z^9\eta^3\\[2ex]
&\quad+620p^5z^{10}\eta^3+124p^6z^{11}\eta^3+3z^3\eta^4+6z^4\eta^4+26pz^4\eta^4+53pz^5\eta^4+42p^2z^5\eta^4\\[2ex]
\end{aligned}
\end{array}
\end{equation*}
\begin{equation*}
\begin{array}{ll}
\begin{aligned}
&\quad+137p^2z^6\eta^4+19p^3z^6\eta^4+154p^3z^7\eta^4+79p^4z^8\eta^4+15p^5z^9\eta^4.
\end{aligned}
\end{array}
\end{equation*}
\begin{figure}
\begin{center}
\begin{overpic}[scale=0.50]{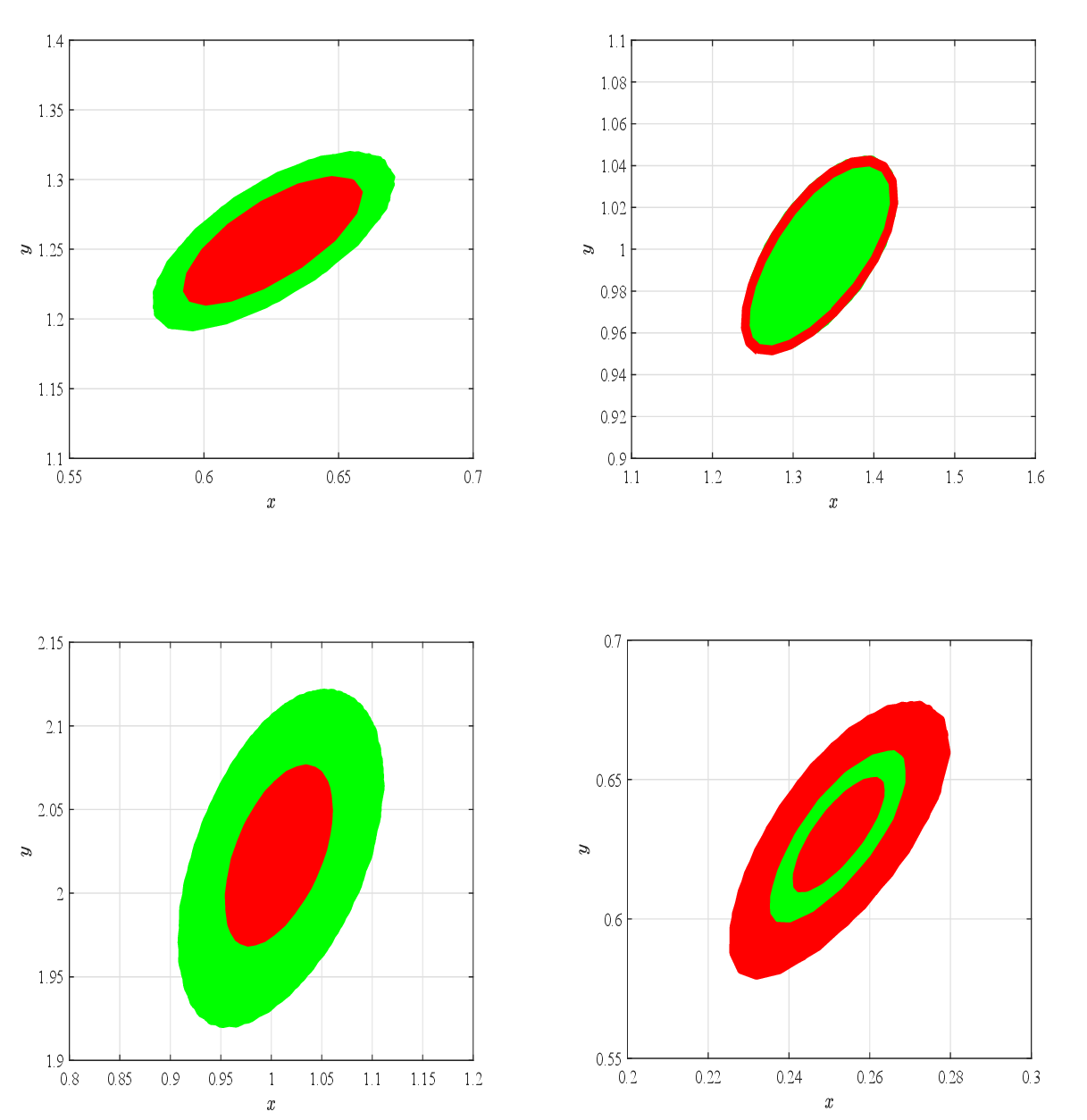}
\end{overpic}
\put(-228,145){$(a)$}
\put(-78,145){$(b)$}
\put(-228,-13){$(c)$}
\put(-78,-13){$(d)$}
\vspace{-1mm}
\end{center}
\caption{ For $k=2$, system \eqref{1.5} can undergo at most two limit cycles bifurcated by Hopf bifurcation. (a) $\Lambda_0=5.106$, $p=1$, $\gamma=5.24$ and $\eta=2.01$; (b) $\Lambda_0=5.417$, $p=1$, $\gamma=7.195$ and $\eta=0.75$; (c) $\Lambda_0=6.02$, $p=1$, $\gamma=6$ and $\eta=2.01$; (d) $\Lambda_0=1.753$, $p=8$, $\gamma=2.626$ and $\eta=2.5$.\label{w4}}
\end{figure}
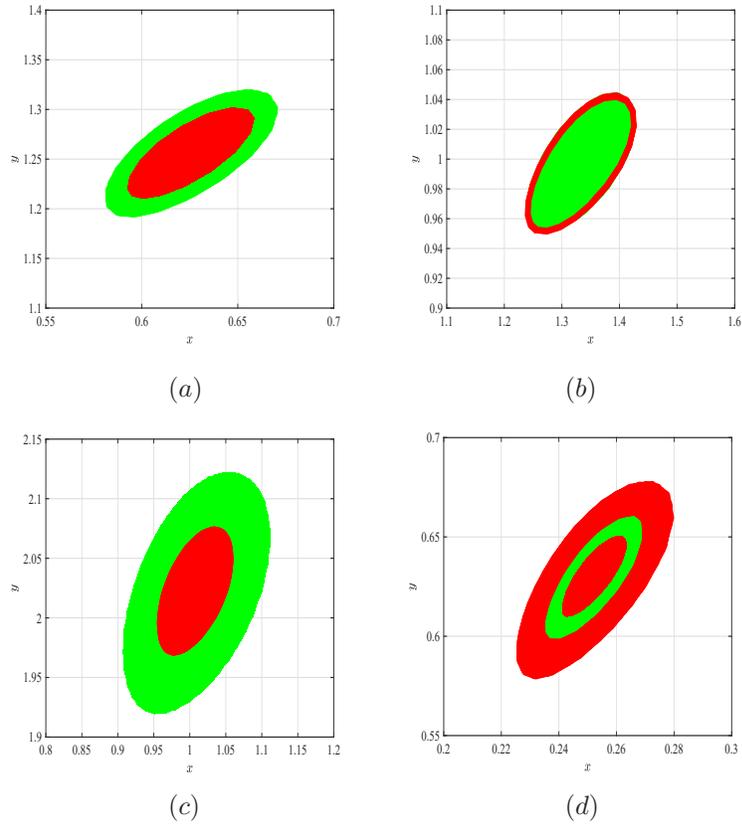
When $k=2$, we have the following results.
\begin{lemma}
For system \eqref{1.5}, the following statements hold.

(i) If $0<z<\frac{3}{4}$ and $2<\eta<\frac{2z-3}{4-z}+\frac{2\sqrt{9+4z}}{(4-z)\sqrt{z}}$ or $z>0$ and $\frac{1}{1+2z}<\eta<2$, then $E_{02}^*$ is a weak focus of order 1, i.e., system \eqref{1.5} admit one limit cycle bifurcated from Hopf bifurcation, see Fig.\ref{w4}(a) and Fig.\ref{w4}(b).

(ii) If $\frac{3}{4}<z<4$ and $2<\eta<\frac{2z-3}{4-z}+\frac{2\sqrt{9+4z}}{(4-z)\sqrt{z}}$, then $E_{02}^*$ is a stable weak focus of order 1, i.e., system \eqref{1.5} admit one stable limit cycle bifurcated from Hopf bifurcation, see Fig.\ref{w4}(c).

(iv) If $f_1=0$, we detected that $f_2>0$ for $0<z<\frac{3}{4}$ and $2<\eta<\frac{2z-3}{4-z}+\frac{2\sqrt{9+4z}}{(4-z)\sqrt{z}}$ or $z>0$ and $\frac{1}{1+2z}<\eta<2$, then $E_{02}^*$ is a weak focus of order 2, i.e., system \eqref{1.5} admit two limit cycles bifurcated from Hopf bifurcation, see Fig.\ref{w4}(d).
\end{lemma}
\begin{proof}
We treat $L_{11}$ as a quadratic functions of $p$, whose discriminant is
\begin{equation*}
\begin{array}{ll}
\begin{aligned}
\tilde{\Delta}^*=z^2(\eta^2z^2+4\eta z^2+4z^2-4\eta^2z-6\eta z+4z+9).
\end{aligned}
\end{array}
\end{equation*}
By analysis, we have $\tilde{\Delta}^*>0$ ($\tilde{\Delta}^*<0$) when $0<z<4$ and $0<\eta<\frac{2z-3}{4-z}+\frac{2\sqrt{9+4z}}{(4-z)\sqrt{z}}$ or $z>4$ ($0<z<4$ and $\eta>\frac{2z-3}{4-z}+\frac{2\sqrt{9+4z}}{(4-z)\sqrt{z}}$), i.e., $L_{11}=0$ has at most two real roots $\tilde{p}_{11}$ and $\tilde{p}_{22}$, where
\begin{equation*}
\begin{array}{ll}
\begin{aligned}
\tilde{p}_{11}=\frac{3z+2z^2-3\eta z^2-\sqrt{\tilde{\Delta}^*}}{2z^3(\eta-2)},\quad \tilde{p}_{22}=\frac{3z+2z^2-3\eta z^2+\sqrt{\tilde{\Delta}^*}}{2z^3(\eta-2)}.
\end{aligned}
\end{array}
\end{equation*}

If $0<z<\frac{3}{4}$ and $2<\eta<\frac{2z-3}{4-z}+\frac{2\sqrt{9+4z}}{(4-z)\sqrt{z}}$, then $L_{11}=0$ has two real roots $\tilde{p}_{11}$ and $\tilde{p}_{22}$ ($0<\tilde{p}_{11}<\tilde{p}_{22}$). And $\mathrm{Lcoeff}(L_{11},p)=z^3(2-\eta)<0$. Thus, if $0<p<\tilde{p}_{11}$ and $p>\tilde{p}_{22}$ ($\tilde{p}_{11}<p<\tilde{p}_{22}$), we have $f_1<0$ ($f_1>0$), i.e., $E_{02}^*$ is a stable (an unstable) weak focus of order 1.

If $z>0$ and $\frac{1}{1+2z}<\eta<2$, we have $\tilde{p}_{11}<0<\tilde{p}_{22}$. Thus, if $0<p<\tilde{p}_{22}$ ($p>\tilde{p}_{22}$), then $f_1<0$ ($f_1>0$), i.e., $E_{02}^*$ is a stable (an unstable) weak focus of order 1.

If $\frac{3}{4}<z<4$ and $2<\eta<\frac{2z-3}{4-z}+\frac{2\sqrt{9+4z}}{(4-z)\sqrt{z}}$, we have $\tilde{p}_{11}<\tilde{p}_{22}<0$. Thus, $f_1<0$, i.e., $E_{02}^*$ is a stable weak focus of order 1.

If $f_1=0$, we detected that $f_2>0$ for $0<z<\frac{3}{4}$ and $2<\eta<\frac{2z-3}{4-z}+\frac{2\sqrt{9+4z}}{(4-z)\sqrt{z}}$ or $z>0$ and $\frac{1}{1+2z}<\eta<2$.

\end{proof}

\subsection{$k>2$}

In this section, we consider Hopf bifurcation of system \eqref{1.5} around equilibrium $\bar{E}_5^*$ (or $\bar{E}_3^*$). In the following, Next, we mainly analyze the dynamic phenomenon near $\bar{E}_5^*$ and $\bar{E}_3^*$ is similar to $\bar{E}_5^*$, which is omitted here.

\subsubsection{$R_i=0$ ($i=0, 2, 3$)}
\begin{figure}[ht!]
\begin{center}
\begin{overpic}[scale=0.60]{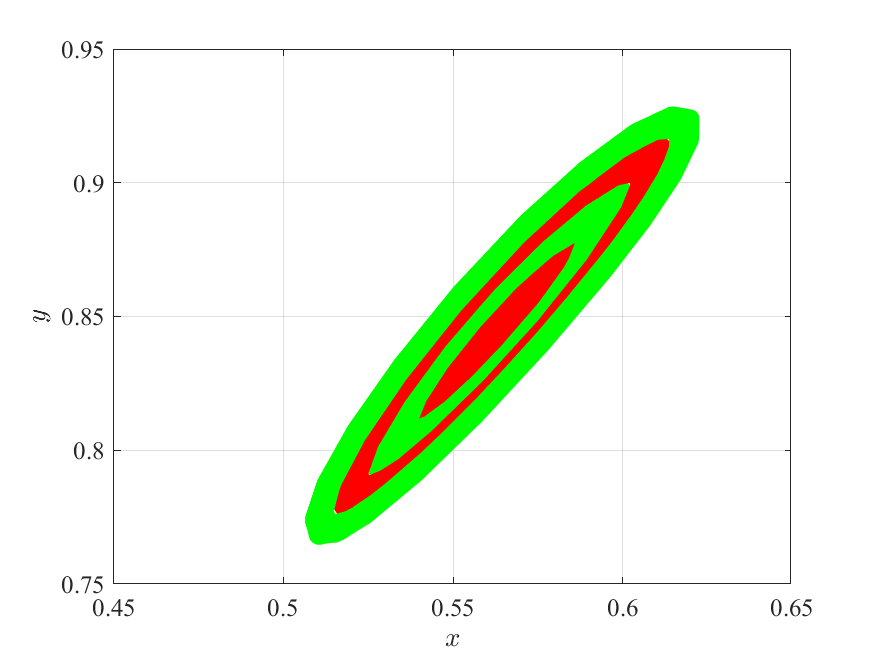}
\end{overpic}
\vspace{-1mm}
\end{center}
\caption{ For $k>2$, system \eqref{1.5} can undergo at most three limit cycles bifurcated by Hopf bifurcation when $R_i=0$, where $k=3$, $\Lambda_0=4.03906$, $p=1.05$, $\gamma=3.5075$ and $\eta=1.5$.\label{w5}}
\end{figure}
For $R_i=0$ ($i=0, 2, 3$), we have the following results.
\begin{theorem}\label{k4}
Let $(z,p,\eta)\in\Omega^*$ and \eqref{4.1} hold, if $R_i=0$ ($i=0, 2, 3$), then $\bar{E}_5^*$ is a weak focus of order at most 3, i.e., system \eqref{1.5} can undergo at most three limit cycles bifurcated by Hopf bifurcation, see Fig.\ref{w5}.
\end{theorem}
\begin{proof}
The proof is similar to Theorem \ref{k1}, which is omitted here.
\end{proof}

\subsubsection{$R_i\neq0$ ($i=0, 2, 3$)}

For $R_i\neq0$ ($i=0, 2, 3$) and $k>2$, we have $l_1>0$ and $l_3>0$. If $l_2=0$ (or $l_4=0$) when $k>2$, then the following statements hold.
\begin{theorem}\label{k5}
Let $(z,p,\eta)\in\Omega^*$ and \eqref{4.1} hold. For $l_2=0$ (or $l_4=0$), $\bar{E}_2^*$ can be a weak focus of order at most 3, i.e., system \eqref{1.5} can undergo at most three limit cycles bifurcated by Hopf bifurcation, see Fig.\ref{w6}.
\end{theorem}
\begin{proof}
The proof is similar to Theorem \ref{k2}, which is omitted here for brevity.
\end{proof}
\begin{figure}
\begin{center}
\begin{overpic}[scale=0.60]{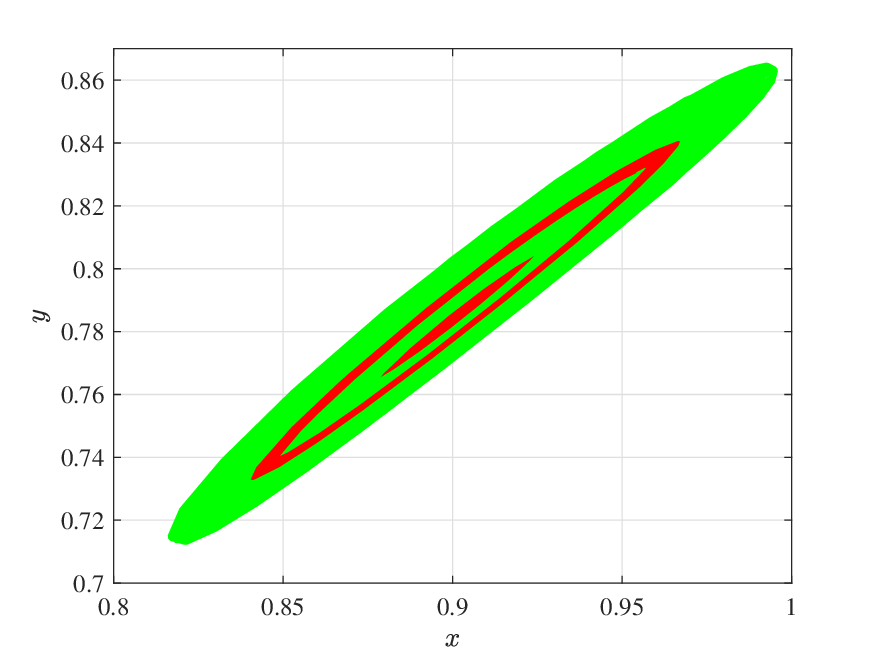}
\end{overpic}
\vspace{-1mm}
\end{center}
\caption{ For $k>2$, system \eqref{1.5} can undergo at most three limit cycles bifurcated by Hopf bifurcation when $R_i\neq0$ and $l_2=0$, where $k=3$, $\Lambda_0=5.06313$, $p=0.4$, $\gamma=4.4746$ and $\eta=0.870547$.\label{w6}}
\end{figure}

\begin{remark}
For $R_i\neq0$ ($i=0, 2, 3$) and $l_2\neq0$ (or $l_4\neq0$), $\bar{E}_5^*$ can be a weak focus of order 4. Here, we present only the numerical simulation results, as shown in Fig.\ref{Hopf}.
\end{remark}

\begin{figure}
\begin{center}
\begin{overpic}[scale=0.60]{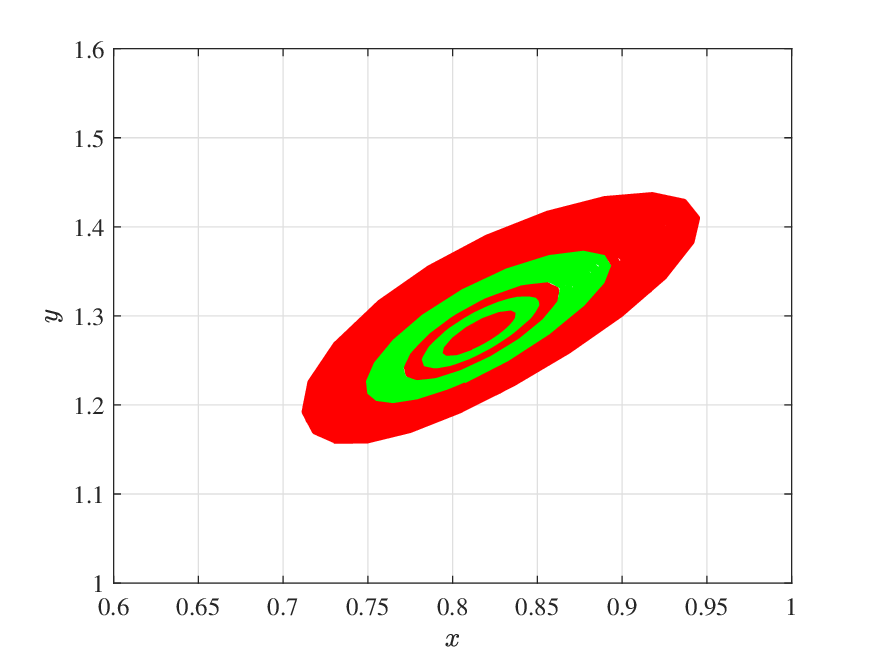}
\end{overpic}
\vspace{-5mm}
\end{center}
\caption{ System \eqref{1.5} exists four limit cycles created by Hopf bifurcation around $\bar{E}_5^*$ when $R_i\neq0$ ($i=0, 1, 2$) and $l_i\neq0$ ($i=2, 4$).\label{Hopf}}
\end{figure}

\section{Discussion}

In this paper, we study an SIRS epidemic model with a nonlinear incidence rate $f(I)S=\beta I(1+\upsilon I^{k-1})S$ with $\beta>0$, $\upsilon>0$ and $k>0$, which was introduced in epidemic models in \cite{PJ}. Although previous studies for system \eqref{1.1} with \eqref{0.7} primarily focused on a specific value of $k$, we investigated the dynamics of system \eqref{1.1} with \eqref{0.7} for general $k$ and revealed richer dynamic behaviors.

By simplification, system \eqref{1.2} is equivalent to system \eqref{1.5}. We found that the disease free equilibrium $E_{0}(0,0)$ always exists and is asymptotically stable when $R_0<1$ and unstable when $R_0>1$ and a saddle-node when $R_0=1$. When $R_0>1$, system \eqref{1.5} has at most three positive equilibria. When positive equilibria exist, it may be a saddle-node, a cusp of codimension 3 or 4, a degenerate node, a nilpotent focus of codimension 3, a nilpotent elliptic equilibrium of codimension 3, a nilpotent focus of codimension 4 or a weak focus of high order as the parameters change.

If the positive equilibrium is degenerate, then system \eqref{1.5} can undergo a Bogdanov-Takens bifurcation of codimension 4 when $(\Lambda_0,\gamma,\eta,p)$ varies near $(\tilde{\Lambda}_0,\tilde{\gamma},\tilde{\eta},\check{p})$. If $(p,\Lambda_0,\gamma,\eta)$ vary in the small neighborhood of $(\check{p},\bar{\Lambda}_0,\bar{\gamma},\bar{\eta})$, system \eqref{1.5} can undergo a nilpotent focus bifurcation of codimension 4 around the degenerate equilibrium. When positive equilibrium satisfies $Tr(J(E))=0$ and $Det(J(E))>0$, which indicating that it may be a weak focus. As parameters vary, the stability of positive equilibrium has changed, and we show that it can be a weak focus of high order. Consequently, system \eqref{1.5} admits four limit cycles bifurcated from Hopf bifurcation. But for some special values of $k$, other dynamic phenomena may occur, which we will discuss in the future work.Finally, we notice that all above dynamics of the model are sensitive to the parameter $\Lambda_0$ and $\gamma$, i.e., the change of $R_0$ will affect the dynamic behavior of the system \eqref{1.5}. However, the expression of $\Lambda_0$ and $\gamma$ shows that $\beta$ plays the important role in the system \eqref{1.2}. Therefore, it is in deed the nonlinear incidence rate that produces the complicated dynamics of epidemic models and makes the models more reasonable and practical.


\begin{thebibliography}{99}

\bibitem{YAC}
Kuznetsov Y A, Piccardi C. Bifurcation analysis of periodic SEIR and SIR epidemic models, J. Math. Biol. 1994, 32(2): 109--121.

\bibitem{LH}
Smith H L. Subharmonic bifurcation in an SIR epidemic model, J. Math. Biol. 1983, 17(2), 163--177.

\bibitem{EM}
Alexander M E, Moghadas S M. Bifurcation analysis of an SIRS epidemic model with generalized incidence, SIAM J. Appl. Math. 2005, 65(5): 1794--1816.

\bibitem{WW}
Liu W, Hethcote H W, Levin S A. Dynamical behavior of epidemiological models with nonlinear incidence rates, J. Math. Biol. 1987, 25(4): 359--380.

\bibitem{MA}
Liu W M, Levin S A, Iwasa Y. Influence of nonlinear incidence rates upon the behavior of SIRS epidemiological models, J. Math. Biol. 1986, 23(2): 187--204.

\bibitem{NM}
Chitnis N, Cushing J M, Hyman J M. Bifurcation analysis of a mathematical model for malaria transmission, SIAM J. Appl. Math. 2006, 67(1): 24--45.

\bibitem{MVK}
Aguiar M, Steindorf V, Srivastav A K, et al. Bifurcation analysis of a two infection SIR-SIR epidemic model with temporary immunity and disease enhancement, Nonlinear Dyn. 2024, 112(15): 13621--13639.

\bibitem{CC}
Arino J, McCluskey C C, van den Driessche P. Global results for an epidemic model with vaccination that exhibits backward bifurcation. SIAM J. Appl. Math. 2003, 64(1), 260--276.

\bibitem{WS}
Ruan S, Wang W. Dynamical behavior of an epidemic model with a nonlinear incidence rate, J. Differ. Equ. 2003, 188(1): 135--163.

\bibitem{PJ}
van den Driessche P, Watmough J. A simple SIS epidemic model with a backward bifurcation, J. Math. Biol. 2000, 40(6): 525--540.

\bibitem{PR}
Derrick W R, Van den Driessche P. A disease transmission model in a nonconstant population, J. Math. Biol. 1993, 31(5): 495--512.

\bibitem{ME}
Alexander M E, Moghadas S M. Periodicity in an epidemic model with a generalized non-linear incidence, Math. Biosci. 2004, 189(1): 75--96.

\bibitem{WP}
Yu P, Zhang W. Complex dynamics in a unified SIR and HIV disease model: a bifurcation theory approach, J. Nonlinear Sci. 2019, 29: 2447--2500.

\bibitem{GV}
Capasso V, Serio G. A generalization of the Kermack-McKendrick deterministic epidemic model, Math. Biosci. 1978, 42(1-2): 43--61.

\bibitem{SD}
Xiao D, Ruan S. Global analysis of an epidemic model with nonmonotone incidence rate, Math. Biosci. 2007, 208(2), 419--429.

\bibitem{YDS}
Tang Y, Huang D, Ruan S, et al. Coexistence of limit cycles and homoclinic loops in a SIRS model with a nonlinear incidence rate, SIAM J. Appl. Math. 2008, 69(2): 621--639.

\bibitem{YD}
Zhou Y, Xiao D, Li Y. Bifurcations of an epidemic model with non-monotonic incidence rate of saturated mass action, Chaos, Solitons Fractals, 2007, 32(5): 1903--1915.

\bibitem{ST}
Levin S A, Hallam T G, Gross L J. Applied mathematical ecology, Springer-Verlag, Nes York, 1989.

\bibitem{DY}
Xiao D, Zhou Y. Qualitative analysis of an epidemic model, Can. Appl. Math. Q. 2006, 14(4): 480--484.

\bibitem{WF}
Zhang F, Cui W, Dai Y, et al. Bifurcations of an SIRS epidemic model with a general saturated incidence rate, Math. Biosci. Eng. 2022, 19(11): 10710--10730.

\bibitem{WYY}
Cui W, Zhao Y. Saddle-node bifurcation and Bogdanov-Takens bifurcation of a SIRS epidemic model with nonlinear incidence rate. J. Differ. Equ. 2024, 384: 252--278.

\bibitem{PZ}
Hu Z, Bi P, Ma W, et al. Bifurcations of an SIRS epidemic model with nonlinear incidence rate, Discrete Contin. Dyn. Syst. Ser. B, 2011, 15(3): 93--112.

\bibitem{MM}
Alexander M E, Moghadas S M. Bifurcation analysis of an SIRS epidemic model with generalized incidence. SIAM J. Appl. Math. 2005, 65(5): 1794--1816.

\bibitem{HP}
Hethcote H W, van den Driessche P. Some epidemiological models with nonlinear incidence. J. Math. Biol. 1991, 29(3): 271--287.

\bibitem{YS}
Tang Y, Huang D, Ruan S, Zhang W. Coexistence of limit cycles and homoclinic loops in a SIRS model with a nonlinear incidence rate. SIAM J. Appl. Math. 2008, 69(2): 621--639.

\bibitem{LC}
Wang L, Gao C, Rifhat R, Wang K, Teng Z. Stationary distribution and bifurcation analysis for a stochastic SIS model with nonlinear incidence and degenerate diffusion. Chaos Solitons Fractals, 2024, 182: 114872.

\bibitem{WA}
Liu W M, Levin S A, Iwasa Y. Influence of nonlinear incidence rates upon the behavior of SIRS epidemiological models. J. Math. Biol. 1986, 23: 187--204.

\bibitem{GO}
Kermack W O, McKendrick A G. A contribution to the mathematical theory of epidemics, Proceedings of the royal society of London. Series A, Containing papers of a mathematical and physical character, 1927, 115(772): 700--721.

\bibitem{VP}
van den Driessche P, Watmough J. Epidemic solutions and endemic catastrophes, Fields Inst. Commun. 2003, 36: 247--257.

\bibitem{MD}
Lu M, Gao D, Huang J, Wang H. Relative prevalence-based dispersal in an epidemic patch model. J. Math. Biolo. 2023, 86(4): 52.

\bibitem{YW}
Jin Y, Wang W, Xiao S. An SIRS model with a nonlinear incidence rate, Chaos, Solitons Fractals, 2007, 34(5): 1482--1497.

\bibitem{ZTW}
Zhang Z, Ding T, Huang W, Dong Z. Qualitative Theory of Differential Equations, Transl. Math. Monogr. vol 101 Amer. Math. Soc. Providence RI, 1992.

\bibitem{CB}
Khibnik A I,  Krauskopf B, Rousseau C. Global study of a family of cubic Li$\acute{e}$nard equations. Nonlinearity, 1998, 11(6): 1505.

\bibitem{CP}
Dumortier F, Fiddelaers P, Li C. Generic unfolding of the nilpotent saddle of codimension four, Global analysis of dynamical systems. CRC Press, 2001: 137--172.

\bibitem{DG}
Cai L, Chen G, Xiao D. Multiparametric bifurcations of an epidemiological model with strong Allee effect, J. Math. Biol. 2013, 67(2): 185--215.

\bibitem{JG}
G. Dangelmayr, J. Guckenheimer, On a four parameter family of planar vector fields, Archive for Rational Mechanics and Analysis, 1987, 97: 321--352.

\bibitem{AI}
A. I. Khibnik, B. Krauskopf, C. Rousseau, Global study of a family of cubic Li$\acute{e}$nard equations. Nonlinearity, 1998, 11(6): 1505.
\end{thebibliography}
\end{document}